\documentclass{article}
\usepackage[utf8]{inputenc}
\usepackage[sort&compress,numbers]{natbib}
\usepackage{float}
\usepackage{comment}
\usepackage{amsmath,amsthm,amssymb,mathtools,booktabs}
\usepackage{hyperref}
\usepackage{bm}
\usepackage{subcaption}
\usepackage[colorinlistoftodos]{todonotes}
\usepackage[a4paper, total={6in, 9in}]{geometry}
\usepackage{cleveref}
\theoremstyle{plain}
\newtheorem{theorem}{Theorem}[section]
\newtheorem{proposition}[theorem]{Proposition}
\newtheorem{corollary}[theorem]{Corollary}
\newtheorem{lemma}[theorem]{Lemma}
\newtheorem{definition}[theorem]{Definition}
\newtheorem{example}[theorem]{Example}
\newtheorem{remark}[theorem]{Remark}
\newtheorem{conjecture}[theorem]{Conjecture}

\newcount\rowc

\makeatletter
\def\ttabular{
\hbox\bgroup
\let\\\cr
\def\rulea{\ifnum\rowc=\@ne \hrule height 1.5pt \fi}
\def\ruleb{
\ifnum\rowc=1\hrule height2pt \else
\ifnum\rowc=6\hrule height \heavyrulewidth 
   \else \hrule height \lightrulewidth\fi\fi}
\valign\bgroup
\global\rowc\@ne
\rulea
\hbox to 4em{\strut \vline \hfill##\hfill \vline }%
\ruleb
&&%
\global\advance\rowc\@ne
\hbox to 4em{\vline \strut\hfill##\hfill \vline}%
\ruleb
\cr}
\def\endttabular{%
\crcr\egroup\egroup}
\DeclareMathOperator{\Aut}{Aut}
\DeclareMathOperator{\PSL}{PSL}
\DeclareMathOperator{\Pow}{Pow}
\DeclareMathOperator{\stab}{stab}
\DeclareMathOperator{\vf}{vf}

\DeclareMathOperator{\Sym}{Sym}

\DeclareMathOperator{\GL}{GL}
\DeclareMathOperator{\F}{\mathcal{F}}
\DeclareMathOperator{\G}{\mathcal{G}}
\DeclareMathOperator{\C}{\mathcal{C}}

\DeclareMathOperator{\E}{\mathcal{E}}

\DeclareMathOperator{\tetrahedron}{\mathcal{T}}

\usepackage{amsmath, amssymb, amsfonts} % Standard Math-stuff

\usepackage{ifthen}

\usepackage{tikz}
\usetikzlibrary{calc}
\usetikzlibrary{positioning}
\usetikzlibrary{shapes}
\usetikzlibrary{patterns}

\makeatletter
\edef\texforht{TT\noexpand\fi
  \@ifpackageloaded{tex4ht}
    {\noexpand\iftrue}
    {\noexpand\iffalse}}
\makeatother

\makeatletter
\newif\iftikz@node@phantom
\tikzset{
  phantom/.is if=tikz@node@phantom,
  text/.code=%
    \edef\tikz@temp{#1}%
    \ifx\tikz@temp\tikz@nonetext
      \tikz@node@phantomtrue
    \else
      \tikz@node@phantomfalse
      \let\tikz@textcolor\tikz@temp
    \fi
}
\usepackage{etoolbox}
\patchcmd\tikz@fig@continue{\tikz@node@transformations}{%
  \iftikz@node@phantom
    \setbox\pgfnodeparttextbox\hbox{}
  \fi\tikz@node@transformations}{}{}
\makeatother

\newcommand{\tikzAngleOfLine}{\tikz@AngleOfLine}
\def\tikz@AngleOfLine(#1)(#2)#3{%
  \pgfmathanglebetweenpoints{%
    \pgfpointanchor{#1}{center}}{%
    \pgfpointanchor{#2}{center}}
  \pgfmathsetmacro{#3}{\pgfmathresult}%
}

\tikzset{ 
    vertexNodePlain/.style = {fill=#1, shape=circle, inner sep=0pt, minimum size=2pt, text=none},
    vertexNodePlain/.default=gray,
    vertexPlain/labels/.style = {
        vertexNode/.style={vertexNodePlain=##1},
        vertexLabel/.style={gray}
    },
    vertexPlain/nolabels/.style = {
        vertexNode/.style={vertexNodePlain=##1},
        vertexLabel/.style={text=none}
    },
    vertexPlain/.style = vertexPlain/#1,
    vertexPlain/.default=labels
}
\tikzset{
    vertexNodeNormal/.style = {fill=#1, shape=circle, inner sep=0pt, minimum size=4pt, text=none},
    vertexNodeNormal/.default = blue,
    vertexNormal/labels/.style = {
        vertexNode/.style={vertexNodeNormal=##1},
        vertexLabel/.style={blue}
    },
    vertexNormal/nolabels/.style = {
        vertexNode/.style={vertexNodeNormal=##1},
        vertexLabel/.style={text=none}
    },
    vertexNormal/.style = vertexNormal/#1,
    vertexNormal/.default=labels
}
\tikzset{
    vertexNodeBallShading/pdf/.style = {ball color=#1},
    vertexNodeBallShading/svg/.style = {fill=#1},
    vertexNodeBallShading/.code = {% Conditional shading depending whether we want pdf or svg output
        \if\texforht
            \tikzset{vertexNodeBallShading/svg=#1!90!black}
        \else
            \tikzset{vertexNodeBallShading/pdf=#1}
        \fi
    },
    vertexNodeBall/.style = {shape=circle, vertexNodeBallShading=#1, inner sep=2pt, outer sep=0pt, minimum size=3pt, font=\tiny},
    vertexNodeBall/.default = white,
    vertexBall/labels/.style = {
        vertexNode/.style={vertexNodeBall=##1, text=black},
        vertexLabel/.style={text=none}
    },
    vertexBall/nolabels/.style = {
        vertexNode/.style={vertexNodeBall=##1, text=none},
        vertexLabel/.style={text=none}
    },
    vertexBall/.style = vertexBall/#1,
    vertexBall/.default=labels
}
\tikzset{ 
    vertexStyle/.style={vertexNormal=#1},
    vertexStyle/.default = labels
}

\newcommand{\vertexLabelR}[4][]{
    \ifthenelse{ \equal{#1}{} }
        { \node[vertexNode] at (#2) {#4}; }
        { \node[vertexNode=#1] at (#2) {#4}; }
    \node[vertexLabel, #3] at (#2) {#4};
}
\newcommand{\vertexLabelA}[4][]{
    \ifthenelse{ \equal{#1}{} }
        { \node[vertexNode] at (#2) {#4}; }
        { \node[vertexNode=#1] at (#2) {#4}; }
    \node[vertexLabel] at (#3) {#4};
}

\newcommand{\edgeLabelColor}{blue!20!white}
\tikzset{
    edgeLineNone/.style = {draw=none},
    edgeLineNone/.default=black,
    edgeNone/labels/.style = {
        edge/.style = {edgeLineNone=##1},
        edgeLabel/.style = {fill=\edgeLabelColor,font=\small}
    },
    edgeNone/nolabels/.style = {
        edge/.style = {edgeLineNone=##1},
        edgeLabel/.style = {text=none}
    },
    edgeNone/.style = edgeNone/#1,
    edgeNone/.default = labels
}
\tikzset{
    edgeLinePlain/.style={line join=round, draw=#1},
    edgeLinePlain/.default=black,
    edgePlain/labels/.style = {
        edge/.style={edgeLinePlain=##1},
        edgeLabel/.style={fill=\edgeLabelColor,font=\small}
    },
    edgePlain/nolabels/.style = {
        edge/.style={edgeLinePlain=##1},
        edgeLabel/.style={text=none}
    },
    edgePlain/.style = edgePlain/#1,
    edgePlain/.default = labels
}
\tikzset{
    edgeLineDouble/.style = {very thin, double=#1, double distance=.8pt, line join=round},
    edgeLineDouble/.default=gray!90!white,
    edgeDouble/labels/.style = {
        edge/.style = {edgeLineDouble=##1},
        edgeLabel/.style = {fill=\edgeLabelColor,font=\small}
    },
    edgeDouble/nolabels/.style = {
        edge/.style = {edgeLineDouble=##1},
        edgeLabel/.style = {text=none}
    },
    edgeDouble/.style = edgeDouble/#1,
    edgeDouble/.default = labels
}
\tikzset{
    edgeStyle/.style = {edgePlain=#1},
    edgeStyle/.default = labels
}

\newcommand{\faceColorY}{yellow!60!white}   % yellow
\newcommand{\faceColorB}{blue!60!white}     % blue
\newcommand{\faceColorC}{cyan!60}           % cyan
\newcommand{\faceColorR}{red!60!white}      % red
\newcommand{\faceColorG}{green!60!white}    % green
\newcommand{\faceColorO}{orange!50!yellow!70!white} % orange

\newcommand{\faceColor}{\faceColorY}
\newcommand{\faceColorSwap}{\faceColorC}

\tikzset{
    face/.style = {fill=#1},
    face/.default = \faceColor,
    faceY/.style = {face=\faceColorY},
    faceB/.style = {face=\faceColorB},
    faceC/.style = {face=\faceColorC},
    faceR/.style = {face=\faceColorR},
    faceG/.style = {face=\faceColorG},
    faceO/.style = {face=\faceColorO}
}
\tikzset{
    faceStyle/labels/.style = {
        faceLabel/.style = {}
    },
    faceStyle/nolabels/.style = {
        faceLabel/.style = {text=none}
    },
    faceStyle/.style = faceStyle/#1,
    faceStyle/.default = labels
}
\tikzset{ face/.style={fill=#1} }
\tikzset{ faceSwap/.code=
    \ifdefined\swapColors
        \tikzset{face=\faceColorSwap}
    \else
        \tikzset{face=\faceColor}
    \fi
}

\usepackage{hyperref}

\title{A census of face-transitive surfaces}
\author{Reymond Akpanya and Jonathan Spreer}
\date{}

\begin{document}
\maketitle
\pagestyle{plain} 
\begin{abstract}
A face-transitive surface is a triangulated $2$-dimensional manifold whose automorphism group acts transitively on its set of triangles. In this paper, we investigate this class of highly symmetric surface triangulations. We identify  seven types of such face-transitive surfaces, splitting up further into a total of thirteen sub-types, distinguished by how their automorphism groups act on them. We use these theoretical results to compute a census of face-transitive surfaces with up to $1280$ faces by constructing suitable cycle double covers of cubic node-transitive graphs.
\end{abstract}

\section{Introduction} \label{section:Intro}

Combinatorial objects arise across many mathematical disciplines.
One example of such an object is the notion of a {\bf simplicial surface}, defined as the incidence relations between the vertices, edges and faces of a given triangulated $2$-dimensional manifold. Here, we only consider simplicial surfaces whose underlying triangulation is a simplicial complex. For further research related to the theory of simplicial surfaces, we refer to 
\cite{automorphism,Altmann24,onetriangle,icosahedron,burton2018pachner, isosceles,NEGAMI1994225,simplicialsurfacesbook}, to name only a few. 

This paper examines the class of {\bf face-transitive surfaces}, that is, simplicial surfaces (with underlying triangulations being simplicial complexes, and) with automorphism groups acting transitively on their faces. We prove the following statement.

\begin{theorem}
    \label{thm:main}
    There are exactly $86\, 802$ isomorphism classes of face-transitive surfaces with up to $1280$ faces.
\end{theorem}

\paragraph{Related work and background.}

Many examples of face-transitive surfaces, such as chiral and regular maps, have been described in the literature \cite{kuehnel,CONDER2001224,DATTA20183296,DATTA2022112652}. Arguably, one of the most famous such examples is the {\bf simplicial tetrahedron} $\tetrahedron$. This simplicial surface, comprising of four vertices, six edges, and four faces, emerges from the incidence relations between the vertices, edges, and faces of the regular tetrahedron, also known as the Platonic solid with four triangular faces. By labelling the vertices and faces of the simplicial tetrahedron with natural numbers, we obtain the folding plan illustrated in \Cref{fig:tetrahedron} on the left.
\begin{figure}[H]
  \centering
  \begin{subfigure}{.45\textwidth}
   \centering
   \scalebox{1.}{\begin{tikzpicture}[vertexBall, edgeDouble=nolabels, faceStyle, scale=1.5]

% Define the coordinates of the vertices
\coordinate (V1_1) at (0., 0.);
\coordinate (V2_1) at (1., 0.);
\coordinate (V3_1) at (0.4999999999999999, 0.8660254037844386);
\coordinate (V4_1) at (-0.9999999999999996, -1.732050807568877);
\coordinate (V4_1) at (0.5000000000000001, -0.8660254037844386);
\coordinate (V5_1) at (1.5, 0.8660254037844388);
\coordinate (V5_2) at (1.5, -0.8660254037844385);
\coordinate (V5_3) at (0., -1.732050807568877);
\coordinate (V6_1) at (-0.4999999999999999, 0.8660254037844385);
\coordinate (V6_2) at (-0.4999999999999998, -0.8660254037844388);

% Fill in the faces
\fill[face]  (V2_1) -- (V3_1) -- (V1_1) -- cycle;
\node[faceLabel] at (barycentric cs:V2_1=1,V3_1=1,V1_1=1) {$2$};
\fill[face]  (V1_1) -- (V4_1) -- (V2_1) -- cycle;
\node[faceLabel] at (barycentric cs:V1_1=1,V4_1=1,V2_1=1) {$1$};
%\fill[face]  (V2_1) -- (V5_1) -- (V3_1) -- cycle;
%\node[faceLabel] at (barycentric cs:V2_1=1,V5_1=1,V3_1=1) {$5$};
%\fill[face]  (V6_2) -- (V3_2) -- (V5_3) -- cycle;
%\node[faceLabel] at (barycentric cs:V6_2=1,V3_2=1,V5_3=1) {$8$};
\fill[face]  (V1_1) -- (V6_2) -- (V4_1) -- cycle;
\node[faceLabel] at (barycentric cs:V1_1=1,V6_2=1,V4_1=1) {$3$};
\fill[face]  (V4_1) -- (V5_2) -- (V2_1) -- cycle;
\node[faceLabel] at (barycentric cs:V4_1=1,V5_2=1,V2_1=1) {$4$};
%\fill[face]  (V3_1) -- (V6_1) -- (V1_1) -- cycle;
%\node[faceLabel] at (barycentric cs:V3_1=1,V6_1=1,V1_1=1) {$4$};

% Draw the edges
\draw[edge] (V2_1) -- node[edgeLabel] {$1$} (V1_1);
\draw[edge] (V3_1) -- node[edgeLabel] {$2$} (V2_1);
\draw[edge] (V1_1) -- node[edgeLabel=blue] {$3$} (V3_1);
\draw[edge] (V4_1) -- node[edgeLabel] {$4$} (V1_1);
\draw[edge] (V2_1) -- node[edgeLabel] {$5$} (V4_1);
\draw[edge] (V4_1) -- node[edgeLabel] {$7$} (V6_2);
\draw[edge] (V5_2) -- node[edgeLabel] {$8$} (V4_1);
\draw[edge] (V2_1) -- node[edgeLabel] {$10$} (V5_2);
%\draw[edge] (V6_1) -- node[edgeLabel] {$11$} (V3_1);
%\draw[edge] (V3_2) -- node[edgeLabel] {$11$} (V6_2);
%\draw[edge] (V1_1) -- node[edgeLabel] {$12$} (V6_1);
\draw[edge] (V6_2) -- node[edgeLabel] {$12$} (V1_1);

% Draw the vertices
\vertexLabelR{V1_1}{left}{$1$}
\vertexLabelR{V2_1}{left}{$2$}
\vertexLabelR{V3_1}{left}{$4$}
%\vertexLabelR{V3_2}{left}{$5$}
\vertexLabelR{V4_1}{left}{$3$}
%\vertexLabelR{V5_1}{left}{$6$}
\vertexLabelR{V5_2}{left}{$4$}
%\vertexLabelR{V5_3}{left}{$6$}
%\vertexLabelR{V6_1}{left}{$44$}
\vertexLabelR{V6_2}{left}{$4$}

\end{tikzpicture}}
  \caption{}
  \label{fig:tetrahedron}  
  \end{subfigure}
  \begin{subfigure}{.45\textwidth}
   \centering
   \scalebox{.8}{\begin{tikzpicture}[vertexBall, edgeDouble=nolabels, faceStyle=nolabels, scale=3]

\coordinate (V1) at (0,0);
\coordinate (V2) at (1 , 0);
\coordinate (V3) at (0.5 , 0.8660);

\coordinate (V4) at (0.5,0.38*0.8660);

\draw[edge] (V1) --  (V2);
\draw[edge] (V3) --  (V2);
\draw[edge] (V1) --  (V3);
\draw[edge] (V4) --  (V1);
\draw[edge] (V2) --  (V4);
\draw[edge] (V3) --  (V4);
% Draw the faces
\vertexLabelR[]{V1}{left}{$ $}
\vertexLabelR[]{V2}{left}{$ $}
\vertexLabelR[]{V3}{left}{$ $}
\vertexLabelR[]{V4}{left}{$ $}

\node at (-0.1,.0) {$3$};
\node at (1.1 , 0) {$4$};
\node at (0.5,1.) {$2$};
\node at (0.5,0.2) {$1$};

\end{tikzpicture}}
  \caption{}
  \label{fig:facegraphtetrahedron}  
  \end{subfigure}
  \caption{(a) The simplicial tetrahedron $\tetrahedron$ and (b) its face graph $\F(X)$}
\end{figure}
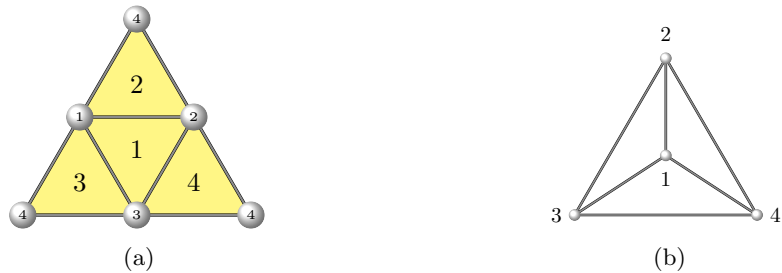
The incidence relations between the faces and edges of a simplicial surface yield a cubic graph, known as its {\bf face graph}, or {\bf dual} or {\bf face pairing graph}. This graph can be used to study the simplicial surface itself. The face graph of the simplicial tetrahedron is shown in \Cref{fig:facegraphtetrahedron} on the right. In this particular example, the face graph is isomorphic to the complete graph on four vertices and hence to the graph describing the incidences between vertices and edges of the simplicial tetrahedron. In general, this is not the case.

The construction of a cubic graph from a simplicial surface is a simple procedure. In contrast, the task of constructing a simplicial surface from a cubic graph is more complicated and may result in many distinct simplicial surfaces (or none at all). 

Instead, a simplicial surface is determined by a unique {\bf cycle double cover} on a cubic graph. A cycle double cover of a cubic graph is a set of cycles such that every edge of the graph is contained in exactly two cycles, see \cite{szekeres}. The cycles in the cubic graph contain the information of the vertices in the arising simplicial surface. As an example, the face graph of the simplicial tetrahedron has, up to isomorphism, exactly two cycle double covers. Only one results in a face transitive simplicial surface, see Section \ref{subsection:facegraphs}.

\paragraph{Summary of contributions.}

Our classification of small face-transitive surfaces rests on a distinguishing invariant, which we call the {\bf vertex-face type} of a face-transitive surface (see \Cref{definitionvertexfacetype}). There are exactly seven such vertex-face types of face-transitive surfaces, as is detailed in  \Cref{theorem:invariants}. For some of them, multiple different sub-types of face-transitive surfaces arise, resulting in a total of thirteen sub-types. We provide an explicit example surface for each sub-type in \Cref{example31} to \ref{example114}.

The face graph of a face-transitive surface is node-transitive, that is, the automorphism group of the resulting cubic graph acts transitively on the nodes of the graph, see \Cref{lemma:nodetransitive}. We obtain our classification of face-transitive surfaces with up to $1280$ faces from the census of cubic node-transitive graphs with up to $1280$ nodes due to Poto\v{c}nik et al.\cite{cubicvertextransitive}: For each cubic node-transitive graph $\G$ from the census, we classify all cycle double covers on $\G$ defining (simplicial) face-transitive surfaces. This entails ensuring that the corresponding cycles are chordless and that any two cycles intersect in at most two vertices. Since, in general, the construction of such collections of cycles is a task of high complexity \cite{NPcomplete}, we resort to a more direct approach. 

Namely, we exploit that a face-transitive surface with face graph $\G$ is defined by one of finitely many types of actions of a finite number of types of subgroups of $\G$. This classification is the heart of this paper and given in \Cref{section:construction}. We then classify all suitable subgroups $H\leq \Aut(\G)$ and use their properties to explicitly compute cycle double covers and hence face-transitive surfaces of a given vertex-face type. 
We describe the conditions for our subgroups for all vertex-face types separately in \Cref{vf31,vf22,vf21,vf11,vf12,vf13,vf16}. 
Note that in this setting, even with $\G$ and $H \leq \Aut(\G)$ fixed, we may still encounter multiple, non-isomorphic face-transitive surfaces, see \Cref{rem:multiple}.

\paragraph{Structure of the Paper}

In \Cref{section:TheoreticalBackground}, we formally present the essential notions and definitions that are necessary for conducting our investigations. In \Cref{subsection:SimplicialSurfaces}, we introduce the theory of simplicial surfaces and in \Cref{subsection:ColouringsofSimplicialSurfaces} we examine different colourings of simplicial surfaces. In  \Cref{subsection:facegraphs} we present the face graph of a simplicial surface and its properties. \Cref{sec:groupactions} introduces group actions of the automorphism group of face-transitive surfaces and their face graphs. In \Cref{section:facetransitivesurfaces}, we discuss face-transitive surfaces. We elaborate on the possible orders of their automorphism groups and examine the actions of their automorphism groups on its vertices, edges, and faces. As a result, we present the vertex-face type of a face-transitive surface and prove that there are precisely seven of them. In \Cref{section:construction} we analyse the different types of face-transitive surfaces and provide procedures to construct them from their face graphs and their automorphism group. We translate these results into an algorithm to enumerate all face-transitive surfaces with a given node-transitive cubic graph as face graph. The algorithm is practical, except for cubic graphs with excessively large automorphism groups. For a large class of such graphs we present theoretical obstructions in \Cref{section:problems} that make a complete classification up to $1280$ faces feasible. We present our results and discuss our implementations in \Cref{section:implementation}.

The census of face-transitive surfaces with up to $1280$ faces is available from \cite{facetransitivesurfaces}. We use the computer algebra systems GAP \cite{GAP4} and Magma \cite{magma} to implement our classification algorithm. In particular, we make use of the GAP-packages \texttt{simpcomp} \cite{simpcomp}, \texttt{SimplicialSurfaces} \cite{simplicialsurfacegap}, \texttt{GraphSym} \cite{cubicvertextransitive} and \texttt{Digraphs} \cite{DeBeule2024aa} to examine and construct face-transitive surfaces and cubic graphs. We further employ Magma in our studies to compute subgroups of automorphism groups with prescribed group orders. The implementation is available from \cite{facetransitivesurfaces}.

\subsection*{Acknowledgements}

R.\ Akpanya acknowledges funding by the Deutsche Forschungsgemeinschaft (DFG, German Research Foundation) in the framework of the Collaborative Research Centre CRC/TRR 280 ``Design Strategies for Material-Minimized Carbon Reinforced Concrete Structures – Principles of a New Approach to Construction'' (project ID 417002380). This project was started while R.\ Akpanya was visiting J.\ Spreer at the University of Sydney. The research visit as well as research by J.\ Spreer was partially funded by the Australian Research Council under the Discovery Project Scheme, grant number DP220102588.

\section{Preliminaries}
\label{section:TheoreticalBackground}
This section contains notions and definitions that are necessary to study simplicial, and in particular, face-transitive surfaces.
\subsection{Simplicial surfaces}
\label{subsection:SimplicialSurfaces}
All simplicial surfaces considered in this paper are simplicial complexes. Hence, we use their language to formally define simplicial surfaces. First, we give the definition of a simplicial complex.
\begin{definition}
Let $P$ be a non-empty finite set. A \emph{\textbf{simplicial complex}} $X$ is a subset $\emptyset \neq X\subseteq \Pow(P)$ which is closed under taking subsets. If $s\subseteq x$ holds for $s,x\in X$, we say that {\em \bf $s$ is incident to $x$}.
\end{definition}
In particular, a simplicial complex requires every subset of an element in the complex to also be contained in the simplicial complex. It follows that a simplicial complex is uniquely determined by the set of its inclusion maximal elements, its {\bf facets}. 
To further simplify dealing with the incidence structure of a simplicial complex, we introduce the following notation.
\begin{definition}
  For a simplicial complex $X,$ we define the set containing all $x\in X$ satisfying $\vert x\vert =i+1$ as $X_{i}$.
  Moreover, for $x\in X$ we define the set $X_i(x)$ as 
  \begin{align*}
       X_i(x):=
       \begin{cases}
    \{s \mid s\subseteq x, \text{ and } s\in X_i\}, & i \leq \vert x\vert \\
    \{s \mid x\subseteq s\in X_i\}, & i>\vert x\vert.
  \end{cases}
  \end{align*}
\end{definition}
For simplicity, we denote the one-element subsets $\{v\}$ of a simplicial complex $X$ by $v$. Thus, we identify the set $X_0$ with a subset of the underlying finite set $P$ of the complex $X$. 
We refer to the elements of $X_0$ as {\bf vertices}, the elements of $X_1$ as {\bf edges}, and the elements of $X_2$ as {\bf triangles} or {\bf faces} of $X$. The {\bf dimension} of $X$ is defined as the cardinality of the largest facet minus one, and it is said to be {\bf pure}, if all its facets have equal cardinality. 

The simplicial complexes of interest in this paper are what we call {\bf simplicial surfaces}, i.e., complexes that simplicially decompose 2-dimensional manifolds. Their formal definition below is tailored to our needs. A more general definition of these surfaces can be found in \cite{simplicialsurfacesbook}.
\begin{definition}
  Let $X$ be a simplicial complex. The complex $X$ is a \emph{\textbf{simplicial surface}}, if it satisfies the following conditions:
  \begin{itemize}
    \item $X$ is pure and of dimension $2$,
    \item every edge of $X$ is contained in exactly two triangles, that is, $\vert X_2(e)\vert  = 2$ for all $e \in X_1$,
    \item for all $v\in X_0$ there exists an $n>0$ such that $X_2(v)=\{F_1,\ldots,F_n\}$ holds. These faces can be arranged in a sequence $(F_1,\ldots,F_n)$ such that $\vert F_i\cap F_{i+1}\vert =2$ for $i=1,\ldots, n,$ where we read the subscripts modulo $n$. We call $\deg_X(v)=n$ the \emph{\bf degree} of $v$. We denote the sequence $(F_1,\ldots,F_n)$ by $u(v)$ and call it the \emph{\bf umbrella} of $v$, or a \emph{\bf vertex-defining umbrella} of $X$.
  \end{itemize}
\end{definition}
As an example, we refer the reader to the simplicial tetrahedron shown in \Cref{fig:tetrahedron}, which can be described by facets $\{\{1,2,3\},\{1,2,4\},\{1,3,4\},\{2,3,4\}\}$.
The simplicial surfaces in this paper are (strongly) connected, i.e.\ for any two faces $F_1,F_n$ of a simplicial surface $X$ there exists a path of faces $(F_1,e_1,\ldots,e_{n-1},F_n)$ with $e_i\in X_1(F_i) \cap X_1(F_{i+1})$ for $i=1,\ldots,n-1$. The {\bf Euler-Characteristic} of a simplicial surface is given by the alternating sum $\chi (X) = |X_0| - |X_1| + |X_2|$. It is a topological invariant of the underlying surface $X$. We have the following notion of a homomorphism of simplicial surfaces.
{\begin{definition}
Let $X$ and $Y$ be simplicial surfaces. A map $\phi:X\to Y$ is called a \emph{\textbf{homomorphism}} between the two simplicial surfaces, if the incidence $s\subseteq x$ in $X$ implies the incidence $\phi(s) \subseteq \phi(x)$ in $Y$. 
We define iso-, mono- and epimorphisms as usual.  An isomorphism $\phi:X\to X$ is called an \textbf{\emph{automorphism}} of  $X$. We denote the set of all automorphisms of $X$ by $\Aut(X)$, which is a group with the composition of maps as group multiplication. 
\end{definition}
Note that a homomorphism of simplicial complexes is sometimes referred to as a simplicial map. We observe that, given a simplicial surface $X$, there exists a natural action of $\Aut(X)$ on $X$ as a set via $\Aut(X)\times X\to X,(\phi,x)\mapsto \phi(x)$. With respect to this action, the \textbf{\emph{orbit}} of $x\in X$ under a subgroup $H\leq \Aut(X)$ is denoted by $x^H$. 

An automorphism of a simplicial surface is uniquely determined by the images of the vertices of an arbitrary face of the surface. More precisely, we have the following statement.

\begin{lemma}\label{lemma:autunique}
  Let $X$ be a simplicial surface, $F=\{v_1,v_2,v_3\}\in X_2$ a face and $\phi$ an automorphism of $X$. If there exists another automorphism $\psi\in \Aut(X)$ with $\phi(v_i)=\psi(v_i)$ for $i=1,2,3,$ then $\phi=\psi$.
\end{lemma}

\begin{proof}
    This is a direct consequence of the fact that $X$ is strongly connected: fixing the images of the vertices of one face fixes the images of the adjacent faces, and iterating this argument proves the lemma.
\end{proof}

\subsection{Colourings of simplicial surfaces}
\label{subsection:ColouringsofSimplicialSurfaces}
Throughout this article, we make frequent use of colouring the vertices and edges of a simplicial surface. In particular, we have the following definitions.
\begin{definition}
  Let $X$ be a simplicial surface. A map $c:X_0\to \{1,\ldots,k\}$ is called a \emph{\textbf{vertex-$k$-colouring}} of $X,$ if $c$ is surjective, and $c(v) \neq c(v')$ holds for all edges $\{v,v'\}\in X_1$.
\end{definition}
\begin{definition}[\cite{grunbaum1969conjecture}]
  Let $X$ be a simplicial surface. A map $\omega:X_1 \to \{1,2,3\}$ is called an \emph{\textbf{edge-colouring}}. We say that $\omega$ is a \emph{\textbf{Gr\"unbaum-colouring}} of $X,$ if the restriction of $\omega$ to the edges of a face is always a bijection.
\end{definition} 
In the figures of this paper, we illustrate the colours $1$, $2$, and $3$ of Gr\"unbaum and vertex-colourings with the colours red, green and blue, respectively. 
For an example, consider the Gr\"unbaum-colouring of the edges of the simplicial tetrahedron $\tetrahedron$ as indicated in \Cref{fig:wildtetrahedron}. A vertex-$4$-colouring of the simplicial tetrahedron $\tetrahedron$ is shown in \Cref{fig:vcdt}.
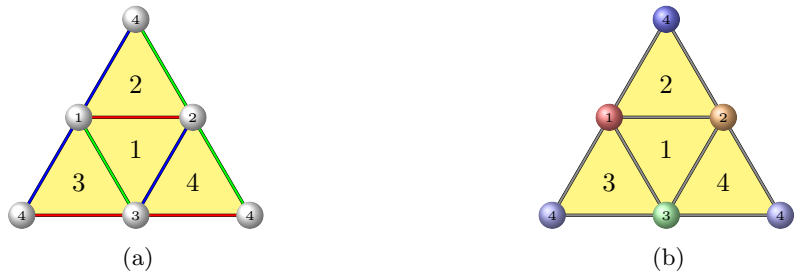
\begin{figure}[H]
  \centering
  \begin{subfigure}{.45\textwidth}
   \centering
   \scalebox{1.}{\begin{tikzpicture}[vertexBall, edgeDouble=nolabels, faceStyle, scale=1.5]

% Define the coordinates of the vertices
\coordinate (V1_1) at (0., 0.);
\coordinate (V2_1) at (1., 0.);
\coordinate (V3_1) at (0.4999999999999999, 0.8660254037844386);
\coordinate (V4_1) at (-0.9999999999999996, -1.732050807568877);
\coordinate (V4_1) at (0.5000000000000001, -0.8660254037844386);
\coordinate (V5_1) at (1.5, 0.8660254037844388);
\coordinate (V5_2) at (1.5, -0.8660254037844385);
\coordinate (V5_3) at (0., -1.732050807568877);
\coordinate (V6_1) at (-0.4999999999999999, 0.8660254037844385);
\coordinate (V6_2) at (-0.4999999999999998, -0.8660254037844388);

% Fill in the faces
\fill[face]  (V2_1) -- (V3_1) -- (V1_1) -- cycle;
\node[faceLabel] at (barycentric cs:V2_1=1,V3_1=1,V1_1=1) {$2$};
\fill[face]  (V1_1) -- (V4_1) -- (V2_1) -- cycle;
\node[faceLabel] at (barycentric cs:V1_1=1,V4_1=1,V2_1=1) {$1$};
%\fill[face]  (V2_1) -- (V5_1) -- (V3_1) -- cycle;
%\node[faceLabel] at (barycentric cs:V2_1=1,V5_1=1,V3_1=1) {$5$};
%\fill[face]  (V6_2) -- (V3_2) -- (V5_3) -- cycle;
%\node[faceLabel] at (barycentric cs:V6_2=1,V3_2=1,V5_3=1) {$8$};
\fill[face]  (V1_1) -- (V6_2) -- (V4_1) -- cycle;
\node[faceLabel] at (barycentric cs:V1_1=1,V6_2=1,V4_1=1) {$3$};
\fill[face]  (V4_1) -- (V5_2) -- (V2_1) -- cycle;
\node[faceLabel] at (barycentric cs:V4_1=1,V5_2=1,V2_1=1) {$4$};
%\fill[face]  (V3_1) -- (V6_1) -- (V1_1) -- cycle;
%\node[faceLabel] at (barycentric cs:V3_1=1,V6_1=1,V1_1=1) {$4$};

% Draw the edges
\draw[edge=red] (V2_1) -- node[edgeLabel] {$1$} (V1_1);
\draw[edge=green] (V3_1) -- node[edgeLabel] {$2$} (V2_1);
\draw[edge=blue] (V1_1) -- node[edgeLabel=blue] {$3$} (V3_1);
\draw[edge=green] (V4_1) -- node[edgeLabel] {$4$} (V1_1);
\draw[edge=blue] (V2_1) -- node[edgeLabel] {$5$} (V4_1);
\draw[edge=red] (V4_1) -- node[edgeLabel] {$7$} (V6_2);
\draw[edge=red] (V5_2) -- node[edgeLabel] {$8$} (V4_1);
\draw[edge=green] (V2_1) -- node[edgeLabel] {$10$} (V5_2);
%\draw[edge] (V6_1) -- node[edgeLabel] {$11$} (V3_1);
%\draw[edge] (V3_2) -- node[edgeLabel] {$11$} (V6_2);
%\draw[edge] (V1_1) -- node[edgeLabel] {$12$} (V6_1);
\draw[edge=blue] (V6_2) -- node[edgeLabel] {$12$} (V1_1);

% Draw the vertices
\vertexLabelR{V1_1}{left}{$1$}
\vertexLabelR{V2_1}{left}{$2$}
\vertexLabelR{V3_1}{left}{$4$}
%\vertexLabelR{V3_2}{left}{$5$}
\vertexLabelR{V4_1}{left}{$3$}
%\vertexLabelR{V5_1}{left}{$6$}
\vertexLabelR{V5_2}{left}{$4$}
%\vertexLabelR{V5_3}{left}{$6$}
%\vertexLabelR{V6_1}{left}{$44$}
\vertexLabelR{V6_2}{left}{$4$}

\end{tikzpicture}}
  \caption{}
  \label{fig:wildtetrahedron}  
  \end{subfigure}
  \begin{subfigure}{.45\textwidth}
   \centering
   \scalebox{1.}{\begin{tikzpicture}[vertexBall, edgeDouble=nolabels, faceStyle, scale=1.5]

% Define the coordinates of the vertices
\coordinate (V1_1) at (0., 0.);
\coordinate (V2_1) at (1., 0.);
\coordinate (V3_1) at (0.4999999999999999, 0.8660254037844386);
\coordinate (V4_1) at (-0.9999999999999996, -1.732050807568877);
\coordinate (V4_1) at (0.5000000000000001, -0.8660254037844386);
\coordinate (V5_1) at (1.5, 0.8660254037844388);
\coordinate (V5_2) at (1.5, -0.8660254037844385);
\coordinate (V5_3) at (0., -1.732050807568877);
\coordinate (V6_1) at (-0.4999999999999999, 0.8660254037844385);
\coordinate (V6_2) at (-0.4999999999999998, -0.8660254037844388);

% Fill in the faces
\fill[face]  (V2_1) -- (V3_1) -- (V1_1) -- cycle;
\node[faceLabel] at (barycentric cs:V2_1=1,V3_1=1,V1_1=1) {$2$};
\fill[face]  (V1_1) -- (V4_1) -- (V2_1) -- cycle;
\node[faceLabel] at (barycentric cs:V1_1=1,V4_1=1,V2_1=1) {$1$};
%\fill[face]  (V2_1) -- (V5_1) -- (V3_1) -- cycle;
%\node[faceLabel] at (barycentric cs:V2_1=1,V5_1=1,V3_1=1) {$5$};
%\fill[face]  (V6_2) -- (V3_2) -- (V5_3) -- cycle;
%\node[faceLabel] at (barycentric cs:V6_2=1,V3_2=1,V5_3=1) {$8$};
\fill[face]  (V1_1) -- (V6_2) -- (V4_1) -- cycle;
\node[faceLabel] at (barycentric cs:V1_1=1,V6_2=1,V4_1=1) {$3$};
\fill[face]  (V4_1) -- (V5_2) -- (V2_1) -- cycle;
\node[faceLabel] at (barycentric cs:V4_1=1,V5_2=1,V2_1=1) {$4$};
%\fill[face]  (V3_1) -- (V6_1) -- (V1_1) -- cycle;
%\node[faceLabel] at (barycentric cs:V3_1=1,V6_1=1,V1_1=1) {$4$};

% Draw the edges
\draw[edge] (V2_1) -- node[edgeLabel] {$1$} (V1_1);
\draw[edge] (V3_1) -- node[edgeLabel] {$2$} (V2_1);
\draw[edge] (V1_1) -- node[edgeLabel] {$3$} (V3_1);
\draw[edge] (V4_1) -- node[edgeLabel] {$4$} (V1_1);
\draw[edge] (V2_1) -- node[edgeLabel] {$5$} (V4_1);
%\draw[edge] (V5_3) -- node[edgeLabel] {$6$} (V6_2);
\draw[edge] (V4_1) -- node[edgeLabel] {$7$} (V6_2);
\draw[edge] (V5_2) -- node[edgeLabel] {$8$} (V4_1);
%\draw[edge] (V4_1) -- node[edgeLabel] {$8$} (V5_3);
%\draw[edge] (V3_1) -- node[edgeLabel] {$9$} (V5_1);
%\draw[edge] (V5_3) -- node[edgeLabel] {$9$} (V3_2);
%\draw[edge] (V5_1) -- node[edgeLabel] {$10$} (V2_1);
\draw[edge] (V2_1) -- node[edgeLabel] {$10$} (V5_2);
%\draw[edge] (V6_1) -- node[edgeLabel] {$11$} (V3_1);
%\draw[edge] (V3_2) -- node[edgeLabel] {$11$} (V6_2);
%\draw[edge] (V1_1) -- node[edgeLabel] {$12$} (V6_1);
\draw[edge] (V6_2) -- node[edgeLabel] {$12$} (V1_1);

% Draw the vertices
\vertexLabelR[red!50]{V1_1}{left}{$1$}
\vertexLabelR[orange!50]{V2_1}{left}{$2$}
\vertexLabelR[blue!50]{V3_1}{left}{$4$}
%\vertexLabelR{V3_2}{left}{$5$}
\vertexLabelR[green!30]{V4_1}{left}{$3$}
%\vertexLabelR{V5_1}{left}{$6$}
\vertexLabelR[blue!30]{V5_2}{left}{$4$}
%\vertexLabelR{V5_3}{left}{$6$}
%\vertexLabelR{V6_1}{left}{$44$}
\vertexLabelR[blue!30]{V6_2}{left}{$4$}

\end{tikzpicture}}
  \caption{}
  \label{fig:vcdt}  
  \end{subfigure}
  \caption{(a) Gr\"unbaum-colouring and (b) vertex-$4$-colouring of the simplicial tetrahedron $\tetrahedron$}
\end{figure}

There are two types of edges in a Gr\"unbaum-colouring of a surface.
\begin{definition}
  Let $X$ be a simplicial surface, $e\in X_1$, $\omega:X_1\to \{1,2,3\}$ a Gr\"unbaum-colouring of $X$, and let $F_1,F_2\in X_2$ be faces with $X_2(e)=\{F_1,F_2\}$, and $e_i\in X_1(F_i)$, $i=1,2$, such that $e \neq e_1 \neq e_2 \neq e$ and $X_0(e)\cap X_0(e_1)\cap X_0(e_2)\neq \emptyset$. If $\omega(e_1)\neq \omega(e_2)$ we call $e$ a \textbf{\emph{rotational edge}}, see \Cref{remark:rotational}, if $\omega(e_1)=\omega(e_2)$ we call $e$ a \textbf{\emph{mirror edge}}, see \Cref{remark:mirror}.
\end{definition} 
\begin{figure}[H]
  \centering
  \begin{subfigure}{0.45\textwidth}
\centering\begin{tikzpicture}[vertexBall, edgeDouble, faceStyle, scale=2]

% Define the coordinates of the vertices
\coordinate (V1) at (0., 0.);
\coordinate (V2) at (0, 1.);
\coordinate (V3) at ( 0.8660254037844384,0.5);
\coordinate (V4) at ( -0.8660254037844384,0.5);

% Fill in the faces
\fill[face]  (V1) -- (V2) -- (V3) -- cycle;
\node[faceLabel] at (barycentric cs:V1=1,V2=1,V3=1) {$F_2$};
\fill[face]  (V1) -- (V2) -- (V4) -- cycle;
\node[faceLabel] at (barycentric cs:V1=1,V4=1,V2=1) {$F_1$};

% Draw the edges
\draw[edge=red] (V1) -- node[edgeLabel] {$e$} (V2);
\draw[edge=blue] (V3) -- (V1);
\draw[edge=green] (V4) -- (V1);
\draw[edge=green] (V3) --node[edgeLabel] {$e_2$} (V2);
\draw[edge=blue] (V4) -- node[edgeLabel] {$e_1$}(V2);

% Draw the vertices
\vertexLabelR{V1}{left}{$ $}
\vertexLabelR{V2}{left}{$ $}
\vertexLabelR{V3}{left}{$ $}
\vertexLabelR{V4}{left}{$ $}
\end{tikzpicture}
    \caption{}
    \label{remark:rotational}
  \end{subfigure}
  \begin{subfigure}{0.45\textwidth}
\centering
\begin{tikzpicture}[vertexBall, edgeDouble, faceStyle, scale=2]

% Define the coordinates of the vertices
\coordinate (V1) at (0., 0.);
\coordinate (V2) at (0, 1.);
\coordinate (V3) at ( 0.8660254037844384,0.5);
\coordinate (V4) at ( -0.8660254037844384,0.5);

% Fill in the faces
\fill[face]  (V1) -- (V2) -- (V3) -- cycle;
\node[faceLabel] at (barycentric cs:V1=1,V2=1,V3=1) {$F_2$};
\fill[face]  (V1) -- (V2) -- (V4) -- cycle;
\node[faceLabel] at (barycentric cs:V1=1,V4=1,V2=1) {$F_1$};

% Draw the edges
\draw[edge=red] (V1) -- node[edgeLabel] {$e$} (V2);
\draw[edge=green] (V3) -- (V1);
\draw[edge=green] (V4) -- (V1);
\draw[edge=blue] (V3) -- node[edgeLabel] {$e_2$}(V2);
\draw[edge=blue] (V4) -- node[edgeLabel] {$e_1$}(V2);

% Draw the vertices
\vertexLabelR{V1}{left}{$ $}
\vertexLabelR{V2}{left}{$ $}
\vertexLabelR{V3}{left}{$ $}
\vertexLabelR{V4}{left}{$ $}

\end{tikzpicture}
    \caption{}
        \label{remark:mirror}
  \end{subfigure}
  \caption{(a) Rotational edge and (b) mirror edge of a Gr\"unbaum-coloured simplicial surface}
  \label{fig:wildedges}
\end{figure}
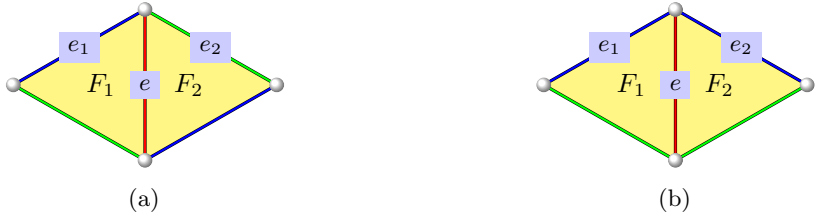

\begin{remark}
    A simplicial surface with a Gr\"unbaum colouring without rotational edges forms a {\bf graph encoding}, or {\bf balanced triangulation} \cite{novik}. In this article we sometimes implicitly describe simplicial surfaces by graph encodings without explicitly referring to them as such.
\end{remark}

Although the existence of a Gr\"unbaum-colouring on the edges of a simplicial surface is a combinatorial property, it can originate from geometric properties: A simplicial surface which can be embedded into Euclidean $3$-space as a polyhedron whose faces form congruent triangles with three distinct edge lengths admits a Gr\"unbaum colouring by using the lengths of the edges as colours. For more on this and related constructions, see  \cite{automorphism,icosahedron,coolsaet,isosceles}.

\subsection{Face graphs of simplicial surfaces}
\label{subsection:facegraphs}

All graphs in this paper are undirected, connected, simple and finite. We refer to the elements of a graph as {\bf nodes} and {\bf arcs} to distinguish them from the vertices and edges of a simplicial surface.

\begin{definition}\label{def:facegraph}
  Let $X$ be a simplicial surface. The \textbf{\emph{face graph}}, or \textbf{\emph{dual graph}}, $\F(X)$ of $X$ is defined as the graph $\F(X)=(V,E)$ with set of nodes $V=X_2$ and set of arcs $E=\{X_2(e)\mid e\in X_1\}$. Hence, two nodes $F_1,F_2\in V$ are connected in $\F(X)$ if and only if the intersection of the two faces $F_1 \cap F_2$ is an edge in $X$.
\end{definition}

Since simplicial surfaces consist of faces that are incident to three edges, the face graphs of simplicial surfaces are cubic.
We define the automorphism group $\Aut(\G)$ of a cubic graph $\G=(V,E)$ as the set of all automorphisms of $\G,$ i.e.\ the set of all bijective functions $\sigma:V\mapsto V$ satisfying $\{\sigma(\{x,y\})\mid \{x,y\}\in E\}=E$.
The automorphism group of the graph $\G$ acts on the nodes and arcs of $\G$ via $\Aut(X)\times V \to V, (\sigma,x)\mapsto \sigma(x)$ and $\Aut(X)\times E \to E, (\sigma,e)\mapsto \sigma(e)$. We denote the orbit of $x\in V$ ($e \in E)$ under a subgroup $H\leq \Aut(\G)$ by $x^H$ ($e^H$). We say that the graph $\G$ is \textbf{\emph{node-transitive}}, if $v^{\Aut(\G)}=V$. That is, the automorphism group $\Aut (\G)$ acts transitively on the nodes $V$. The face graph of the simplicial tetrahedron $\tetrahedron$ forms an example of a node-transitive face graph, see \Cref{fig:facegraphtetrahedron}.
Let $X$ be a simplicial surface and let $X_2=\{F_1,\ldots,F_n\}$ be its faces. In figures, we label the node $F_i$ of a face graph $\F(X)$ by $i$, and use the labels $1,\ldots,n$ to describe the umbrellas of $X$ and the automorphism group $\Aut(\F(X))$.

The faces and edges of a simplicial surface translate directly into the nodes and arcs of the corresponding face graph. Computing the vertices of a surface from its face graph, however, is more involved, and usually not uniquely determined: The vertices, or, more precisely, the umbrellas corresponding to the vertices of a surface appear in the face graph as what is called a {\bf cycle double cover}.

\begin{definition}[Szekeres \cite{szekeres}, Seymour \cite{seymour}]
Let $\G=(V,E)$ be a (cubic) graph and $\C$ a set of cycles of $\G$. The set $\C$ is a \emph{\textbf{cycle double cover}} of $\G,$ if every arc in $\G$ is contained in exactly two cycles in $\C$. 
\end{definition}

Note that cycles in this definition need to be embedded, i.e., they must use every arc at most once. Hence, a given cubic graph has to be bridgeless in order to have a cycle double cover. 

\begin{example}
\label{ex:cdc}
Up to isomorphism there exist exactly two cycle double covers of the face graph of the simplicial tetrahedron $\tetrahedron$ shown in \Cref{fig:facegraphtetrahedron}: 
\begin{align*}
  \C_1:=\{(1,2,3),(1,3,4),(1,2,4),(2,3,4)\}, \quad
\C_2:=\{(1,2,3,4),(1,3,2,4),(1,2,4,3)\}.
\end{align*}
\end{example}

Vertex-defining umbrellas of a simplicial surface $X$ naturally biject to the set of cycles of a cycle double cover of $\F(X)$. If $C$ is a cycle in $\F(X)$ describing an umbrella of a vertex of $X,$ we say that $C$ is a \textbf{ \emph{vertex-defining cycle} }of $X$. 
It follows that, to construct a simplicial surface from a cubic graph, we have to compute a cycle double cover such that any two cycles of this cycle double cover intersect in at most one arc.

In \Cref{ex:cdc}, the cycles in $\C_2$ are not chordless and pairwise intersect in more than two arcs, while $\C_1$ contains the vertex-defining cycles of the simplicial tetrahedron $\tetrahedron$.
We refer to a cycle double cover satisfying the two properties above, such as $\C_1$, as \emph{\bf vertex-faithful}. 
Note that there are bridgeless cubic graphs that are not isomorphic to a face graph of a simplicial surface: The generalised Petersen graph $G(9,3)$ is a cubic graph with $140$ cycle double covers, but none of them is vertex-faithful. This can be checked using the GAP-package \texttt{SimplicialSurfaces} \cite{simplicialsurfacegap}.

More generally, even the existence of arbitrary cycle double covers is not established for cubic graphs. In fact, this is the subject of a famous conjecture due to Szekeres and Seymour.
\begin{conjecture}[Cycle double cover conjecture \cite{szekeres,seymour}]  \label{cdcconjecture}
Every bridgeless (cubic) graph has a cycle double cover.
\end{conjecture}
For more work and background on cycle double covers of cubic graphs, we refer to \cite{Hoffman1991,JAEGER19851,Evencycles}.

\subsection{Group actions on simplicial surfaces and their face graphs}
\label{sec:groupactions}

We can relate the automorphism group of a simplicial surface $X$ to the automorphism group of its face graph as follows:
Let $\phi \in \Aut(X)$ be an automorphism and $F = \{v_1,v_2,v_3\} \in X_2$. The map $\phi$ naturally acts on $X_2$ by $\phi \cdot F = \phi(F)=\{\phi(v_1),\phi(v_2),\phi(v_3)\} \in X_2$. Hence, this action naturally defines a homomorphism $\Aut(X) \to \Sym(X_2)$, associating the automorphism $\phi$ to the permutation it induces on $X_2$. Since $\phi$ preserves the face, edge incidences of $X$, the image of this homomorphism can be interpreted as an automorphism of $\F(X)$ defining a homomorphism
\begin{equation}
\label{eq:lambdaX}
\lambda_X : \Aut(X) \to \Aut(\F(X)).
\end{equation}
Due to \Cref{lemma:autunique}, every automorphism of $X$ leaving a triangle and its neighbours fixed must be the identity. Hence, $\lambda_X$ is a monomorphism, and a natural way to embed the automorphism group of a simplicial surface into the automorphism group of its face graph.

In what follows we describe a group action of the automorphism group of a cubic graph on its cycles and an action of the automorphism group of a simplicial surface on its vertex-defining umbrellas.
These group actions are used in \Cref{section:construction} to construct cycle double covers of cubic node-transitive graphs that result in face-transitive surfaces. 
\begin{definition}
\label{rem:b}
Let $\G=(V,E)$ be a cubic graph, and $\Gamma$ the set of all cycles in $\G$. The group action of $\Aut(\G)$ on $V$ defines a group action of $\Aut(\G)$ on $\Gamma$ by 
\[
\Aut(\G) \times \Gamma\to \Gamma, \quad (\sigma,(F_1,\ldots,F_n))\mapsto \sigma((F_1,\ldots,F_n)):=(\sigma(F_1),\ldots,\sigma(F_n)).
\]
Furthermore, let $X$ be a simplicial surface. The action of $\Aut(X)$ on $X_2$ extends to an action on the vertex-defining umbrellas $u(X):=\{u(v)\mid v\in X_0\}$ of $X$ by the map
\[
\Aut(X) \times u(X)\to u(X),\quad(\phi,(F_1,\ldots,F_n))\mapsto \phi((F_1,\ldots,F_n)):=(\phi(F_1),\ldots,\phi(F_n)).
\]
\end{definition}

\begin{example}
Consider the simplicial tetrahedron $\tetrahedron$ with labels for vertices and faces as in \Cref{fig:tetrahedron}. A vertex-defining umbrella of the vertex labelled by $1$ is given by $u:=(1,2,3)$. Thus, the orbit of $u$ under $\Aut(\tetrahedron)$ is given by 
\begin{align*}
u^{\Aut(\tetrahedron)}=u^{\Aut(\F(\tetrahedron))}
=\{(1,2,3),(1,2,4),(1,3,4), (2,3,4)\}.
\end{align*}
This set of cycles is exactly the set of vertex-defining umbrellas (or vertex-defining cycles) of $\mathcal{T}$.
\end{example}

If $X$ is a simplicial surface, then $H:=\lambda_X(\Aut(X))$ is maximal as a subgroup of $\Aut(\F(X))$ with the property that the action of $H$ on the cycles of $\F(X)$ leaves the cycle double cover of the vertex-defining cycles of $X$ invariant. Thus, $\Aut(X)\cong \Aut(\F(X))$ holds if and only if for every vertex-defining cycle $(F_1,\ldots,F_n)$ of $X$ and every automorphism $\sigma \in \Aut(\F(X))$ the image $\sigma((F_1,\dots,F_n))=(\sigma(F_1),\dots,\sigma(F_n))$ is again a vertex-defining cycle of $X$.
Additionally, by interpreting an umbrella $u(v)$ of a vertex $v$ of $X$ as a cycle in $\F(X)$, the set
 \[\bigcup_{v\in X_0}u(v)^H\]
describes the vertex-faithful cycle double cover of the vertex-defining cycles (umbrellas) of $X$. In particular, we have the following lemma.

\begin{lemma}\label{lemma:vertexisoumbrella}
Let $X$ be a simplicial surface, and let $v_1,v_2\in X_0$ vertices with umbrellas $u(v_1)$ and $u(v_2)$ and vertex-defining cycles $C_1$ and $C_2$ in $\F(X)$, respectively. Then $v_1\in v_2^{\Aut(X)}$ if and only if $C_1 \in C_2^{\lambda_X(\Aut(X))}$.
\end{lemma}

\begin{proof}
Note that automorphisms of a simplicial surface as well as automorphisms of its face graph preserve the face-edge incidences and the node-arc incidences, respectively.
\end{proof}

\section{Face-transitive surfaces}
\label{section:facetransitivesurfaces}

The goal of this section is to derive the {\bf vertex-face type} of a face-transitive surface, an important tool for our constructions in \Cref{section:construction}.

\begin{definition}
  \label{def:fts}
Let $X$ be a simplicial surface. We say that $X$ is \textbf{\emph{face-transitive}}, if $\Aut(X)$ acts transitively on the set of faces $X_2$ of $X$.
\end{definition}

The simplicial tetrahedron $\tetrahedron$ from \Cref{section:TheoreticalBackground} is an example of a face-transitive surface. We have the following fundamental observation.

\begin{proposition}\label{lemma:nodetransitive}
  The face graph $\F(X)$ of a face-transitive surface $X$ is node-transitive. 
\end{proposition}

\begin{proof}
  This follows from noting that $\lambda_X(\Aut(X))$ is a node-transitive subgroup of $\Aut(\F(X))$.
\end{proof}

The stabiliser of an element of a simplicial surface $X$ forms a subgroup of $\Aut(X)$. We have the following lemma.

\begin{lemma}\label{lemma:conjugate}
Let $X$ be a simplicial surface, and let $x_1,x_2 \in X_i$ be two elements of $X$. If $x_1\in {x_2}^{\Aut(X)}$ holds, then the stabilisers of $x_1$ and $x_2$ are conjugate in $\Aut(X)$.
\end{lemma}
\begin{proof}
Let $S_i=\stab_{\Aut(X)}(x_i)$ be the $\Aut(X)$-stabiliser of the element $x_i,$ $i=1,2$. Since $x_1$ and $x_2$ are contained in the same $\Aut(X)$-orbit, there exists an automorphism $\phi \in \Aut(X)$ mapping $x_1$ onto $x_2$. We show that
\[
f: S_1 \to S_2, \psi\mapsto \phi \circ \psi \circ \phi^{-1}
\]
is a group isomorphism. First, for $\psi \in S_1$, we have $f(\psi)(x_2)=x_2$, and the above map is well-defined.
With 
\[
f^{-1}: S_2 \to S_1, \psi\mapsto \phi^{-1} \circ \psi \circ \phi
\]
and $f (\psi_1\circ \psi_2)=\phi \circ \psi_1 \circ \phi^{-1}\circ\phi \circ \psi_2 \circ \phi^{-1}=f(\psi_1)\circ f(\psi_2)$ for all $\psi_1,\psi_2\in S_1$
it follows that $f$ is a group isomorphism. Thus $S_1$ and $S_2$ are conjugate.
\end{proof}

The following definition is central to this paper.

\begin{definition}\label{definitionvertexfacetype}
  Let $X$ be a face-transitive surface, $F\in X_2$, and let $S=\stab_{\Aut(X)}(F)$ be the stabiliser of $F$ in $\Aut(X)$. We define the \emph{\textbf{vertex-face type}} $\vf(X)$ of $X$ as the tuple 
  \[
  \vf(X):=(\vert {X_0}^{\Aut(X)}\vert, \vert S \vert ),
  \]
where ${X_0}^{\Aut(X)}:=\{v^{\Aut(X)}\mid v\in X_0\}$.

\end{definition}

Since all the face-stabilisers are conjugate by \Cref{lemma:conjugate}, they all have the same order. Thus, $\vf(X)$ does not depend on the choice of $F$ and is therefore well-defined. 

\begin{lemma}\label{lemma:3vertexorbits}
  Let $X$ be a face-transitive surface. Then $\Aut(X)$ has at most three orbits on~$X_0$.
\end{lemma}
\begin{proof}
  Let $F=\{v_1,v_2,v_3\} \in X_2$. We show that an arbitrary vertex $v \in X_0$ must lie in the $\Aut(X)$-orbit of one of these three vertices. For this, let $F' \in X_2$ be a face containing $v$. Because $X$ is face-transitive, there exists an automorphism $\phi \in \Aut(X)$ mapping $F$ to $F'$. Since $\phi$ respects the incidences of $X$, the vertices $\phi(v_1),\phi(v_2)$ and $\phi(v_3)$ must be the vertices of the face $F'=\phi(F),$ and thus $\phi(v_i) = v$ for $i \in \{1,2,3\}$ follows, which concludes the proof.
\end{proof}
\begin{corollary}\label{lemma:3edgeorbits}
  Let $X$ be a face-transitive surface. Then $\Aut(X)$ has at most three orbits on $X_1$.
\end{corollary}
It follows from \Cref{lemma:3edgeorbits}  and the orbit-stabiliser theorem, that the orbit of an edge $e\in X_1$ of a face-transitive surface $X$ satisfies $\vert e^{\Aut(X)} \vert \in \{\tfrac{1}{2}\vert X_2\vert,\vert X_2\vert,\tfrac{3}{2}\vert X_2\vert\}$.
Moreover, we know that two vertices of a simplicial surface $X$ that are contained in the same orbit under the action of $\Aut(X)$ on $X_0$ must have the same face degree. 
This allows us to compute the Euler characteristic $\chi(X)$ of a face-transitive surface. For instance, let $X$ be a face-transitive surface with exactly one vertex-orbit and $v\in X_0$ a vertex with $\deg (v)=d$. Then the Euler-Characteristic of $X$ is given by 
\[
\chi(X)=\vert X_0 \vert-\vert X_1\vert+\vert X_2 \vert=\vert X_0 \vert-\tfrac{1}{2}\vert X_2 \vert = \vert X_0 \vert-\tfrac{d}{6}\vert X_0 \vert =\tfrac{6-d}{6}\vert X_0 \vert.
\]
This implies that a simplicial surface $X$ whose automorphism group acts transitively on the vertices $X_0$ is a simplicial torus or a simplicial Klein bottle if and only if $\deg(v)=6$ for all $v\in X_0$. See \cite{kuehnel} for a classification of such triangulations of tori.
Next, we formulate a restriction to the order of a face-stabiliser of a simplicial surface.
\begin{lemma}
  \label{lem:subgroupS3}
  Let $X$ be a simplicial surface and let $F=\{v_1,v_2,v_3\}\in X_2$ be a face.
  Then the stabiliser $\stab_{\Aut(X)}(F)$ is isomorphic to a 
  subgroup of the symmetry group $\Sym(\{v_1,v_2,v_3\})$.
\end{lemma}
\begin{proof}
Since an automorphism $\phi\in \stab_{\Aut(X)}(F)$ respects the incidence structure of $X,$ the vertices $\phi(v_1),\phi(v_2)$ and $\phi(v_3)$ must be incident to the face $\phi(F)=F$. Therefore, the vertices $v_1,v_2,v_3$ are permuted. By \Cref{lemma:autunique} the automorphism $\phi$ is uniquely identified by the images of $v_1,v_2$ and $v_3$ under $\phi$ and thus we conclude that $ \stab_{\Aut(X)}(F)$ is isomorphic to a subgroup of the symmetric group on 3 elements.
\end{proof}

Hence, by \Cref{lem:subgroupS3}, face-stabilisers in simplicial surfaces have order a divisor of six, and, by the orbit stabiliser theorem, $\vert \Aut(X) \vert = s \vert X_2 \vert,$ where $s \in \{1,2,3,6\}$, and $\vert \Aut(X)\vert \leq 6\vert X_2\vert$. This is the well-known upper bound on the number of automorphisms of a triangulated surface, closely related to Hurwitz's automorphism theorem \cite{hurwitz_ueber_1892}.

\Cref{lemma:3vertexorbits,lem:subgroupS3} combined imply that there are $4\cdot3=12$ possible vertex-face types for a face-transitive surface. In the following theorem we prove that out of these $12$ possibilities $5$ can be excluded.

\begin{theorem}\label{theorem:invariants}
Let $X$ be a face-transitive surface. Then the vertex-face type of $X$ satisfies
\[
\vf(X)\in \{(3,1),(2,1),(2,2),(1,1),(1,2),(1,3),(1,6)\}.
\]
\end{theorem}
\begin{proof}
Let $F=\{v_1,v_2,v_3\}\in X_2$ be a face of simplicial surface $X$. Using \Cref{lemma:3vertexorbits}, we give a case analysis on $n=\vert X_0^{\Aut (X)}\vert \leq 3$, the number of vertex orbits under the action of $\Aut(X)$.

\noindent
{\bf Case $n=3$.} The vertices $v_1,v_2$ and $v_3$ are contained in three pairwise disjoint $\Aut(X)$-orbits. Hence, an automorphism $\phi\in \Aut(X)$ stabilising the face $F$ has to satisfy $\phi(v_i)=v_i$ for $i=1,2,3$. By \Cref{lemma:autunique} we know that $\phi$ must be the identity in $\Aut(X)$. So, $\stab_{\Aut(X)}(F)=\{id\}$ and $\vf(X)=(3,1)$ follows.

\noindent
{\bf Case $n=2$.} If $\vert \stab_{\Aut(X)}(F)\vert =3$ holds, then $\stab_{\Aut(X)}(F)$ is cyclic. Thus, $\stab_{\Aut(X)}(F)$ contains an automorphism $\phi$ satisfying 
\[
\phi_{\mid \{v_1,v_2,v_3\}}=(v_1,v_2,v_3)\in \Sym(\{v_1,v_2,v_3\}).
\]
This means that the vertices $v_1,v_2$ and $v_3$ are all contained in the same orbit. Thus, $\vert {X_0}^{\Aut(X)}\vert=1$ which contradicts $n=2$. With similar arguments we can exclude the case $\vert \stab_{\Aut(X)}(F)\vert =6$.
\end{proof}

\section{Construction of face-transitive surfaces}\label{section:construction}
This section describes procedures to classify face-transitive surfaces, based on constructing double cycle covers of cubic node-transitive graphs. \Cref{theorem:invariants} is the key ingredient to achieve this goal: In a case analysis by vertex-face type, we examine the structure of face-transitive surfaces and describe algorithms to enumerate vertex-faithful cycle double covers in a cubic graph yielding the vertex-defining umbrellas of the corresponding simplicial surfaces. 

More precisely, given a cubic node-transitive graph $\G=(V,E)$, we first compute subgroups $H\leq \Aut(\G)$ satisfying $\vert H \vert =s\vert V\vert,$ where $s\in \{1,2,3,6\}$. Depending on the properties of these subgroups $H$, we then classify actions of $H$ on $\G$ giving rise to vertex-faithful cycle double covers of $\G$ and hence face-transitive surfaces $X$. The vertex-face type of $X$ is directly determined through the type of action of $H$ on $\G$. Note that any given graph-subgroup pair $(\G,H)$ can produce multiple non-isomorphic face-transitive surfaces. An example is given in \Cref{rem:nonisomorphic}. We start with the following definition.

\begin{definition}\label{def:cycleColours}
  Let $X$ be a simplicial surface, $\omega:X_1 \to \{1,\ldots,k\}$, $k \leq 3$, an edge-colouring of $X$ and $u(v)=(F_1,\ldots, F_n)$ the umbrella of the vertex $v\in X_0$. For simplicity, we read all subscripts modulo $n$. Furthermore, let $e_1,\ldots,e_n$ be the edges in $X$ satisfying $X_2(e_i)=\{F_i,F_{i+1}\}$.
  \begin{itemize}
    \item The umbrella $u(v)$ is called \textbf{\emph{mono-coloured}}, if 
$ \vert \{\omega(e_i)\mid i=1,\ldots, n\}\vert =1$.
  \item We say that $u(v)$ is \textbf{\emph{bi-coloured}}, if
  $ \vert \{\omega(e_i)\mid i=1 ,\ldots, n\}\vert =2$
 and the colouring satisfies $\omega(e_{i})\neq \omega(e_{i+1})$ for $i=1,\ldots, n$. 
\item We refer to a cycle as \textbf{\emph{tri-coloured}} if 
 $\vert \{\omega(e_i)\mid i=1 ,\ldots, n\}\vert =3$
 and the colouring satisfies $\omega(e_{i})\neq \omega(e_{i+1})$ for $i=1,\ldots, n$.
  \end{itemize}
  
  Moreover, we define the \emph{\bf colour of the umbrella of $v$} by
  $\omega(v):=\omega(u(v)):=(\omega(e_{1}),\ldots,\omega(e_{n}))$.
\end{definition}
An edge-colouring $\omega$ of a simplicial surface $X$ can be translated into an arc-colouring $\kappa$ of the corresponding face graph $\F(X)$ via $\kappa(\{F,F'\}):=\omega(e)$ for $e\in X_1$ with $X_2(e)=\{F,F'\}$. This way we extend the notion of mono-, bi- and tri-coloured cycles to cubic graphs. Recall that edge colours $1,2,3$ are illustrated in figures as colours red, green and blue, respectively. Moreover, recall that the automorphism group of a simplicial surface $X$ can be embedded into $\Aut(\F(X))$ by the map $\lambda_X$, see \Cref{eq:lambdaX} in \Cref{sec:groupactions}. We introduce the following refinements of  coloured umbrellas (cycles) that are essential to study the umbrellas of the face-transitive surfaces in this section.
\begin{definition}\label{def:typesOfColours}
   Let $X$ be a simplicial surface and $v\in X_0$ be a vertex of $X$ with corresponding umbrella $u(v)$.  Furthermore, let $\omega$ be a Gr\"unbaum-colouring of $X$. We call the umbrella of $v$
    \begin{itemize}
        \item a \emph{\bf$(1,2)$-umbrella of $X$}, if $\omega(u(v))=(1,2,\ldots,1,2),$
        \item a \emph{\bf$(1,2,3)$-umbrella of $X$}, if $\omega(u(v))=(1,2,3,\ldots,1,2,3)$,
        \item a \emph{\bf$(1,1,2)$-umbrella of $X$}, if $\omega(u(v))=(1,1,2,\ldots,1,1,2)$,
        \item a \emph{\bf$(1,2,3,1,3,2)$-umbrella of $X$}, if $\omega(u(v))=(1,2,3,1,3,2,\ldots,1,2,3,1,3,2)$,  
        \item a \emph{\bf$(3,1,3,2)$-umbrella of $X$}, if $\omega(u(v))=(3,1,3,2,\ldots,3,1,3,2)$.
    \end{itemize}
\end{definition}
Again, the coloured umbrellas of a simplicial surface from \Cref{def:typesOfColours} can be translated into corresponding coloured cycles of the underlying face graph.
We conclude the introduction of this section with the definition of an automorphism induced cycle of a cubic graph -- another essential tool to describe the umbrellas of face-transitive surfaces.
\begin{definition}\label{def:alpha}
Let $\G=(V,E)$ be a cubic node-transitive graph and $F_1,\ldots,F_n\in V$ nodes such that $\{F_i,F_{i+1}\}\in E$ for $i=1,\ldots,n-1$. Furthermore, let $\sigma\in \Aut(\G)$ be an automorphism of order $\ell$ satisfying $\sigma(F_1)=F_n$ and $\Sigma:= \{\sigma^i(F_j)\mid i=0,\ldots, \ell-1,\, j=1,\ldots,n-1\}$. If $\sigma$ exists, we define a \emph{\bf$\sigma$-induced $\alpha$-cycle} as
    \[\alpha(\sigma,F_1,\ldots,F_n):=(\sigma(F_1),\ldots,\sigma(F_{n-1}),\ldots,\sigma^{\ell}(F_1),\ldots,\sigma^{\ell}(F_{n-1}))\]
    if $\vert \Sigma \vert =(n-1)\ell$, and as the empty cycle $()$ otherwise.
\end{definition}

\subsection{Face-transitive surfaces with vertex-face type (3,1)}\label{vf31}

In this section we examine face-transitive surfaces with three vertex orbits.

\begin{definition}
  Let $\G=(V,E)$ be a cubic node-transitive graph and $H\leq \Aut(\G)$ such that
  \begin{enumerate}
    \item $H$ acts transitively on the nodes $V$ with $\vert H\vert =\vert V\vert,$ and
    \item the action of $H$ on $E$ results in exactly three orbits $E_1$, $E_2$ and $E_3$.
  \end{enumerate}
  The partition of $E = E_1 \cup E_2 \cup E_3$ induces an arc-colouring 
  $\kappa:E\to \{1,2,3\},$ with $\kappa(\{F,F'\}):=i$ for $\{F,F'\}\in E_i$. 
  If the set of all cycles that are bi-coloured with respect to $\kappa$ is a 
  vertex-faithful cycle double cover of $\G$, we call $H$ a \emph{\bf$(3,1)$-group} of 
  $\F(X)$ and we denote the above cycle double cover by 
  $\C^{(3,1)}(H)$. 
\end{definition}

The following theorem classifies face-transitive surfaces $X$ of vertex-face $(3,1)$.

\begin{theorem}
    \label{thm:31}
    A simplicial surface $X$ is a face-transitive surface with 
    $\vf(X)=(3,1)$ if and only if $H:=\lambda_X(\Aut(X))$ is a $(3,1)$-group of 
    $\F(X)$ and the cycle double cover $\C^{(3,1)}(H)$ contains exactly the vertex-defining cycles of $X$.
\end{theorem}

\begin{proof}
Let $X$ be a face-transitive surface with $\vf(X)=(3,1),$ and let $X_0 = V_1 \uplus V_2 \uplus V_3$
be the three vertex orbits of $X$ under $\Aut(X)$. These orbits define a partition of the
edges of $X$ into
\begin{align*}
  &\E_1:=\{\{v,v'\} \in X_1 \mid v\in V_2 \text{ and }v'\in V_3\},\\
  &\E_2:=\{\{v,v'\} \in X_1 \mid v\in V_1 \text{ and }v'\in V_3\},\\
  &\E_3:=\{\{v,v'\} \in X_1 \mid v\in V_1 \text{ and }v'\in V_2\}.
\end{align*}
Note that this is the partition of $X_1$ into $\Aut(X)$-orbits. Since every face is incident 
to exactly one edge of each class, this defines a Gr\"unbaum-colouring of $X$.
Even more, the Gr\"unbaum-colouring $\omega:X_1\to \{1,2,3\}$ defined by $\omega(e):=i$ for 
$e\in \E_i$, forces all edges of $X$ to be mirror edges with respect to $\omega,$ cf. \Cref{remark:mirror}. 
It follows that the umbrellas of $X$ must all be bi-coloured, see \Cref{fig:vf31}. 
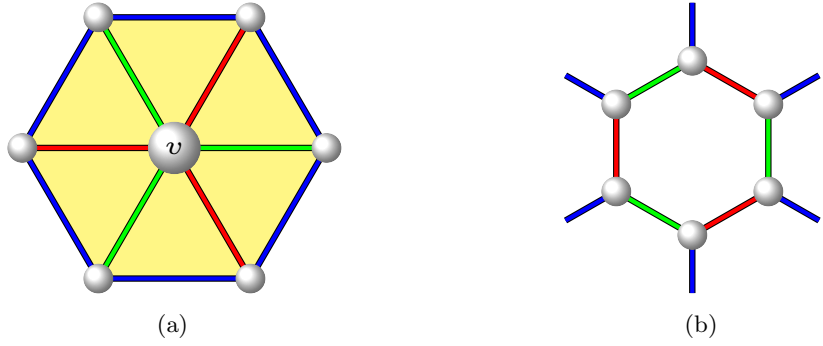
\begin{figure}[H]
  \centering
  \begin{subfigure}{0.45\textwidth}
   \centering \scalebox{2.}{\begin{tikzpicture}[vertexBall, edgeDouble, faceStyle, scale=1]

% Define the coordinates of the vertices
\coordinate (V) at (0., 0.);
\coordinate (V1) at (1., 0.);
\coordinate (V2) at (0.4999999999999999, 0.8660254037844386);
\coordinate (V3) at (-0.5000000000000001, 0.8660254037844386);
\coordinate (V4) at (-1, 0.);
\coordinate (V5) at (-0.5, -0.8660);
\coordinate (V6) at (0.5, -0.8660);

% Fill in the faces
\fill[face]  (V) -- (V1) -- (V2) -- cycle;
\fill[face]  (V) -- (V3) -- (V2) -- cycle;
\fill[face]  (V) -- (V3) -- (V4) -- cycle;
\fill[face]  (V) -- (V4) -- (V5) -- cycle;
\fill[face]  (V) -- (V5) -- (V6) -- cycle;
\fill[face]  (V) -- (V1) -- (V6) -- cycle;
\fill[face]  (V2) -- (V3) -- (V5) -- cycle;
%\node[faceLabel] at (barycentric cs:V2=1,V3=1,V5=1) {$F_3$};

% Draw the edges
\draw[edge=green] (V) --(V1);
\draw[edge=red] (V) --(V2);
\draw[edge=green] (V) --(V3);
\draw[edge=red] (V) --(V4);
\draw[edge=green] (V) --(V5);
\draw[edge=red] (V) --(V6);

\draw[edge=blue] (V1) --(V2);
\draw[edge=blue] (V3) --(V2);
\draw[edge=blue] (V3) --(V4);
\draw[edge=blue] (V4) --(V5);
\draw[edge=blue] (V5) --(V6);
\draw[edge=blue] (V1) --(V6);
% Draw the vertices
\vertexLabelR{V1}{left}{$ $}
\vertexLabelR{V2}{left}{$ $}
\vertexLabelR{V3}{left}{$ $}
\vertexLabelR{V4}{left}{$ $}
\vertexLabelR{V5}{left}{$ $}
\vertexLabelR{V6}{left}{$ $}
\vertexLabelR{V}{left}{$v$}

\end{tikzpicture}}
    \caption{}
  \label{fig:vf31}
  \end{subfigure}
  \hspace{1cm}
    \begin{subfigure}{0.3\textwidth}
    \centering\scalebox{2.}{\begin{tikzpicture}[vertexBall, edgeDouble, faceStyle, scale=1]

% Define the coordinates of the vertices
\coordinate (V) at (0., 0.);
\coordinate (V1) at (1., 0.);
\coordinate (V2) at (0.4999999999999999, 0.8660254037844386);
\coordinate (V3) at (-0.5000000000000001, 0.8660254037844386);
\coordinate (V4) at (-1, 0.);
\coordinate (V5) at (-0.5, -0.8660);
\coordinate (V6) at (0.5, -0.8660);

\coordinate (f1) at (barycentric cs:V=1,V1=1,V2=1);
\coordinate (f2) at (barycentric cs:V=1,V2=1,V3=1);
\coordinate (f3) at (barycentric cs:V=1,V3=1,V4=1);
\coordinate (f4) at (barycentric cs:V=1,V4=1,V5=1);
\coordinate (f5) at (barycentric cs:V=1,V5=1,V6=1);
\coordinate (f6) at (barycentric cs:V=1,V1=1,V6=1);

\coordinate(w1) at (barycentric cs:V=1,f1=-2.5);
\coordinate(w2) at (barycentric cs:V=1,f2=-2.5);
\coordinate(w3) at (barycentric cs:V=1,f3=-2.5);
\coordinate(w4) at (barycentric cs:V=1,f4=-2.5);
\coordinate(w5) at (barycentric cs:V=1,f5=-2.5);
\coordinate(w6) at (barycentric cs:V=1,f6=-2.5);

% Draw the edges
\draw[edge=red] (f1) --(f2);
\draw[edge=green] (f3) --(f2);
\draw[edge=red] (f4) --(f3);
\draw[edge=green] (f5) --(f4);
\draw[edge=red] (f6) --(f5);
\draw[edge=green] (f1) --(f6);
\draw[edge=blue] (f1) --(w1);
\draw[edge=blue] (f2) --(w2);
\draw[edge=blue] (f3) --(w3);
\draw[edge=blue] (f4) --(w4);
\draw[edge=blue] (f5) --(w5);
\draw[edge=blue] (f6) --(w6);

% Draw the vertices
\vertexLabelR{f1}{left}{$ $}
\vertexLabelR{f2}{left}{$ $}
\vertexLabelR{f3}{left}{$ $}
\vertexLabelR{f4}{left}{$ $}
\vertexLabelR{f5}{left}{$ $}
\vertexLabelR{f6}{left}{$ $}

\end{tikzpicture}}
    \caption{}\label{bicoloredfgvf31}
  \end{subfigure}
  \caption{(a) Vertex-defining umbrella of $X$ with $\vf(X)=(3,1)$, (b) corresponding arc-coloured subgraph in $\F(X)$}
\end{figure}
The Gr\"unbaum-colouring of $X$ corresponds to an arc-colouring $\kappa$ of $\F(X)$ via $\kappa(\{F,F'\}):=\omega(e)$ where $e\in X_1$ such that $X_2(e)=\{F,F'\}$. Note that the colour classes of $\kappa$ are exactly the $H$-orbits of the arcs of $\F(X)$. Thus, the bi-coloured umbrellas of $X$ (with respect to $\omega$) directly translate into bi-coloured cycles of $\F(X)$ (with respect to $\kappa$), see \Cref{bicoloredfgvf31}.
Moreover, since the vertex-defining cycles of $X$ are exactly the bi-coloured cycles in $\F(X)$ with respect to $\kappa,$ we conclude that $H$ is a $(3,1)$-group of $\F(X)$ and $\C^{(3,1)}(H)=\{u(v)\mid v\in X_0\},$ where we interpret an umbrella of a vertex as a cycle in the face graph $\F(X)$.

Now, let $H = \lambda_X(\Aut(X))$ be a $(3,1)$-group of $\F(X)$ and let $\kappa$ be the arc-colouring 
of $\F(X)$ resulting in the cycle double cover $\C:=\C^{(3,1)}(H)$. Since $H = \lambda_X(\Aut(X))$, and $H$ acts transitively on the nodes of $\F(X)$ with trivial node-stabiliser (according to the orbit-stabiliser theorem), we have that $X$ is face-transitive with trivial face stabiliser.
Moreover, by \Cref{lemma:vertexisoumbrella} it follows that vertices $v_1$ and $v_2$ of $X$ satisfy $v_1\in {v_2}^{\Aut(X)}$ if and only if the corresponding vertex-defining cycles $C_1$ and $C_2$ in $\C$ satisfy $C_1\in C_2^H$. Since colour classes of $\kappa$ coincide with $H$-orbits on the arcs of $\F(X),$ $u(v_1)\in u(v_2)^H$ is equivalent to $\kappa(C_1)=\kappa(C_2)$.
This implies that the action of ${\Aut(X)}$ on the vertices $X_0$  produces three orbits, and $X$ is a face-transitive surface with $\vf(X)=(3,1)$.  
\end{proof}

Note that, in \Cref{thm:31} we fix the automorphism group of $X$ to be isomorphic to the $(3,1)$-group $H$. In general, a $(3,1)$-group can construct a simplicial surface with larger automorphism group and hence a different vertex-face type. An analogous fact is true for the groups defined in the \Cref{vf31,vf22,vf21,vf11,vf12,vf13}.

\begin{corollary}\label{corollary:constructionvf31}
    Let $\G$ be a cubic node-transitive graph and $H\leq \Aut(\G)$ a $(3,1)$-group of $\G$, then $\G$ is the face graph of a face-transitive surface.
\end{corollary}

An example of a simplicial surface with vertex-face type $(3,1)$ is given in \Cref{example31}. This surface, denoted by $X^{(3,1)}$, is a simplicial real projective plane with $13$ vertices, $36$ edges and $24$ faces. Its automorphism group $\Aut(X^{(3,1)})$ and the automorphism group of its face graph are both isomorphic to the symmetric group $S_4$.
\subsection{Face-transitive surfaces with vertex-face type (2,2)}\label{vf22}

In this section we classify face-transitive surfaces with order two face stabilisers and two vertex~orbits.

\begin{definition}
  \label{def:22}
  Let $\G=(V,E)$ be a cubic node-transitive graph and $H\leq \Aut(\G)$ such that
  \begin{enumerate}
    \item $H$ acts transitively on the nodes $V$ with $\vert  H\vert =2\cdot\vert V\vert,$
    \item $E$ partitions into two $H$-orbits $E_1$ and $E_2$ with 
    $\vert E_1\vert =\vert V \vert $ and $\vert E_2\vert =\tfrac{1}{2}\vert V \vert,$ and
    \item there exist nodes $F_m',F_1',F_2'\in V$ satisfying $\{F_1',F_2'\}\in E_1$ and 
    $\{F_m',F_1'\}\in E_2,$ and an automorphism $\sigma \in H$ satisfying 
    $\sigma(F_m')=F_2'$.
  \end{enumerate}
  The partition $E = E_1 \cup E_2$ defines an arc-colouring $\kappa:E\to \{1,2\},$ by 
  setting $\kappa(\{F,F'\}):=i$ for $\{F,F'\}\in E_i$. Let $u$ be a mono-coloured cycle in
  $\G$ with respect to $\kappa$. If $\C:=u^H \cup {\alpha(\sigma,F_m',F_1',F_2')}^H$ 
  forms a vertex-faithful cycle double 
  cover of $\G$, we say that $H$ is a \emph{\bf$(2,2)$-group} of $\G$. We 
  denote the above cycle double cover by $\C^{(2,2)}(H)$. 
\end{definition}

With \Cref{def:22} in place, the structure of face-transitive surfaces with vertex-face type $(2,2)$ can be described as follows.

\begin{theorem}\label{theorem22}
 Let $X$ be a simplicial surface. Then $X$ is face-transitive and the vertex-face type of $X$ satisfies $\vf(X)=(2,2)$ if and only if $H:=\lambda_X(\Aut(X))$ is a $(2,2)$-group of $\F(X)$ and the cycle double cover $\C^{(2,2)}(H)$ contains exactly the vertex-defining cycles of $X$.
\end{theorem}
\begin{proof}
Let $X$ be a face-transitive surface satisfying $\vf(X)=(2,2)$. Then, 
$\vert\Aut(X)\vert = 2\cdot \vert X_2\vert,$ and the vertices of $X$ partition into two 
$\Aut(X)$-orbits $V_1$ and $V_2$. Without loss of generality, assume that every face of 
$X$ is incident to exactly one vertex in $V_1$ and exactly two vertices in $V_2$. This 
defines a partition of $X_1$ into 
\begin{align*}
  \E_1:=\{\{v,v'\} \in X_1 \mid v\in V_1, \ v'\in V_2\}, \text{ and } \E_2:=\{\{v,v'\} \in X_1 \mid v\in V_2, \ v'\in V_2\}.
\end{align*}
Note that, since the face stabilisers are of order two, these are precisely the $\Aut(X)$-orbits of $X_1$, and that 
$\vert \E_1\vert =\vert X_2 \vert$ and $\vert \E_2\vert =\tfrac{1}{2}\vert X_2 \vert $. 
Denote the corresponding edge colouring by $\omega$, with edges in $\E_1$ coloured red, 
and edges in $\E_2$ coloured green. 
Let $v_i \in V_i$, $i=1,2$, and let $u(v_1)=(F_1,\ldots ,F_n)$ and 
$u(v_2)=(F_1',\ldots,F_m')$ be their umbrellas in $X$. The edges in $u(v_1)$ are all in 
$\E_1$ and hence the umbrella is mono-coloured with respect to $\omega$, see 
\Cref{fig:vf222m}.
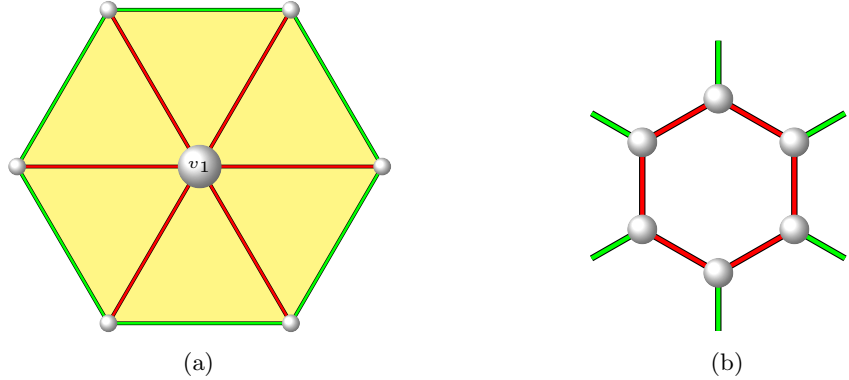
\begin{figure}[H]
  \centering
\begin{subfigure}{.45\textwidth}
   \centering\scalebox{1.2}{\begin{tikzpicture}[vertexBall, edgeDouble, faceStyle, scale=2]

% Define the coordinates of the vertices
\coordinate (V) at (0., 0.);
\coordinate (V1) at (1., 0.);
\coordinate (V2) at (0.4999999999999999, 0.8660254037844386);
\coordinate (V3) at (-0.5000000000000001, 0.8660254037844386);
\coordinate (V4) at (-1, 0.);
\coordinate (V5) at (-0.5, -0.8660);
\coordinate (V6) at (0.5, -0.8660);

% Fill in the faces
\fill[face]  (V) -- (V1) -- (V2) -- cycle;
\fill[face]  (V) -- (V3) -- (V2) -- cycle;
\fill[face]  (V) -- (V3) -- (V4) -- cycle;
\fill[face]  (V) -- (V4) -- (V5) -- cycle;
\fill[face]  (V) -- (V5) -- (V6) -- cycle;
\fill[face]  (V) -- (V1) -- (V6) -- cycle;
\fill[face]  (V2) -- (V3) -- (V5) -- cycle;
%\node[faceLabel] at (barycentric cs:V2=1,V3=1,V5=1) {$F_3$};

% Draw the edges
\draw[edge=red] (V) --(V1);
\draw[edge=red] (V) --(V2);
\draw[edge=red] (V) --(V3);
\draw[edge=red] (V) --(V4);
\draw[edge=red] (V) --(V5);
\draw[edge=red] (V) --(V6);

\draw[edge=green] (V1) --(V2);
\draw[edge=green] (V3) --(V2);
\draw[edge=green] (V3) --(V4);
\draw[edge=green] (V4) --(V5);
\draw[edge=green] (V5) --(V6);
\draw[edge=green] (V1) --(V6);
% Draw the vertices
\vertexLabelR{V1}{left}{$ $}
\vertexLabelR{V2}{left}{$ $}
\vertexLabelR{V3}{left}{$ $}
\vertexLabelR{V4}{left}{$ $}
\vertexLabelR{V5}{left}{$ $}
\vertexLabelR{V6}{left}{$ $}
\vertexLabelR{V}{left}{$v_1$}

\end{tikzpicture}}
  \caption{}
  \label{fig:vf222m}
\end{subfigure}
\hspace{1cm}
\begin{subfigure}{.3\textwidth}
\centering  \scalebox{2}{\begin{tikzpicture}[vertexBall, edgeDouble, faceStyle, scale=1]

% Define the coordinates of the vertices
\coordinate (V) at (0., 0.);
\coordinate (V1) at (1., 0.);
\coordinate (V2) at (0.4999999999999999, 0.8660254037844386);
\coordinate (V3) at (-0.5000000000000001, 0.8660254037844386);
\coordinate (V4) at (-1, 0.);
\coordinate (V5) at (-0.5, -0.8660);
\coordinate (V6) at (0.5, -0.8660);

\coordinate (f1) at (barycentric cs:V=1,V1=1,V2=1);
\coordinate (f2) at (barycentric cs:V=1,V2=1,V3=1);
\coordinate (f3) at (barycentric cs:V=1,V3=1,V4=1);
\coordinate (f4) at (barycentric cs:V=1,V4=1,V5=1);
\coordinate (f5) at (barycentric cs:V=1,V5=1,V6=1);
\coordinate (f6) at (barycentric cs:V=1,V1=1,V6=1);

\coordinate(w1) at (barycentric cs:V=1,f1=-2.5);
\coordinate(w2) at (barycentric cs:V=1,f2=-2.5);
\coordinate(w3) at (barycentric cs:V=1,f3=-2.5);
\coordinate(w4) at (barycentric cs:V=1,f4=-2.5);
\coordinate(w5) at (barycentric cs:V=1,f5=-2.5);
\coordinate(w6) at (barycentric cs:V=1,f6=-2.5);

% Draw the edges
\draw[edge=red] (f1) --(f2);
\draw[edge=red] (f3) --(f2);
\draw[edge=red] (f4) --(f3);
\draw[edge=red] (f5) --(f4);
\draw[edge=red] (f6) --(f5);
\draw[edge=red] (f1) --(f6);
\draw[edge=green] (f1) --(w1);
\draw[edge=green] (f2) --(w2);
\draw[edge=green] (f3) --(w3);
\draw[edge=green] (f4) --(w4);
\draw[edge=green] (f5) --(w5);
\draw[edge=green] (f6) --(w6);

% Draw the vertices
\vertexLabelR{f1}{left}{$ $}
\vertexLabelR{f2}{left}{$ $}
\vertexLabelR{f3}{left}{$ $}
\vertexLabelR{f4}{left}{$ $}
\vertexLabelR{f5}{left}{$ $}
\vertexLabelR{f6}{left}{$ $}

\end{tikzpicture}}
  \caption{}
  \label{fig:fgvf221}
\end{subfigure}
\caption{(a) A mono-coloured vertex-defining umbrella of a simplicial surface $X$ with 
$\vf(X)=(2,2)$, (b) the corresponding arc-coloured subgraph in the face graph 
$\F(X)$}
  \label{fig:221}
\end{figure}
The edges in $u(v_2)$ alternate between $\E_1$ and $\E_2$ and the umbrella is bi-coloured 
with respect to $\omega$, see \Cref{fig:vf222}.

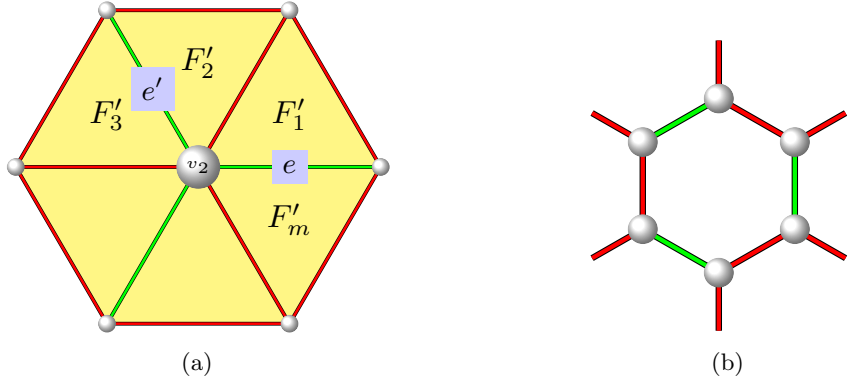
\begin{figure}[H]
  \centering
\begin{subfigure}{.45\textwidth}
\centering  \scalebox{1.2}{\begin{tikzpicture}[vertexBall, edgeDouble, faceStyle, scale=2]

% Define the coordinates of the vertices
\coordinate (V) at (0., 0.);
\coordinate (V1) at (1., 0.);
\coordinate (V2) at (0.4999999999999999, 0.8660254037844386);
\coordinate (V3) at (-0.5000000000000001, 0.8660254037844386);
\coordinate (V4) at (-1, 0.);
\coordinate (V5) at (-0.5, -0.8660);
\coordinate (V6) at (0.5, -0.8660);

% Fill in the faces
\fill[face]  (V) -- (V1) -- (V2) -- cycle;
\fill[face]  (V) -- (V3) -- (V2) -- cycle;
\fill[face]  (V) -- (V3) -- (V4) -- cycle;
\fill[face]  (V) -- (V4) -- (V5) -- cycle;
\fill[face]  (V) -- (V5) -- (V6) -- cycle;
\fill[face]  (V) -- (V1) -- (V6) -- cycle;
\fill[face]  (V2) -- (V3) -- (V5) -- cycle;
\node[faceLabel] at (barycentric cs:V=1,V1=1,V2=1) {$F_1'$};
\node[faceLabel] at (barycentric cs:V=1,V3=1,V2=1) {$F_2'$};
\node[faceLabel] at (barycentric cs:V=1,V1=1,V6=1) {$F_m'$};

\node[faceLabel] at (barycentric cs:V=1,V3=1,V4=1) {$F_3'$};
% Draw the edges
\draw[edge=green] (V) -- node[edgeLabel] {$e$}  (V1);
\draw[edge=red] (V) --(V2);
\draw[edge=green] (V) -- node[edgeLabel] {$e'$}(V3);
\draw[edge=red] (V) --(V4);
\draw[edge=green] (V) --(V5);
\draw[edge=red] (V) --(V6);
\draw[edge=red] (V1) --(V2);
\draw[edge=red] (V3) --(V2);
\draw[edge=red] (V3) --(V4);
\draw[edge=red] (V4) --(V5);
\draw[edge=red] (V5) --(V6);
\draw[edge=red] (V1) --(V6);
% Draw the vertices
\vertexLabelR{V1}{left}{$ $}
\vertexLabelR{V2}{left}{$ $}
\vertexLabelR{V3}{left}{$ $}
\vertexLabelR{V4}{left}{$ $}
\vertexLabelR{V5}{left}{$ $}
\vertexLabelR{V6}{left}{$ $}
\vertexLabelR{V}{left}{$v_2$}

\end{tikzpicture}}
  \caption{}
  \label{fig:vf222}
\end{subfigure}
\begin{subfigure}{.45\textwidth}
\centering  \scalebox{2}{\begin{tikzpicture}[vertexBall, edgeDouble, faceStyle, scale=1]

% Define the coordinates of the vertices
\coordinate (V) at (0., 0.);
\coordinate (V1) at (1., 0.);
\coordinate (V2) at (0.4999999999999999, 0.8660254037844386);
\coordinate (V3) at (-0.5000000000000001, 0.8660254037844386);
\coordinate (V4) at (-1, 0.);
\coordinate (V5) at (-0.5, -0.8660);
\coordinate (V6) at (0.5, -0.8660);

\coordinate (f1) at (barycentric cs:V=1,V1=1,V2=1);
\coordinate (f2) at (barycentric cs:V=1,V2=1,V3=1);
\coordinate (f3) at (barycentric cs:V=1,V3=1,V4=1);
\coordinate (f4) at (barycentric cs:V=1,V4=1,V5=1);
\coordinate (f5) at (barycentric cs:V=1,V5=1,V6=1);
\coordinate (f6) at (barycentric cs:V=1,V1=1,V6=1);

\coordinate(w1) at (barycentric cs:V=1,f1=-2.5);
\coordinate(w2) at (barycentric cs:V=1,f2=-2.5);
\coordinate(w3) at (barycentric cs:V=1,f3=-2.5);
\coordinate(w4) at (barycentric cs:V=1,f4=-2.5);
\coordinate(w5) at (barycentric cs:V=1,f5=-2.5);
\coordinate(w6) at (barycentric cs:V=1,f6=-2.5);

% Draw the edges
\draw[edge=red] (f1) --(f2);
\draw[edge=green] (f3) --(f2);
\draw[edge=red] (f4) --(f3);
\draw[edge=green] (f5) --(f4);
\draw[edge=red] (f6) --(f5);
\draw[edge=green] (f1) --(f6);
\draw[edge=red] (f1) --(w1);
\draw[edge=red] (f2) --(w2);
\draw[edge=red] (f3) --(w3);
\draw[edge=red] (f4) --(w4);
\draw[edge=red] (f5) --(w5);
\draw[edge=red] (f6) --(w6);

% Draw the vertices
\vertexLabelR{f1}{left}{$ $}
\vertexLabelR{f2}{left}{$ $}
\vertexLabelR{f3}{left}{$ $}
\vertexLabelR{f4}{left}{$ $}
\vertexLabelR{f5}{left}{$ $}
\vertexLabelR{f6}{left}{$ $}

\end{tikzpicture}}
  \caption{}
  \label{fig:fgvf222}
\end{subfigure}
\caption{(a) Bi-coloured vertex-defining umbrella of a simplicial surface $X$ with 
  $\vf(X)=(2,2)$, (b) the corresponding arc-coloured subgraph in the face graph 
  $\F(X)$}
  \label{fig:22}
\end{figure}
 The edge-colouring $\omega$ yields an arc-colouring $\kappa$ of $\F(X)$ defined by 
 $\kappa(\{F,F'\}):=\omega(e),$ where $e\in X_1$ with $X_2(e)=\{F,F'\}$. Hence, the colour
 classes of $\kappa$ form $H$-orbits on the arcs of $\F(X)$.
Since $\kappa^{-1}(2)=\{X_2(e')\mid e'\in\E_2\}$ forms a perfect matching of $\F(X)$, the deletion of these arcs results in a disjoint union of (mono-coloured) cycles in $\F(X)$ corresponding to the 
 vertex-defining umbrellas in ${v_1}^{\Aut(X)}$ in $X$.
 
Further, to describe the cycles in $\F(X)$ corresponding to the umbrellas of $v_2$, assume 
without loss that $e \in X_1$ with $X_2(e)=\{F_m',F_1'\}$ is in $\E_2$.
Since $X$ is face-transitive there exists an automorphism $\phi\in \Aut(X)$ mapping
 $F_m'$ onto $F_2'$. Further, the stabiliser of the face $F_2'$ in $\Aut(X)$ has order 
 two (due to the orbit-stabiliser theorem). Thus, there exists a non-trivial 
 $\psi\in \stab_{\Aut(X)}(F_2')$ that stabilises the edge $e'=\phi(e)$ with 
 $X_2(e')=\{\phi(F_m'),\phi(F_1')\}=\{F_2',F_3'\}$. It follows, that one of
 $\phi$ or $\psi \circ \phi$ fixes $v_2,$  cf. the green edges in \Cref{fig:vf222}. 
 For simplicity, we assume it is $\phi$. Thus, by defining $\ell$ to be the order of $\phi,$ 
 we conclude $\phi^{\ell-1}(F'_2)=F_m'$ and the umbrella of $v_2$ can be written as 
\[
(F_1',\ldots,F_m')=(\phi(F_m'),\phi(F_1'),\ldots,\phi^{\ell}(F'_m),\phi^{\ell}(F_1')).
\]
Let $\sigma:=\lambda_X(\phi)$. The cycle 
$(F_1',\ldots,F_m')=(\sigma(F_m'),\sigma(F_1'),\ldots,\sigma^{\ell}(F'_m),\sigma^{\ell}(F_1'))=\alpha(\sigma,F_m',F_1',F_2')$
is a $\sigma$-induced $\alpha$-cycle in $\F(X)$. This implies that every vertex-defining 
cycle in $v_2^{\Aut(X)}$ is an automorphism induced $\alpha$-cycle. Altogether, the set of 
vertex-defining cycles of $X$ is given by 
\[
\C:=(F_1,\ldots,F_n)^H \cup {\alpha(\sigma,F_m',F_1',F_2')}^H.\] 
Since $H$ is acting transitively on the nodes $X_2$ with 
$\vert H\vert =\vert \Aut(X)\vert =2\cdot\vert X_2\vert, $ we conclude that $H$ is a 
$(2,2)$-group of $\F(X)$ with $\C=\C^{(2,2)}(H)$.

Now, let $H$ be a $(2,2)$-group of $\F(X)$ with vertex-faithful cycle double cover 
$\C^{(2,2)}(H)$. The arc-colouring of $\F(X)$ corresponds to an edge-colouring $\omega$ 
of $X$ via $\omega(e):=\kappa(\{F,F'\}),$ where $e\in X_1$ with $X_2(e)=\{F,F'\}$. 
Since $H = \lambda_X(\Aut(X))$, and $H$ acts transitively on the nodes of $\F(X)$ with 
$\vert H \vert = 2\cdot \vert V \vert$, we have that 
$X$ is face-transitive with face stabiliser of order two (according to the orbit-stabiliser theorem).
Moreover, due to \Cref{lemma:vertexisoumbrella}, two vertices $v_1$ and $v_2$ of $X$ lie 
in the same $\Aut(X)$-orbit if and only if the corresponding cycles 
$C_1,C_2 \in \C^{(2,2)}(H)$ are contained in the same $H$-orbit. This is equivalent to saying that the 
corresponding vertex-defining cycles of $X$ have the same colour with respect to $\kappa$. 
We therefore know that there exist exactly two $\Aut(X)$-orbits on $X_0$, namely 
the set of vertices with mono-coloured umbrella and the set of bi-coloured umbrellas, each with respect to $\omega$, respectively. Thus, $\vf(X)=(2,2)$ follows.
\end{proof}

\Cref{theorem22} yields the following immediate corollary.

\begin{corollary}\label{corollary:constructionvf22}
   Let $\G$ be a cubic node-transitive graph and $H\leq \Aut(\G)$ a $(2,2)$-group of $\G$, then $\G$ is the face graph of a face-transitive 
   surface.
\end{corollary}

For instance, the simplicial surface $X^{(2,2)}$ given in \Cref{example22} is an example of a face-transitive surface of vertex-face type $(2,2)$. It is a simplicial sphere with $7$ vertices, $15$ edges and $10$ faces. 
Its automorphism group $\Aut(X^{(2,2)})$ is isomorphic to the automorphism group of its face graph with $\Aut(X^{(2,2)})\cong D_{2\cdot 10}$.

\subsection{Face-transitive surfaces with vertex-face type (1,6)}
\label{vf16}

Here, we discuss the structure of a face-transitive surface with vertex-face type $(1,6)$. 
\begin{definition}
  Let $\G=(V,E)$ be a cubic node-transitive graph and $H\leq \Aut(\G)$ such that
  \begin{enumerate}
    \item $H$ acts transitively on the nodes $V$ with $\vert  H\vert =6\cdot\vert V\vert,$ and
    \item for $F_1,F_2\in V$, $\{F_1,F_2\} \in E,$ there exists an automorphism $\sigma\in H$ such that $\sigma(F_1) = F_2$.
  \end{enumerate}
  If $\alpha(\sigma,F_1,F_2)^{H}$ is a vertex-faithful cycle double cover of $\G$, then we say that $H$ is a \emph{\bf $(1,6)$-group of} $\G$. We denote the above cycle double cover by $\C^{(1,6)}(H)$.
\end{definition}
We give the following theorem classifying face-transitive surfaces with order six face-stabilisers.
\begin{theorem}
 Let $X$ be a simplicial surface. Then $X$ is a face-transitive surface satisfying $\vf(X)=(1,6)$ if and only if $H:=\lambda_X(\Aut(X))$ is a $(1,6)$-group of $\F(X)$ and the cycle double cover $\C:=\C^{(1,6)}(H)$ contains exactly the vertex-defining cycles of $X$.
\end{theorem}
\begin{proof} 
Let $X$ be a face-transitive surface with $\vf(X)=(1,6)$. Since all stabilisers of faces of $X$ have order $6$, we know that the automorphism group $\Aut(X)$ acts transitively on the edges and on the vertices of the surface $X$. Additionally, if $e\in X_1$ is an edge, then the stabiliser of $e$ satisfies $\vert \stab_{\Aut(X)}(e)\vert=4$ (due to the orbit-stabiliser theorem). 

Let $v\in X_0$ be a vertex and $u(v)=(F_1,\ldots,F_n)$ be the umbrella of $v$. Moreover, let $F_i$, $1\leq i \leq n$, be given by $F_i=\{v,v_i,v_{i+1}\},$ where we read the subscripts modulo $n,$ see \Cref{fig:vf1622}.
\begin{figure}[H]
    \centering
\scalebox{1.5}{\begin{tikzpicture}[vertexBall, edgeDouble, faceStyle, scale=2]

% Define the coordinates of the vertices
\coordinate (V) at (0., 0.);
\coordinate (V1) at (1., 0.);
\coordinate (V2) at (0.4999999999999999, 0.8660254037844386);
\coordinate (V3) at (-0.5000000000000001, 0.8660254037844386);
\coordinate (V4) at (-1, 0.);
\coordinate (V5) at (-0.5, -0.8660);
\coordinate (V6) at (0.5, -0.8660);

% Fill in the faces
\fill[face]  (V) -- (V1) -- (V2) -- cycle;
\fill[face]  (V) -- (V3) -- (V2) -- cycle;
\fill[face]  (V) -- (V3) -- (V4) -- cycle;
\fill[face]  (V) -- (V4) -- (V5) -- cycle;
\fill[face]  (V) -- (V5) -- (V6) -- cycle;
\fill[face]  (V) -- (V1) -- (V6) -- cycle;
\fill[face]  (V2) -- (V3) -- (V5) -- cycle;
\node[faceLabel] at (barycentric cs:V=1,V1=1,V2=1) {$F_1$};
\node[faceLabel] at (barycentric cs:V=1,V2=1,V3=1) {$F_2$};
\node[faceLabel] at (barycentric cs:V=1,V3=1,V4=1) {$F_3$};
\node[faceLabel] at (barycentric cs:V=1,V4=1,V5=1) {$F_4$};
\node[faceLabel] at (barycentric cs:V=1,V5=1,V6=1) {$F_5$};
\node[faceLabel] at (barycentric cs:V=1,V1=1,V6=1) {$F_6$};

%node[edgeLabel] {$e_k$}
% Draw the edges
\draw[edge] (V) --(V1);
\draw[edge] (V) --(V2);
\draw[edge] (V) --(V3);
\draw[edge] (V) --(V4);
\draw[edge] (V) --(V5);
\draw[edge] (V) --(V6);

\draw[edge] (V1) --(V2);
\draw[edge] (V3) --(V2);
\draw[edge] (V3) --(V4);
\draw[edge] (V4) --(V5);
\draw[edge] (V5) --(V6);
\draw[edge] (V1) --(V6);
% Draw the vertices
\vertexLabelR{V1}{left}{$v_1$}
\vertexLabelR{V2}{left}{$v_2$}
\vertexLabelR{V3}{left}{$v_3$}
\vertexLabelR{V4}{left}{$v_4$}
\vertexLabelR{V5}{left}{$v_5$}
\vertexLabelR{V6}{left}{$v_6$}
\vertexLabelR{V}{left}{$v$}

\end{tikzpicture}}
    \caption{An umbrella of a face-transitive surface $X$ }
    \label{fig:vf1622}
\end{figure}
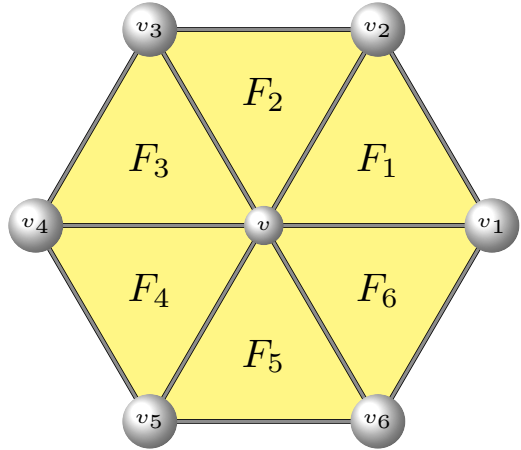

Since the stabiliser $e'=\{v,v_2\}$ has order $4$, 
there exists a unique automorphism $\psi\in \stab_{\Aut(X)} (e')$ that satisfies $(\psi(v),\psi(v_1),\psi(v_2),\psi(v_3))=(v_2,v_3,v,v_1)$. Moreover, we know that the stabiliser $\stab_{\Aut(X)}(F_2)$ contains a unique automorphism $\phi'$ with $(\phi'(v),\phi'(v_2),\phi'(v_3)) = (v_3,v,v_2)$.
Hence, the automorphism $\phi:=\phi'\circ \psi$ satisfies $\phi(v)=v$ and $\phi(F_i)=F_{i+1}$ for $i=1,\ldots,n$. We conclude that the order of $\phi $ is $n$ and the umbrella of $v$ can be written as 
\[ (F_1,\ldots,F_n)=(\phi(F_1),\ldots,\phi^{n}(F_1)).\]
We know that the umbrella $(F_1,\ldots,F_n)$ can be interpreted as a cycle of $\F(X)$. Hence, by defining $\sigma:=\lambda_X(\phi)$ we observe that  $(F_1,\ldots,F_n)=\alpha(\sigma,F_1,F_2)$ is a $\sigma$-induced $\alpha$-cycle. Since all the vertices lie in exactly one $\Aut(X)$-orbit, the cycle double cover consisting of all vertex-defining umbrellas of $X$ is given by $\alpha(\sigma,F_1,F_2)^{H}$. Thus, $H$ forms a $(1,6)$-group of $\F(X)$.

To complete the proof, assume that $H$ is a $(1,6)$-group of $\F(X)$. Let $F\in X_2$ be an arbitrary face. Since $\Aut(X)\cong H$, $X$ is face-transitive, and the stabiliser $\stab_H(F)$ of the node $F$ is isomorphic to $\stab_{\Aut(X)}(F)$, and has order six. We therefore conclude that the stabiliser $\stab_{\Aut(X)}(F)$ is transitive on the set $X_0(F)$ which implies that the automorphism group $\Aut(X)$ is transitive on the vertices of $X$. That means that $X$ is a face-transitive surface with $\vf(X)=(1,6)$.
\end{proof}

\begin{corollary}\label{corollary:constructionvf16}
    Let $\G$ be a cubic node-transitive graph and $H\leq \Aut(\G)$ a $(1,6)$-group of $\G$, then $\G$ is the face graph of a face-transitive surface.
\end{corollary}
An example of a face-transitive surface with vertex-face type $(1,6)$, namely $X^{(1,6)}$, is given in \Cref{example16}.
This surface is the well-known vertex-minimal triangulation of the projective plane consisting of $6$ vertices, $15$ edges and $10$ faces. The face graph of this surface is isomorphic to the Petersen graph, and the automorphism group satisfies
\[
A_5\cong \Aut(X^{(1,6)})\hookrightarrow \Aut(\F(X^{(1,6)})) \cong S_5.
\]
In the literature face-transitive surfaces of vertex-face type $(1,6)$ are called {\em regular maps}. For more details on regular maps we refer to \cite{CONDER2001224}.

\subsection{Face-transitive surfaces with vertex-face type (1,2)}
\label{vf12}

In this section we describe the face-transitive surfaces of vertex-face type $(1,2)$. Again, we show that the vertex-defining umbrellas of these surfaces correspond to automorphism induced $\alpha$-cycles. 
\begin{definition}
  Let $\G=(V,E)$ be a cubic node-transitive graph and $H\leq \Aut(\G)$ such that
  \begin{enumerate}
    \item $H$ acts transitively on the nodes $V$ with $\vert  H\vert =2\cdot\vert V\vert,$
    \item $E$ partitions into $E_1$ and $E_2$ with $\vert E_1 \vert=\vert V \vert$ and 
    $\vert E_2 \vert=\tfrac{1}{2}\vert V \vert$, and
    \item for $F_1,F_2,F_3,F_4\in V$ with $\{F_1,F_2\}\in E_2$ and $\{F_2,F_3\},\{F_3,F_4\}\in E_1$, there exists an automorphism $\sigma\in \Aut(\G)$ such that $\sigma(F_1)=F_4$.
  \end{enumerate}
 If $\alpha(\sigma,F_1,F_2,F_3,F_4)^H$ forms a vertex-faithful cycle double cover of $\G$, then $H$ is said to be a \textbf{\emph{$(1,2)$-group of}} $\G$. We denote the above cycle double cover by $\C^{(1,2)}(H)$.
\end{definition}
The following theorem classifies face-transitive surface $X$ with $\vf(X)=(2,1)$.
\begin{theorem}\label{thm:12}
 Let $X$ be a simplicial surface. Then $X$ is face-transitive with $\vf(X)=(1,2)$ if and only if $H:=\lambda_X(\Aut(X))$ is a $(1,2)$-group of $\F(X)$
 and the cycle double cover $\C^{(1,2)}(H)$ contains the vertex-defining cycles of $X$.
 
\end{theorem}
\begin{proof}
Let $X$ be face-transitive with $\vf(X)=(1,2)$. Thus,
for an edge $e\in X_1$ the following holds:
\begin{align*}\label{equality}
\vert \Aut(X)\vert = 2 \cdot \vert X_2 \vert = \vert \stab_{\Aut(X)}(e)\vert\cdot\vert e^{\Aut(X)} \vert.
\end{align*}
It follows that edge orbit lengths must either be $\frac12 \vert X_2 \vert$ or $\vert X_2 \vert$. This means that one edge orbit, say $\E_2$, must be of order $\frac12 \vert X_2 \vert$ with $\vert \stab_{\Aut(X)}(e)\vert = 4$ for all $e \in \E_2$. 
The stabiliser of an edge $e\in \mathcal{E}_2$ having order $4$ directly implies that $\mathcal{E}_1:=X_1\setminus \mathcal{E}_2$ forms an $\Aut(X)$-orbit of cardinality $\vert X_2 \vert$.

As usual, we obtain an edge-colouring $\omega$ of $X$ via $\omega(e):=i$ for $e\in \E_i$. This edge-colouring yields an arc-colouring $\kappa$ of $\F(X)$ via $\kappa(\{F,F'\}):=\omega(e)$, where $e\in X_1$ is an edge with $X_2(e)=\{F,F'\}$.

Let $v\in X_0$ be a vertex and $u(v)=(F_1,\ldots,F_n)$ be its umbrella. Furthermore, let the face $F_i$, $1 \leq i \leq n$, be given by $\{v,v_i,v_{i+1}\},$ where we read the subscripts modulo $n$. Without loss of generality, assume that $e_1:=\{v,v_1\} \in \E_1,$ $e_2:=\{v,v_2\}\in \E_2$ and $e_3:=\{v_1,v_2\}\in \E_1$. Recall that we have $\vert\stab_{\Aut(X)}(e_1)\vert=2$ and $\vert\stab_{\Aut(X)}(e_2)\vert =4$.

Assume that the non-trivial element $\phi\in \stab_{\Aut(X)}(e_1)$ fixes $v$ and also $v_1$. It then follows from \Cref{lemma:conjugate} that for every other $e\in \E_1$, its non-trivial automorphism $\psi\in \stab_{\Aut(X)}(e)$ fixes the endpoints of $e$ as well.
Since $v_1$ is incident to both $e_1,e_3\in \E_1,$ the above argument enforces the umbrella of $v_1$ to be mono-coloured in $X$ with $\omega(u(v_1))=(1,\ldots,1)$. This means that there cannot be an automorphism $\phi'\in \Aut(X)$ mapping $v$ onto $v_1,$ which contradicts $\vf(X)=(1,2)$.

Hence, we conclude that the automorphism $\phi$ interchanges the vertices $v$ and $v_1$.
We observe that the umbrella of $v$ must follow the pattern $\omega(u(v))=(1,1,2,\ldots,1,1,2)$. That means the umbrella of $v$ is a $(1,1,2)$-umbrella, see \Cref{fig:vf12}. 
\begin{figure}[H]
  \centering
\begin{subfigure}{.45\textwidth}
  \centering\scalebox{1.5}{\begin{tikzpicture}[vertexBall, edgeDouble, faceStyle, scale=1.5]

% Define the coordinates of the vertices
\coordinate (V) at (0., 0.);
\coordinate (V1) at (1., 0.);
\coordinate (V2) at (0.4999999999999999, 0.8660254037844386);
\coordinate (V3) at (-0.5000000000000001, 0.8660254037844386);
\coordinate (V4) at (-1, 0.);
\coordinate (V5) at (-0.5, -0.8660);
\coordinate (V6) at (0.5, -0.8660);

% Fill in the faces
\fill[face]  (V) -- (V1) -- (V2) -- cycle;
\fill[face]  (V) -- (V3) -- (V2) -- cycle;
\fill[face]  (V) -- (V3) -- (V4) -- cycle;
\fill[face]  (V) -- (V4) -- (V5) -- cycle;
\fill[face]  (V) -- (V5) -- (V6) -- cycle;
\fill[face]  (V) -- (V1) -- (V6) -- cycle;
\fill[face]  (V2) -- (V3) -- (V5) -- cycle;
\node[faceLabel] at (barycentric cs:V=1,V1=1,V2=1) {$F_2$};
\node[faceLabel] at (barycentric cs:V=1,V2=1,V3=1) {$F_3$};
\node[faceLabel] at (barycentric cs:V=1,V3=1,V4=1) {$F_4$};
\node[faceLabel] at (barycentric cs:V=1,V4=1,V5=1) {$F_5$};
\node[faceLabel] at (barycentric cs:V=1,V5=1,V6=1) {$F_6$};
\node[faceLabel] at (barycentric cs:V=1,V6=1,V1=1) {$F_1$};

% Draw the edges
\draw[edge=green] (V) --(V1);
\draw[edge=red] (V) --(V2);
\draw[edge=red] (V) --(V3);
\draw[edge=green] (V) --(V4);
\draw[edge=red] (V) --(V5);
\draw[edge=red] (V) --(V6);

\draw[edge=red] (V1) --(V2);
\draw[edge=green] (V3) --(V2);
\draw[edge=red] (V3) --(V4);
\draw[edge=red] (V4) --(V5);
\draw[edge=green] (V5) --(V6);
\draw[edge=red] (V1) --(V6);
% Draw the vertices
\vertexLabelR{V1}{left}{$v_2$}
\vertexLabelR{V2}{left}{$v_3$}
\vertexLabelR{V3}{left}{$v_4 $}
\vertexLabelR{V4}{left}{$v_5 $}
\vertexLabelR{V5}{left}{$v_6$}
\vertexLabelR{V6}{left}{$v_1$}
\vertexLabelR{V}{left}{$v$}

\end{tikzpicture}}
  \caption{}
  \label{fig:vf12}
\end{subfigure}
\hspace{2cm}
\begin{subfigure}{.3\textwidth}
 \centering \scalebox{2}{\begin{tikzpicture}[vertexBall, edgeDouble, faceStyle, scale=1]

% Define the coordinates of the vertices
\coordinate (V) at (0., 0.);
\coordinate (V1) at (1., 0.);
\coordinate (V2) at (0.4999999999999999, 0.8660254037844386);
\coordinate (V3) at (-0.5000000000000001, 0.8660254037844386);
\coordinate (V4) at (-1, 0.);
\coordinate (V5) at (-0.5, -0.8660);
\coordinate (V6) at (0.5, -0.8660);

\coordinate (f1) at (barycentric cs:V=1,V1=1,V2=1);
\coordinate (f2) at (barycentric cs:V=1,V2=1,V3=1);
\coordinate (f3) at (barycentric cs:V=1,V3=1,V4=1);
\coordinate (f4) at (barycentric cs:V=1,V4=1,V5=1);
\coordinate (f5) at (barycentric cs:V=1,V5=1,V6=1);
\coordinate (f6) at (barycentric cs:V=1,V1=1,V6=1);

\coordinate(w1) at (barycentric cs:V=1,f1=-2.5);
\coordinate(w2) at (barycentric cs:V=1,f2=-2.5);
\coordinate(w3) at (barycentric cs:V=1,f3=-2.5);
\coordinate(w4) at (barycentric cs:V=1,f4=-2.5);
\coordinate(w5) at (barycentric cs:V=1,f5=-2.5);
\coordinate(w6) at (barycentric cs:V=1,f6=-2.5);

% Draw the edges
\draw[edge=red] (f1) --(f2);
\draw[edge=red] (f3) --(f2);
\draw[edge=green] (f4) --(f3);
\draw[edge=red] (f5) --(f4);
\draw[edge=red] (f6) --(f5);
\draw[edge=green] (f1) --(f6);
\draw[edge=red] (f1) --(w1);
\draw[edge=green] (f2) --(w2);
\draw[edge=red] (f3) --(w3);
\draw[edge=red] (f4) --(w4);
\draw[edge=green] (f5) --(w5);
\draw[edge=red] (f6) --(w6);

% Draw the vertices
\vertexLabelR{f1}{left}{$ $}
\vertexLabelR{f2}{left}{$ $}
\vertexLabelR{f3}{left}{$ $}
\vertexLabelR{f4}{left}{$ $}
\vertexLabelR{f5}{left}{$ $}
\vertexLabelR{f6}{left}{$ $}

\end{tikzpicture}}
  \caption{}
  \label{fig:fgvf12}
\end{subfigure}
\caption{(a) A coloured vertex-defining umbrella of a simplicial surface $X$ with $\vf(X)=(1,2)$ (b) corresponding arc-coloured subgraph in the face graph $\F(X)$}
  \label{fig:12}
\end{figure}
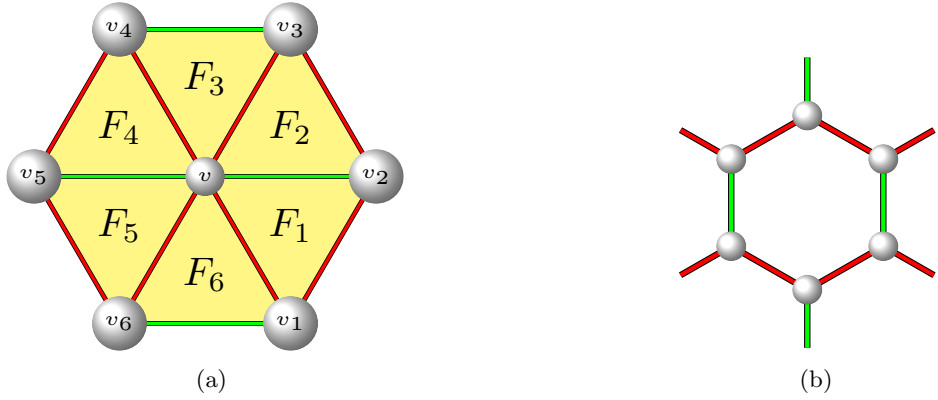
Since ${\Aut(X)}$ acts transitively on the vertices of $X$, all the umbrellas of $X$ form $(1,1,2)$-umbrellas and must be of the same degree a multiple of $3$. If the degree is $3$, then $X$ must be the tetrahedron, which has vertex-face type $(1,6)$, a contradiction. Hence, vertex degrees of $X$ must be at least $6$.

Let $\phi_1$ be the automorphism mapping $F_1$ to $F_4$ in $u(v)$. Since $\stab_{\Aut(X)}(\{v,v_5\})$ has order $4,$ there exists $\phi_2\in \stab_{\Aut(X)}(\{v,v_5\})$ such that $\overline{\phi}=\phi_2\circ \phi_1$ stabilises the vertex $v$ and satisfies $\overline{\phi}(F_1)=F_{4},$ and hence $\overline{\phi}(F_i)=F_{i+3}$.
Thus, the umbrella of $v$ can be written as 
\[
u(v)=(F_1,\ldots,F_n)=(\overline{\phi}(F_1),\overline{\phi}(F_2),\overline{\phi}(F_3),\ldots,\overline{\phi}^\ell(F_1),\overline{\phi}^\ell(F_2),\overline{\phi}^\ell(F_3)),
\]
where $\ell$ denotes the order of $\overline{\phi}$.

Since the faces $F_1,\ldots,F_n$ can also be seen as nodes of the cubic node-transitive graph $\F(X),$ the automorphism $\sigma:=\lambda_X(\Bar{\phi})$ enforces the cycle $(F_1,\ldots,F_n)=\alpha(\sigma,F_1,F_2,F_3,F_4)$ to be a $\sigma$-induced $\alpha$-cycle of $\F(X)$. Since the set of vertex-defining cycles of $X$ can be obtained from the orbit $\alpha(F_1,F_2,F_3,F_4)^H$, it follows that $H$ forms a $(1,2)$-group of $\F(X)$.

Let $H$ be a $(1,2)$-group of the face graph $\F(X)$. We know that ${\Aut(X)}$ is isomorphic to $H$. 
Thus, $\vert {\Aut(X)}\vert =\vert H\vert=2\cdot \vert X_2\vert$ holds and the node-transitivity of $H$ translates into the face-transitivity of ${\Aut(X)}$. That means that $X$ is face-transitive, with face stabilisers of order two, as can be followed from the orbit stabiliser theorem. Moreover, by the construction of $\C^{(1,2)}(H)$, it directly follows that $H$ acts transitively on the vertex-defining umbrellas, and hence ${\Aut(X)}$ must act transitively on $X_0$. It follows that $H$ uniquely determines a face-transitive surface $X$ with $\vf(X)=(1,2)$.
\end{proof}
We have the following corollary.

\begin{corollary}\label{corollary:construction_vf12}
    Let $\G$ be a cubic node-transitive graph and $H\leq \Aut(\G)$ a subgroup. If $H$ is a $(1,2)$-group of $\G$, then $\G$ is the face graph of a face-transitive surface.
\end{corollary}
An example of a face-transitive surface with vertex-face type $(1,2)$, namely $X^{(1,2)}$, is shown in \Cref{example12}.
The surface has $8$ vertices, $24$ edges and $16$ triangles. Since $X^{(1,2)}$ is orientable, this simplicial surface forms a triangulation of the torus that satisfies $\Aut(X^{(1,2)})\cong C_8\rtimes(C_2\times C_2)$ and $\Aut(\F(X^{(1,2)}))\cong \GL(2,3)\rtimes C_2$.

\subsection{Face-transitive surfaces with vertex-face type (1,3)}
\label{vf13}
In this section we present the structure of face-transitive surfaces of vertex-face type $(1,3)$. We prove that the automorphism group of such a surface falls into one of two types. For both types, we observe that the vertex-defining cycles of a face-transitive surface $X$ with $\vf(X)=(1,3)$ form automorphism induced $\alpha$-cycles of the face graph $\F(X)$.
\begin{definition}
  \label{def:13groupII}
  Let $\G=(V,E)$ be a cubic node-transitive graph and $H\leq \Aut(\G)$ such that
  \begin{enumerate}
    \item $H$ acts transitively on the nodes $V$ with $\vert  H\vert =3\cdot\vert V\vert,$
    \item for $F_1,F_2,F_3\in V$ with $\{F_1,F_2\},\{F_2,F_3\} \in E,$ there exists an automorphism $\sigma\in H$ such that $\sigma (F_1) = F_3$.
  \end{enumerate}
  Furthermore, let the orbit $\alpha(\sigma,F_1,F_2,F_3)^{H}$ be a vertex-faithful cycle double cover of $\G$. We say that 
  $H$ is a \emph{\bf$(1,3)$-group of $\G$ of type $1$} if there exists an automorphism $\pi\in \Aut(\G)$ with $\alpha(\pi,F_1,F_2)=\alpha(\sigma,F_1,F_2,F_3)$, we denote the above cycle double 
  cover by $\C_1^{(1,3)}(H)$. If there exists no such automorphism $\pi$, then we say that 
  $H$ is a \emph{\bf$(1,3)$-group of $\G$ of type $2$} and denote the above cycle double 
  cover by $\C_2^{(1,3)}(H)$.
\end{definition}

\Cref{def:13groupII} allows us to classify face-transitive surfaces $X$ of vertex-face type $\vf(X)=(1,3)$.
\begin{theorem}
 Let $X$ be a simplicial surface. Then $X$ is face-transitive with $\vf(X)=(1,3)$ if and only if there exists an $i=1,2$ such that $H:=\lambda_X(\Aut(X))$ is a $(1,3)$-group of $\F(X)$ and the cycle double cover $\C_i^{(1,3)}(H)$ contains the vertex-defining cycles of $X$.
\end{theorem}
\begin{proof}
If $X$ is a face-transitive surface with $\vf(X)=(1,3),$ then $\vert \Aut(X) \vert =3\cdot \vert X_2 \vert$, and face stabilisers are cyclic of order $3$. In particular, for all $F\in X_2$, $\stab_{\Aut(X)}(F)$ acts transitively on the edges $X_1(F)$, thus ${\Aut(X)}$ acts transitively on $X_1$, and all edge-stabilisers are conjugate in ${\Aut(X)}$, see \Cref{lemma:conjugate}.
Let $v\in X_0$ be a vertex of $X$ with umbrella $u(v)=(F_1,\ldots,F_n)$. Moreover, let $F_i=\{v,v_i,v_{i+1}\}$ and $e_i=\{v,v_i\}$, $1 \leq i \leq n$, where the subscripts are read modulo $n$. 
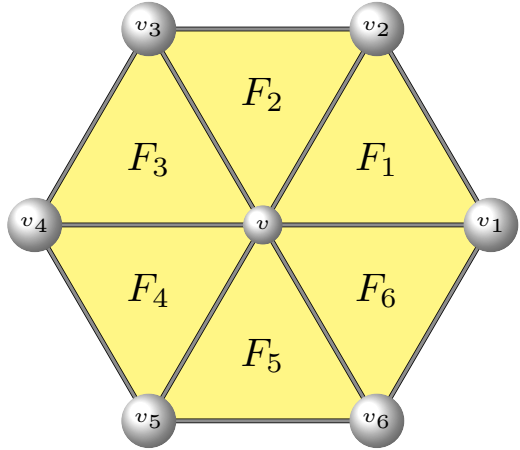
\begin{figure}[H]
    \centering
\scalebox{1.5}{\begin{tikzpicture}[vertexBall, edgeDouble, faceStyle, scale=2]

% Define the coordinates of the vertices
\coordinate (V) at (0., 0.);
\coordinate (V1) at (1., 0.);
\coordinate (V2) at (0.4999999999999999, 0.8660254037844386);
\coordinate (V3) at (-0.5000000000000001, 0.8660254037844386);
\coordinate (V4) at (-1, 0.);
\coordinate (V5) at (-0.5, -0.8660);
\coordinate (V6) at (0.5, -0.8660);

% Fill in the faces
\fill[face]  (V) -- (V1) -- (V2) -- cycle;
\fill[face]  (V) -- (V3) -- (V2) -- cycle;
\fill[face]  (V) -- (V3) -- (V4) -- cycle;
\fill[face]  (V) -- (V4) -- (V5) -- cycle;
\fill[face]  (V) -- (V5) -- (V6) -- cycle;
\fill[face]  (V) -- (V1) -- (V6) -- cycle;
\fill[face]  (V2) -- (V3) -- (V5) -- cycle;
\node[faceLabel] at (barycentric cs:V=1,V1=1,V2=1) {$F_1$};
\node[faceLabel] at (barycentric cs:V=1,V2=1,V3=1) {$F_2$};
\node[faceLabel] at (barycentric cs:V=1,V3=1,V4=1) {$F_3$};
\node[faceLabel] at (barycentric cs:V=1,V4=1,V5=1) {$F_4$};
\node[faceLabel] at (barycentric cs:V=1,V5=1,V6=1) {$F_5$};
\node[faceLabel] at (barycentric cs:V=1,V1=1,V6=1) {$F_6$};

%node[edgeLabel] {$e_k$}
% Draw the edges
\draw[edge] (V) --(V1);
\draw[edge] (V) --(V2);
\draw[edge] (V) --(V3);
\draw[edge] (V) --(V4);
\draw[edge] (V) --(V5);
\draw[edge] (V) --(V6);

\draw[edge] (V1) --(V2);
\draw[edge] (V3) --(V2);
\draw[edge] (V3) --(V4);
\draw[edge] (V4) --(V5);
\draw[edge] (V5) --(V6);
\draw[edge] (V1) --(V6);
% Draw the vertices
\vertexLabelR{V1}{left}{$v_1$}
\vertexLabelR{V2}{left}{$v_2$}
\vertexLabelR{V3}{left}{$v_3$}
\vertexLabelR{V4}{left}{$v_4$}
\vertexLabelR{V5}{left}{$v_5$}
\vertexLabelR{V6}{left}{$v_6$}
\vertexLabelR{V}{left}{$v$}

\end{tikzpicture}}
    \caption{An umbrella of a face-transitive surface $X$ }
    \label{fig:vf162}
\end{figure}
Note that, by the orbit stabiliser theorem, $\vert \stab_{\Aut(X)}(e_i)\vert=2$. Let $\phi_i$ be the non-trivial automorphism of the edge-stabiliser $\stab_{\Aut(X)}(e_i)$.  It follows, that either {\bf (1)} $\phi_i(v)=v_i$ and $\phi_i(v_i)=v$, or {\bf (2)} $\phi_i(v)=v$ and $\phi_i(v_i)=v_i$ for all $1 \leq i \leq n$.

\medskip

\noindent
{\bf Case 1:} Note that because of the stabiliser of $F_1$ being cyclic of order $3$, the automorphism $\phi_2$ cannot fix $F_2$. We therefore know $\phi_2(v_1)=v_3$ and $\phi_2(v_3)=v_1$. With this, we compute an automorphism of $X$ that applies a cyclic shift to the faces $F_i$ incident to $v,$ where $1\leq i \leq n$. Let $\psi\in \stab_{\Aut(X)}(F_2)$ be the unique automorphism satisfying $\psi(v)=v_3$, $\psi(v_2)=v$ and $\psi(v_3)=v_2$. Then $\phi:=\psi \circ \phi_2$ satisfies $\phi(v)=v,\phi(v_1)=v_2$ and $\phi(v_2)=v_3$. Since $\phi$ respects the incidences of the surface $X,$ we know that $\phi$ is an automorphism stabilising the vertex $v$ and mapping $F_i$ onto $F_{i+1}$. It follows that the umbrella of $v$ can be written as 
\[u(v)=(F_1,\ldots,F_n)=(\phi(F_1),\ldots,\phi^n(F_1)).\]

Thus, by considering $F_1,\ldots,F_n$ as nodes of the face graph $\F(X)$ and defining $\sigma:=\lambda_X(\phi)$ we observe that $(F_1,\ldots,F_n)=\alpha(\sigma,F_1,F_2)=\alpha(\sigma^2,F_1,F_2,F_3)$ is a $\sigma$-induced $\alpha$-cycle of $\F(X)$. 
Moreover, all vertices of $X$ lie in the same ${\Aut(X)}$-orbit and the vertex-defining umbrellas (cycles) of $X$ are contained in the orbit $\alpha(\sigma,F_1,F_2)^{H}$. This results in $H$ being a $(1,3)$-group of $\F(X)$ of type $1$.

\noindent
{\bf Case 2:} Since $\phi_2$ and $\phi_3$ are non-trivial, they interchange the faces $F_1$ and $F_2$, and $F_2$ and $F_3$, respectively.
Thus, define $\phi=\phi_3\circ \phi_2$ and note that, necessarily, $\phi(F_i)=F_{i+2}$ and $\phi(v)=v$. Thus, the umbrella of $v$ can be written as \[u(v)=(F_1,\ldots,F_n)=(\phi(F_1),\phi(F_2),\ldots,\phi^\ell(F_1),\phi^\ell(F_2)),\] where $\ell$ denotes the order of $\phi$.
Thus, by interpreting $F_1,\ldots,F_n$ as nodes of the face graph and considering $\sigma:=\lambda_X(\phi),$ we see that $(F_1,\ldots,F_n)=\alpha(\sigma,F_1,F_2,F_3)$ is a $\sigma$-induced $\alpha$-cycle of $\F(X)$. 

It remains to show that there exists no $\pi\in \Aut(\F(X))$ with $\alpha(\pi,F_1,F_2)=\alpha(\sigma,F_1,F_2,F_3).$ We assume that there is such an automorphism $\pi.$ This automorphism $\pi$ must be the image of an automorphism $\psi\in \Aut(X)$ under $\lambda_X$ satisfying  $(F_1,\ldots,F_n)=(\psi(F_1),\ldots,\psi^n(F_1))$ in $X$.
Hence, $\pi$ maps $F_1$ either to $F_2$ or $F_n$. Without loss of generality we assume $\psi(F_1)=F_2.$ This means that the equalities $\psi(v)=v,\psi(v_1)=v_2$ and $\psi(v_2)=v_3$ hold. Since the stabiliser of $F_2$ is cyclic of order $3$ there exists a unique automorphism $\psi'\in \stab_{\Aut(X)}(F_2)$ satisfying $\psi'(v)=v_2,\psi'(v_2)=v_3$ and $\psi'(v_3)=v$. Hence the automorphism $\gamma:=\psi'\circ \psi$ satisfies $\gamma(v)=v_2$ and $\gamma(v_2)=v.$ That means $\gamma\in \stab_{\Aut(X)}(e_2)$. 
Because of $\gamma$ being non-trivial and $\gamma\notin\stab_{\Aut(X)}(v)$ we know $\gamma\neq \phi_2.$ Hence, the stabiliser of $e_2$ contains $\{id,\phi_2,\gamma\}$ and therefore has order $4$, a contradiction because $\vf(X)=(1,3)$ implies $\vert \stab_{\Aut(X)}(e)\vert=2$ for all $e\in X_1.$

Since all the vertices of $X$ lie in one ${\Aut(X)}$-orbit, the vertex-defining umbrellas of $X$ can be translated into the orbit $\alpha(\sigma,F_1,F_2,F_3)^{H}$ and hence $H$ forms a $(1,3)$-group of $\F(X)$ of type $2$.

To complete the proof, assume that $H$ is a $(1,3)$-group of $\F(X)$ of either type.
Let $F\in X_2$ be an arbitrary face. Since $\Aut(X)\cong H$, $X$ is face-transitive, and the stabiliser $\stab_H(F)$ of the node $F$ is isomorphic to $\stab_{\Aut(X)}(F)$ of the face $F$, hence, cyclic of order $3$. We conclude that $\stab_{\Aut(X)}(F)$ is transitive on $X_0(F)$, ${\Aut(X)}$ is transitive on the vertices of $X$, and hence $X$ is a face-transitive surface with $\vf(X)=(1,3)$.
\end{proof}
As usual, we have the following corollary.
\begin{corollary}\label{corollary:constructionvf13}\label{corollary:constructionvf132}
    Let $\G$ be a cubic node-transitive graph and $H\leq \Aut(\G)$ a subgroup. If $H$ is a $(1,3)$-group of $\G$, then $\G$ is the face graph of a face-transitive surface.
\end{corollary}
A face-transitive surface of vertex-face type $(1,3)$ constructed by a $(1,3)$-group of type $1$ is often referred to as a \emph{chiral map}. 
An example of a simplicial surface that forms a chiral map is the $7$-vertex triangulation of the torus, here denoted by  $X^{(1,3)}$. This surface, given in \Cref{example131}, forms a simplicial torus with $7$ vertices, $21$ edges, and $14$ faces. Its automorphism group satisfies
\[
C_7:C_6\cong \Aut(X^{(1,3)})\hookrightarrow \Aut(\F(X^{(1,3)}))\cong \PSL(3,2) \rtimes C_2.
\]

Face-transitive surfaces constructed from $(1,3)$-groups of type $2$ are somewhat rarer, with the smallest example of such a surface being $Y^{(1,3)}$, presented in \Cref{example132}. This surface is orientable and consists of $36$ vertices, $216$ edges, and $144$ faces. Hence, the Euler-Characteristic of this simplicial surface is $\chi(Y^{(1,3)})=-36$. Note that the automorphism group of this simplicial surface satisfies $\Aut(Y^{(1,3)})\cong (((C_3 \times C_3) \rtimes Q_8) \rtimes C_3) \rtimes C_2$. This example demonstrates that a face-transitive surface with $\vert \Aut(X)\vert=3\vert X_2 \vert $ is not necessarily chiral.

\subsection{Face-transitive surfaces with vertex-face type (1,1)}\label{vf11}

In this section we classify face-transitive surfaces with trivial face-stabilisers whose automorphism groups act 
transitively on their vertices. We prove that the umbrellas representing their vertices can be either tri-coloured, or correspond to automorphism induced $\alpha$-cycles. 

\begin{definition}
  Let $\G=(V,E)$ be a cubic node-transitive graph and $H\leq \Aut(\G)$ such that
  \begin{enumerate}
    \item $H$ acts transitively on the nodes $V$ with $\vert  H\vert =1\cdot\vert V\vert,$ and
    \item $E$ partitions into three $H$-orbits $E_1$, $E_2$, and $E_3$, all of cardinality $\frac{1}{2}\vert V \vert$.
  \end{enumerate}
  The partition $E=E_1 \cup E_2 \cup E_3$ defines the arc-colouring 
  $\kappa:E\to \{1,2,3\},$ by $\kappa(\{F,F'\}):=i$ for $\{F,F'\}\in E_i$. 
  If the (unique) set consisting of all the $(1,2,3)$-cycles with respect to $\kappa$ forms a vertex-faithful cycle double cover of $\G$, then we say that
  $H$ is a \emph{\bf$(1,1)$-group of $\G$ of type $1$}, with the cycle double cover denoted by $\C_1^{(1,1)}(H)$. If the (unique) set consisting of all the $(1,2,3,1,3,2)$-cycles with respect to $\kappa$ forms a vertex-faithful cycle double cover of $\G$, then we say that $H$ is a \emph{\bf$(1,1)$-group of $\G$ of type $2$,} with the cycle double cover denoted by $\C_2^{(1,1)}(H)$.
\end{definition}
\begin{definition}
Let $\G=(V,E)$ be a cubic node-transitive graph and $H\leq \Aut(\G)$ such that
  \begin{enumerate}
    \item $H$ acts transitively on the nodes $V$ with $\vert  H\vert =1\cdot\vert V \vert,$
    \item $E$ partitions into two $H$-orbits $E_1$ and $E_2$ with 
    $\vert E_1\vert =\vert V \vert $ and $\vert E_2\vert =\tfrac{1}{2}\vert V \vert,$ and
    \item there exist nodes $F_1,\ldots,F_7\in V $  satisfying $\{F_2,F_3\},\{F_3,F_4\},\{F_5,F_6\},\{F_6,F_7\}\in E_1$ and 
    $\{F_1,F_2\},\{F_4,F_5\}\in E_2,$ and an automorphism 
    $\sigma\in \Aut (\G)$ such that $\sigma(F_1)=F_7$.
  \end{enumerate}
  The partition $E = E_1 \cup E_2$ defines an arc-colouring $\kappa:E\to \{1,2\}$ of $\G$ by 
  $\kappa(\{F,F'\}):=i$ for $\{F,F'\}\in E_i$. Let $u$ be a mono-coloured cycle in $\G$ with respect to $\kappa$ and $u^H \cup \alpha(\sigma,F_1,\ldots,F_7)^H$ a vertex-faithful cycle double cover of $\G$.
  We say that $H$ is a \emph{\bf$(1,1)$-group of $\G$ of type $3$} if there exists an automorphism $\pi\in \Aut(\G)$ with $\alpha(\pi,F_1,\ldots,F_4)=\alpha(\sigma,F_1,\ldots,F_7)$. We denote the above cycle double cover by $\C_3^{(1,1)}(H)$. If there exists no such automorphism $\pi$, then we say that $H$ is a \emph{\bf$(1,1)$-group of $\G$ of type $4$} and denote the above cycle double cover by $\C_4^{(1,1)}(H)$.
\end{definition}

In the following we refer to a subgroup $H$ of the automorphism group of a cubic 
graph as a $(1,1)$-group if the given subgroup is a $(1,1)$-group of any type.
Additionally, we denote the corresponding cycle double cover by $\C^{(1,1)}(H)$. The main result of this subsection is detailed in the below theorem.

\begin{theorem}
 Let $X$ be a simplicial surface. Then $X$ forms a face-transitive surface satisfying 
 $\vf(X)=(1,1)$ if and only if $H:=\lambda_X(\Aut(X))$ is a $(1,1)$-group of $\F(X)$ and the cycle double cover $\C^{(1,1)}(H)$ contains exactly the vertex-defining cycles of $X$.
\end{theorem}

\begin{proof}
Let $X$ be a face-transitive surface with $\vf(X)=(1,1)$. Then the automorphism group
$\Aut(X)$ has exactly $1\cdot \vert X_2 \vert $ elements. 
It follows that there are either {\bf(1)} exactly $3$ edge-orbits under $\Aut(X)$ or {\bf(2)} exactly $2$ edge-orbits under the action of $\Aut(X)$ on $X_1$.

\paragraph*{\bf Case 1:} Let $\E_1,\E_2,\E_3$ with 
$\vert \E_1 \vert=\vert \E_2 \vert=\vert \E_3 \vert=\tfrac{1}{2}\vert X_2\vert$ be the edge-orbits of the action of $\Aut(X)$ on $X_1$. These orbits naturally yield a Gr\"unbaum-colouring $\omega$ of $X$ by defining $\omega(e):=i$ for $e\in \E_i,$ and an arc-colouring $\kappa$ of $\F(X)$ that is given by $\kappa(\{F_1,F_2\}):=\omega(e)$, where $e\in X_1$ is an edge with $X_2(e)=\{F_1,F_2\}$. Since ${\Aut(X)}$ and $H$ are isomorphic, the colour classes of $\kappa$ are exactly the $H$-orbits on the arcs.
All edges that are contained in the same ${\Aut(X)}$-orbit have the same type with respect to $\omega$, see \Cref{lemma:conjugate}. That is, the edges in a chosen $\E_i$, $i \in \{1,2,3\}$, are either all rotational edges or all mirror edges with respect to $\omega$. 
To describe the vertex-defining umbrellas of $X,$ we give a case distinction with respect to the types of the  edges contained in the sets $\E_1,\E_2,\E_3$. Without loss of generality, this case distinction results in four cases:
\begin{enumerate}
  \item[(a)] all the edges of $X$ are mirror edges,
  \item[(b)] the edges in $\E_1\cup \E_2$ are mirror edges and the edges in $\E_3$ are rotational edges,
  \item[(c)] the edges in $\E_1$ are mirror edges and the edges in $\E_2\cup \E_3$ are rotational edges,
  \item[(d)] all the edges of $X$ are rotational edges.
\end{enumerate}
We show that only {\bf (c)} and {\bf (d)} occur as possible structures of the Gr\"unbaum-colouring $\omega$ of $X$.

\paragraph*{\bf Case 1(a):} If all edges of $X$ are mirror edges, then the vertex-defining umbrellas are bi-coloured with respect to $\omega$. More precisely, let $F=\{v_1,v_2,v_3\}\in X_2$ be a face and let the edges $\{v_1,v_2\},\{v_1,v_3\}$ and $\{v_2,v_3\}$ be contained in the orbits $\E_3,\E_2$ and $\E_1,$ respectively. Since all edges of $X$ are mirror edges, we know that the colours of the umbrellas $u(v_1)$, $u(v_2)$ and $u(v_3)$ are given by $\omega(u(v_1))=(2,3,\ldots,2,3)$, $\omega(u(v_2))=(1,3,\ldots,1,3)$ and $\omega(u(v_3))=(1,2,\ldots,1,2),$ respectively. In particular, $v_1$ is not incident to any edges in $\E_1$, but $v_2$ is.
Since the colour classes of $\omega$ are ${\Aut(X)}$-orbits on the edges, it follows that no automorphism in ${\Aut(X)}$ can map $v_1$ to $v_2$. But we have $\vf(X)=(1,1),$ and hence all vertices must be in the same ${\Aut(X)}$-orbit. A contradiction.
\paragraph*{\bf Case 1(b):} As before, let $F=\{v_1,v_2,v_3\}\in X_2$ be a face such that the edges that are incident to $F$ satisfy $\{v_2,v_3\}\in \E_1$, $\{v_1,v_3\}\in \E_2 $ and $\{v_1,v_2\}\in \E_3$. Since $\{v_2,v_3\}$ and $\{v_1,v_3\}$ are mirror edges, the umbrella of $v_3$ must be bi-coloured with $\omega(u(v_3))=(1,2,\ldots,1,2)$.
In particular, $v_3$ is not incident with any edge of colour $3$, but $v_1$ and $v_2$ are. Again, since the colour classes of $\omega$ are ${\Aut(X)}$-orbits on the edges, there exists no automorphism of $X$ mapping $v_3$ onto $v_2$ or $v_1$. This contradicts the fact that $\vf(X)=(1,1)$.

\paragraph*{\bf Case 1(c):} This case leads to tri-coloured vertex umbrellas, as defined in \Cref{def:cycleColours}: 
For a vertex $v\in X_0$ the colour of the corresponding umbrella $u(v)=(F_1,\ldots,F_n)$ is given by 
 \[\omega(u(v))=(1, 2, 3, 1, 3, 2,\ldots, 1, 2, 3, 1, 3, 2),\] 
 see \Cref{fig:112}. Take a moment to verify that, no matter the starting configuration, the colour of the umbrella is determined. In particular, vertex degrees must be multiples of six.
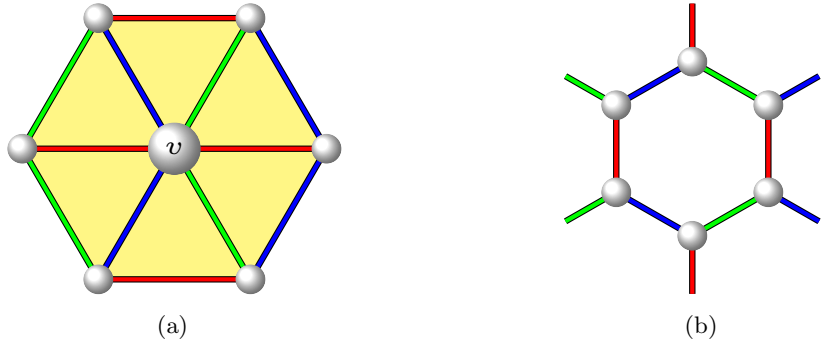
\begin{figure}[H]
  \centering
\begin{subfigure}{.45\textwidth}
 \centering \scalebox{2}{\begin{tikzpicture}[vertexBall, edgeDouble, faceStyle, scale=1]

% Define the coordinates of the vertices
\coordinate (V) at (0., 0.);
\coordinate (V1) at (1., 0.);
\coordinate (V2) at (0.4999999999999999, 0.8660254037844386);
\coordinate (V3) at (-0.5000000000000001, 0.8660254037844386);
\coordinate (V4) at (-1, 0.);
\coordinate (V5) at (-0.5, -0.8660);
\coordinate (V6) at (0.5, -0.8660);

% Fill in the faces
\fill[face]  (V) -- (V1) -- (V2) -- cycle;
\fill[face]  (V) -- (V3) -- (V2) -- cycle;
\fill[face]  (V) -- (V3) -- (V4) -- cycle;
\fill[face]  (V) -- (V4) -- (V5) -- cycle;
\fill[face]  (V) -- (V5) -- (V6) -- cycle;
\fill[face]  (V) -- (V1) -- (V6) -- cycle;
\fill[face]  (V2) -- (V3) -- (V5) -- cycle;
%\node[faceLabel] at (barycentric cs:V2=1,V3=1,V5=1) {$F_3$};

% Draw the edges
\draw[edge=red] (V) --(V1);
\draw[edge=green] (V) --(V2);
\draw[edge=blue] (V) --(V3);
\draw[edge=red] (V) --(V4);
\draw[edge=blue] (V) --(V5);
\draw[edge=green] (V) --(V6);

\draw[edge=blue] (V1) --(V2);
\draw[edge=red] (V3) --(V2);
\draw[edge=green] (V3) --(V4);
\draw[edge=green] (V4) --(V5);
\draw[edge=red] (V5) --(V6);
\draw[edge=blue] (V1) --(V6);
% Draw the vertices
\vertexLabelR{V1}{left}{$ $}
\vertexLabelR{V2}{left}{$ $}
\vertexLabelR{V3}{left}{$ $}
\vertexLabelR{V4}{left}{$ $}
\vertexLabelR{V5}{left}{$ $}
\vertexLabelR{V6}{left}{$ $}
\vertexLabelR{V}{left}{$v$}

\end{tikzpicture}}
  \caption{}
  \label{fig:vf112}
\end{subfigure}
\hspace{1cm}
\begin{subfigure}{.3\textwidth}
\centering  \scalebox{2}{\begin{tikzpicture}[vertexBall, edgeDouble, faceStyle, scale=1]

% Define the coordinates of the vertices
\coordinate (V) at (0., 0.);
\coordinate (V1) at (1., 0.);
\coordinate (V2) at (0.4999999999999999, 0.8660254037844386);
\coordinate (V3) at (-0.5000000000000001, 0.8660254037844386);
\coordinate (V4) at (-1, 0.);
\coordinate (V5) at (-0.5, -0.8660);
\coordinate (V6) at (0.5, -0.8660);

\coordinate (f1) at (barycentric cs:V=1,V1=1,V2=1);
\coordinate (f2) at (barycentric cs:V=1,V2=1,V3=1);
\coordinate (f3) at (barycentric cs:V=1,V3=1,V4=1);
\coordinate (f4) at (barycentric cs:V=1,V4=1,V5=1);
\coordinate (f5) at (barycentric cs:V=1,V5=1,V6=1);
\coordinate (f6) at (barycentric cs:V=1,V1=1,V6=1);

\coordinate(w1) at (barycentric cs:V=1,f1=-2.5);
\coordinate(w2) at (barycentric cs:V=1,f2=-2.5);
\coordinate(w3) at (barycentric cs:V=1,f3=-2.5);
\coordinate(w4) at (barycentric cs:V=1,f4=-2.5);
\coordinate(w5) at (barycentric cs:V=1,f5=-2.5);
\coordinate(w6) at (barycentric cs:V=1,f6=-2.5);

% Draw the edges
\draw[edge=green] (f1) --(f2);
\draw[edge=blue] (f3) --(f2);
\draw[edge=red] (f4) --(f3);
\draw[edge=blue] (f5) --(f4);
\draw[edge=green] (f6) --(f5);
\draw[edge=red] (f1) --(f6);
\draw[edge=blue] (f1) --(w1);
\draw[edge=red] (f2) --(w2);
\draw[edge=green] (f3) --(w3);
\draw[edge=green] (f4) --(w4);
\draw[edge=red] (f5) --(w5);
\draw[edge=blue] (f6) --(w6);

% Draw the vertices
\vertexLabelR{f1}{left}{$ $}
\vertexLabelR{f2}{left}{$ $}
\vertexLabelR{f3}{left}{$ $}
\vertexLabelR{f4}{left}{$ $}
\vertexLabelR{f5}{left}{$ $}
\vertexLabelR{f6}{left}{$ $}

\end{tikzpicture}}
  \caption{}
  \label{fig:fgvf112}
\end{subfigure}
\caption{(a) A tri-coloured vertex-defining umbrella of a simplicial surface $X$ with $\vf(X)=(1,1)$ (b) corresponding arc-coloured subgraph in the face graph $\F(X)$}
  \label{fig:112}
\end{figure}
Since ${\Aut(X)}$ acts transitively on the vertices of $X$, we conclude that all vertex-defining umbrellas of $X$ are $(1,2,3,1,3,2)$-coloured.
Hence, we obtain a vertex-faithful cycle double cover of $\F(X)$ consisting of only $(1,2,3,1,3,2)$-cycles and $H$ is a $(1,1)$-group of $\F(X)$ of type $2$.

\paragraph*{\bf Case 1(d):} By using the same procedure as in the previous case, we can conclude that, for every vertex $v \in X_0$, the colour of the corresponding umbrella $u(v)=(F_1,\ldots,F_n)$ is given by 
$\omega(u(v))=(1,2,3, \ldots,1,2,3)$.
In \Cref{fig:vf11d} we give an illustration of such an umbrella.
\begin{figure}[H]
  \centering
\begin{subfigure}{.45\textwidth}
 \centering \scalebox{2}{\begin{tikzpicture}[vertexBall, edgeDouble, faceStyle, scale=1]

% Define the coordinates of the vertices
\coordinate (V) at (0., 0.);
\coordinate (V1) at (1., 0.);
\coordinate (V2) at (0.4999999999999999, 0.8660254037844386);
\coordinate (V3) at (-0.5000000000000001, 0.8660254037844386);
\coordinate (V4) at (-1, 0.);
\coordinate (V5) at (-0.5, -0.8660);
\coordinate (V6) at (0.5, -0.8660);

% Fill in the faces
\fill[face]  (V) -- (V1) -- (V2) -- cycle;
\fill[face]  (V) -- (V3) -- (V2) -- cycle;
\fill[face]  (V) -- (V3) -- (V4) -- cycle;
\fill[face]  (V) -- (V4) -- (V5) -- cycle;
\fill[face]  (V) -- (V5) -- (V6) -- cycle;
\fill[face]  (V) -- (V1) -- (V6) -- cycle;
\fill[face]  (V2) -- (V3) -- (V5) -- cycle;
%\node[faceLabel] at (barycentric cs:V2=1,V3=1,V5=1) {$F_3$};

% Draw the edges
\draw[edge=blue] (V) --(V1);
\draw[edge=red] (V) --(V2);
\draw[edge=green] (V) --(V3);
\draw[edge=blue] (V) --(V4);
\draw[edge=red] (V) --(V5);
\draw[edge=green] (V) --(V6);

\draw[edge=green] (V1) --(V2);
\draw[edge=blue] (V3) --(V2);
\draw[edge=red] (V3) --(V4);
\draw[edge=green] (V4) --(V5);
\draw[edge=blue] (V5) --(V6);
\draw[edge=red] (V1) --(V6);
% Draw the vertices
\vertexLabelR{V1}{left}{$ $}
\vertexLabelR{V2}{left}{$ $}
\vertexLabelR{V3}{left}{$ $}
\vertexLabelR{V4}{left}{$ $}
\vertexLabelR{V5}{left}{$ $}
\vertexLabelR{V6}{left}{$ $}
\vertexLabelR{V}{left}{$v'$}

\end{tikzpicture}}
  \caption{}
  \label{fig:vf11d}
\end{subfigure}
\hspace{1cm}
\begin{subfigure}{.3\textwidth}
 \centering \scalebox{2}{\begin{tikzpicture}[vertexBall, edgeDouble, faceStyle, scale=1]

% Define the coordinates of the vertices
\coordinate (V) at (0., 0.);
\coordinate (V1) at (1., 0.);
\coordinate (V2) at (0.4999999999999999, 0.8660254037844386);
\coordinate (V3) at (-0.5000000000000001, 0.8660254037844386);
\coordinate (V4) at (-1, 0.);
\coordinate (V5) at (-0.5, -0.8660);
\coordinate (V6) at (0.5, -0.8660);

\coordinate (f1) at (barycentric cs:V=1,V1=1,V2=1);
\coordinate (f2) at (barycentric cs:V=1,V2=1,V3=1);
\coordinate (f3) at (barycentric cs:V=1,V3=1,V4=1);
\coordinate (f4) at (barycentric cs:V=1,V4=1,V5=1);
\coordinate (f5) at (barycentric cs:V=1,V5=1,V6=1);
\coordinate (f6) at (barycentric cs:V=1,V1=1,V6=1);

\coordinate(w1) at (barycentric cs:V=1,f1=-2.5);
\coordinate(w2) at (barycentric cs:V=1,f2=-2.5);
\coordinate(w3) at (barycentric cs:V=1,f3=-2.5);
\coordinate(w4) at (barycentric cs:V=1,f4=-2.5);
\coordinate(w5) at (barycentric cs:V=1,f5=-2.5);
\coordinate(w6) at (barycentric cs:V=1,f6=-2.5);

% Draw the edges
\draw[edge=red] (f1) --(f2);
\draw[edge=green] (f3) --(f2);
\draw[edge=blue] (f4) --(f3);
\draw[edge=red] (f5) --(f4);
\draw[edge=green] (f6) --(f5);
\draw[edge=blue] (f1) --(f6);
\draw[edge=green] (f1) --(w1);
\draw[edge=blue] (f2) --(w2);
\draw[edge=red] (f3) --(w3);
\draw[edge=green] (f4) --(w4);
\draw[edge=blue] (f5) --(w5);
\draw[edge=blue] (f6) --(w6);

% Draw the vertices
\vertexLabelR{f1}{left}{$ $}
\vertexLabelR{f2}{left}{$ $}
\vertexLabelR{f3}{left}{$ $}
\vertexLabelR{f4}{left}{$ $}
\vertexLabelR{f5}{left}{$ $}
\vertexLabelR{f6}{left}{$ $}

\end{tikzpicture}}
  \caption{}
  \label{fig:fgvf11d}
\end{subfigure}
\caption{(a) A tri-coloured vertex-defining umbrella of a simplicial surface $X$ with $\vf(X)=(1,1)$ (b) corresponding arc-coloured subgraph in the face graph $\F(X)$}
  \label{fig:11}
\end{figure}
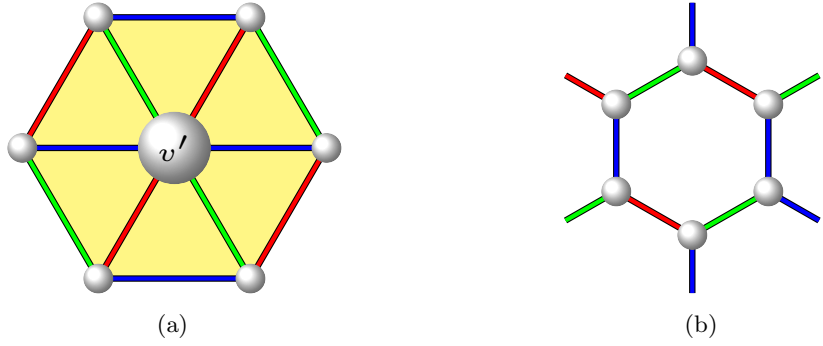
As before, these $(1,2,3)$-umbrellas are translated into $(1,2,3)$-cycles of the underlying face graph $\F(X)$, see \Cref{fig:fgvf11d}. This results in $H$ being a $(1,1)$-group of type $1$.

\medskip
\paragraph*{\bf Case 2:} 
First, we assume that the $\Aut(X)$-orbits on $X_1$ are given by $\mathcal{E}_1,\E_2$ with $\vert \E_1\vert =\vert X_2 \vert$ and 
$\vert \E_2\vert =\tfrac{1}{2}\vert X_2 \vert $.
Hence, the map $\omega:X_1\to \{1,2\} $ defined by $\omega(e):=i$ forms an edge-colouring of $X$ leading to an arc-colouring $\kappa$ of 
$\F(X)$ via $\kappa(\{F,F'\}):=\omega(e)$ for $X_2(e)=\{F,F'\}$. 
This enables us to determine the umbrellas of $X$ as follows: Let $v \in X_0$ be a vertex with corresponding umbrella $u(v)=(F_1,\ldots,F_n)$. As described in the proof of Theorem \ref{thm:12}, the umbrellas of the vertices in $X_0$ form $(1,1,2)$-cycles with respect to $\omega$.
Since ${\Aut(X)}$ acts transitively on the vertices of $X$, all the umbrellas of $X$ are $(1,1,2)$-umbrellas and must be of the same vertex-degree divisible by $3$. If the vertex-degree of any vertex of $X$ is $3$, then $X$ must be isomorphic to the simplicial tetrahedron which results in $\vf(X)=(1,6)\neq (1,1)$. Hence, all vertex-degrees of $X$ are at least $6$.

Next, we construct automorphisms of $X$ to show that the umbrella $u(v)$ is induced by an automorphism.
Therefore, for $i=1,\ldots,n$ let $F_i$ be a face with $X_0(F_i)=\{v,v_i,v_{i+1}\}$ and $e_i\in X_i$ an edge with $X_0(e_i)=\{v,v_i\}$ (subscripts read modulo $n$). Without loss of generality, we assume $e_1\in \E_1$ and $e_2\in \E_2$.  
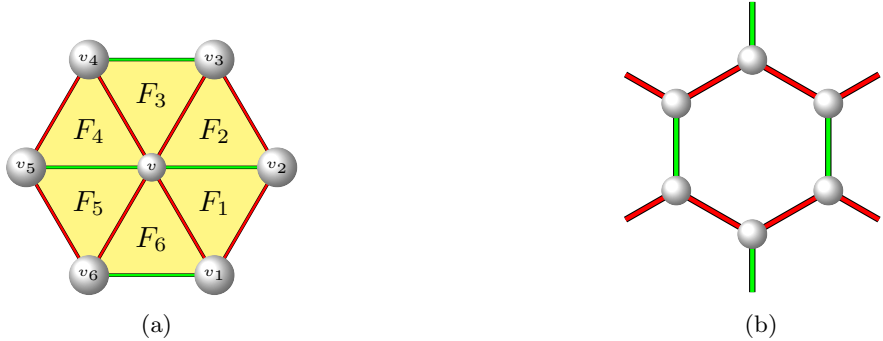
\begin{figure}[H]
  \centering
\begin{subfigure}{.45\textwidth}
  \centering\scalebox{1.1}{\begin{tikzpicture}[vertexBall, edgeDouble, faceStyle, scale=1.5]

% Define the coordinates of the vertices
\coordinate (V) at (0., 0.);
\coordinate (V1) at (1., 0.);
\coordinate (V2) at (0.4999999999999999, 0.8660254037844386);
\coordinate (V3) at (-0.5000000000000001, 0.8660254037844386);
\coordinate (V4) at (-1, 0.);
\coordinate (V5) at (-0.5, -0.8660);
\coordinate (V6) at (0.5, -0.8660);

% Fill in the faces
\fill[face]  (V) -- (V1) -- (V2) -- cycle;
\fill[face]  (V) -- (V3) -- (V2) -- cycle;
\fill[face]  (V) -- (V3) -- (V4) -- cycle;
\fill[face]  (V) -- (V4) -- (V5) -- cycle;
\fill[face]  (V) -- (V5) -- (V6) -- cycle;
\fill[face]  (V) -- (V1) -- (V6) -- cycle;
\fill[face]  (V2) -- (V3) -- (V5) -- cycle;
\node[faceLabel] at (barycentric cs:V=1,V1=1,V2=1) {$F_2$};
\node[faceLabel] at (barycentric cs:V=1,V2=1,V3=1) {$F_3$};
\node[faceLabel] at (barycentric cs:V=1,V3=1,V4=1) {$F_4$};
\node[faceLabel] at (barycentric cs:V=1,V4=1,V5=1) {$F_5$};
\node[faceLabel] at (barycentric cs:V=1,V5=1,V6=1) {$F_6$};
\node[faceLabel] at (barycentric cs:V=1,V6=1,V1=1) {$F_1$};

% Draw the edges
\draw[edge=green] (V) --(V1);
\draw[edge=red] (V) --(V2);
\draw[edge=red] (V) --(V3);
\draw[edge=green] (V) --(V4);
\draw[edge=red] (V) --(V5);
\draw[edge=red] (V) --(V6);

\draw[edge=red] (V1) --(V2);
\draw[edge=green] (V3) --(V2);
\draw[edge=red] (V3) --(V4);
\draw[edge=red] (V4) --(V5);
\draw[edge=green] (V5) --(V6);
\draw[edge=red] (V1) --(V6);
% Draw the vertices
\vertexLabelR{V1}{left}{$v_2$}
\vertexLabelR{V2}{left}{$v_3$}
\vertexLabelR{V3}{left}{$v_4 $}
\vertexLabelR{V4}{left}{$v_5 $}
\vertexLabelR{V5}{left}{$v_6$}
\vertexLabelR{V6}{left}{$v_1$}
\vertexLabelR{V}{left}{$v$}

\end{tikzpicture}}
  \caption{}
  \label{fig:vf11}
\end{subfigure}
\hspace{2cm}
\begin{subfigure}{.3\textwidth}
 \centering \scalebox{2}{\begin{tikzpicture}[vertexBall, edgeDouble, faceStyle, scale=1]

% Define the coordinates of the vertices
\coordinate (V) at (0., 0.);
\coordinate (V1) at (1., 0.);
\coordinate (V2) at (0.4999999999999999, 0.8660254037844386);
\coordinate (V3) at (-0.5000000000000001, 0.8660254037844386);
\coordinate (V4) at (-1, 0.);
\coordinate (V5) at (-0.5, -0.8660);
\coordinate (V6) at (0.5, -0.8660);

\coordinate (f1) at (barycentric cs:V=1,V1=1,V2=1);
\coordinate (f2) at (barycentric cs:V=1,V2=1,V3=1);
\coordinate (f3) at (barycentric cs:V=1,V3=1,V4=1);
\coordinate (f4) at (barycentric cs:V=1,V4=1,V5=1);
\coordinate (f5) at (barycentric cs:V=1,V5=1,V6=1);
\coordinate (f6) at (barycentric cs:V=1,V1=1,V6=1);

\coordinate(w1) at (barycentric cs:V=1,f1=-2.5);
\coordinate(w2) at (barycentric cs:V=1,f2=-2.5);
\coordinate(w3) at (barycentric cs:V=1,f3=-2.5);
\coordinate(w4) at (barycentric cs:V=1,f4=-2.5);
\coordinate(w5) at (barycentric cs:V=1,f5=-2.5);
\coordinate(w6) at (barycentric cs:V=1,f6=-2.5);

% Draw the edges
\draw[edge=red] (f1) --(f2);
\draw[edge=red] (f3) --(f2);
\draw[edge=green] (f4) --(f3);
\draw[edge=red] (f5) --(f4);
\draw[edge=red] (f6) --(f5);
\draw[edge=green] (f1) --(f6);
\draw[edge=red] (f1) --(w1);
\draw[edge=green] (f2) --(w2);
\draw[edge=red] (f3) --(w3);
\draw[edge=red] (f4) --(w4);
\draw[edge=green] (f5) --(w5);
\draw[edge=red] (f6) --(w6);

% Draw the vertices
\vertexLabelR{f1}{left}{$ $}
\vertexLabelR{f2}{left}{$ $}
\vertexLabelR{f3}{left}{$ $}
\vertexLabelR{f4}{left}{$ $}
\vertexLabelR{f5}{left}{$ $}
\vertexLabelR{f6}{left}{$ $}

\end{tikzpicture}}
  \caption{}
  \label{fig:fgvf11}
\end{subfigure}
\caption{The edge-coloured umbrella $u(v)$ of the simplicial surface $X$ (a) and the corresponding arc-coloured subgraph in 
$\F(X)$ (b)}
\end{figure}

Let $\phi_1,\phi_2\in \Aut(X)$ be the unique automorphism that satisfy $\phi_1(F_1)=F_4$ and $\phi_2(F_4)=F_7,$ where we read the subscripts modulo $\deg(v).$ Next, we define the automorphisms $\sigma_1:=\lambda_X(\phi_1)$ and $\sigma_2:=\lambda_X(\phi_2)$. This results in two cases, namely 
\begin{itemize}
    \item[(a)] $(F_1,\ldots,F_n)=\alpha(\sigma_1,F_1,\ldots,F_4),$ and
    \item[(b)]$(F_1,\ldots,F_n)\neq \alpha(\sigma,F_1,\ldots,F_4).$
\end{itemize}

\noindent
\textbf{Case 2(a):} 
In this case, we conclude $\sigma_1=\sigma_2$ and $\alpha(\sigma_1,F_1,\ldots,F_4)=\alpha({\sigma_1}^2,F_1,\ldots,F_7)$. Since the vertex-defining cycles of $X$ are given by ${\alpha(\sigma_1,F_1,\ldots,F_7)}^H$, we obtain that $H$ is a $(1,1)$-group of type $3$.

\noindent
\textbf{Case 2(b):} 
Note that in this case the automorphisms $\phi_1$ and $\phi_2$ both do not stabilise the vertex $v.$ However, we can make use of these automorphisms to construct an automorphism that does. By defining $\phi':=\phi_2\circ \phi_1$, we observe that $\phi'(F_i)=F_{i+6}$, for all $1\leq i \leq n$.
Hence, we define $\gamma:=\lambda_X(\phi')$ and obtain the vertex-defining cycle of $v$ via $(F_1,\ldots,F_n)=\alpha(\gamma,F_1,\ldots,F_7)$. Since there exists no $\pi \in \Aut(\F(X))$ with $\alpha(\pi,F_1,\ldots,F_4)=\alpha(\sigma,F_1,\ldots,F_7)$, the subgroup $H$ is a $(1,1)$-group of type~$4$.

\medskip
To prove the ``only-if'' part of the theorem, assume that $H$ is a $(1,1)$-group of $\F(X)$, with cycle double cover $\C$. Let $F\in X_2$ be an arbitrary face. Since ${\Aut(X)}\cong H$, $X$ is face-transitive, and the stabiliser $\stab_H(F)$ of the node $F$ is isomorphic to $\stab_{\Aut(X)}(F)$ of the face $F$, and hence trivial. Moreover, if follows from \Cref{lemma:vertexisoumbrella} that ${\Aut(X)}$ acts transitively on the vertices of $X$. 
This means that $X$ is a face-transitive surface satisfying $\vf(X)=(1,1)$.
\end{proof}

The above theorem allows us to formulate the following corollary.

\begin{corollary}\label{corollary:construction_vf11}
    Let $\G$ be a cubic node-transitive graph and $H\leq \Aut(\G)$ a subgroup. If $H$ is a $(1,1)$-group of $\G$, then $\G$ is the face graph of a face-transitive surface.
\end{corollary}
In Sections~\ref{example11} to \ref{example114} we provide examples for face-transitive surfaces $X$ with $\vf(X)=(1,1)$ of any type. 
We present the surface $X^{(1,1)}$ given in \Cref{example11} as an example of a face-transitive surface of vertex-face type $(1,1)$ of type 1.
The surface is a triangulated torus with $9$ vertices, $27$ edges, and $18$ triangles. Note that $\Aut(X^{(1,1)})\cong \Aut(\F(X^{(1,1)}))\cong D_{18}$.
With $Y^{(1,1)}$ (\Cref{example112}) we obtain an example of a face-transitive surface of vertex-face type $(1,1)$ of type $2$ satisfying  
$\Aut(Y^{(1,1)})\cong \Aut(Y^{(1,1)})\cong S_3\times S_3 \rtimes C_2.$ Further, the simplicial surface $\overline{X}^{(1,1)}$ is a face-transitive surface of vertex-face type $3$ consisting of $18$ vertices, $108$ edges and $72$ faces, see \Cref{example113}. Note that $\Aut(\overline{X}^{(1,1)})\cong \Aut(\F(\overline{X}^{(1,1)}))\cong(C_3\times C_3) \rtimes C_8.$ Finally, the simplicial surface $\overline{Y}^{(1,1)}$ is an example of a face-transitive surface with $\vf(\overline{Y}^{(1,1)})=(1,1)$ of type $4$, see \Cref{example114}. We observe that this surface consists of $112$ faces and that the automorphism groups of the surface and its face graph satisfy $\Aut(\overline{Y}^{(1,1)}) \cong \Aut(\F(\overline{Y}^{(1,1)}))\cong C_2 \times ((C_2 \times C_2 \times C_2) \rtimes C_7).$
\subsection{Face-transitive surfaces with vertex-face type (2,1)}\label{vf21}

In this section we discuss the structure of face-transitive surfaces with vertex-face 
type $(2,1)$.  
We show that the vertex-defining umbrellas of these surfaces are either mono- and bi-coloured umbrellas that are induced by automorphisms, or correspond to tri-coloured cycles of the 
face graph $\mathcal{F}(X)$.

\begin{definition}
Let $\G=(V,E)$ be a cubic node-transitive graph and $H\leq \Aut(\G)$ such that
  \begin{enumerate}
    \item $H$ acts transitively on the nodes $V$ with $\vert  H\vert =1\cdot\vert V \vert,$
    \item $E$ partitions into two $H$-orbits $E_1$ and $E_2$ with 
    $\vert E_1\vert =\vert V \vert $ and $\vert E_2\vert =\tfrac{1}{2}\vert V \vert,$ and
    \item there exist nodes $F_1,\ldots,F_5\in V $  with $\{F_2,F_3\},\{F_4,F_5\}\in E_1$ and 
    $\{F_1,F_2\},\{F_3,F_4\}\in E_2,$ and an automorphism 
    $\sigma\in \Aut (\G)$ such that $\sigma(F_1)=F_5$.
  \end{enumerate}
  The partition $E = E_1 \cup E_2$ yields a colouring $\kappa:E\to \{1,2\}$ of $\G$ by 
  $\kappa(\{F,F'\}):=i$ for $\{F,F'\}\in E_i$.  Let $u$ be a mono-coloured cycle in $\G$ 
  with respect to $\kappa$ and $\C:= u^H \cup {\alpha(\sigma,F_1,F_2,F_3,F_4,F_5)}^H$ be a vertex-faithful cycle double cover of $\G$.
  We say that 
  $H$ is a \emph{\bf$(2,1)$-group of $\G$ of type $1$} if there exists an automorphism $\pi\in \Aut(\G)$ with $\alpha(\pi,F_1,F_2,F_3)=\alpha(\sigma,F_1,F_2,F_3,F_4,F_5)$. We denote the above cycle double 
  cover by $\C_1^{(2,1)}(H)$. If there exists no such automorphism $\pi$, then we say that 
  $H$ is a \emph{\bf$(2,1)$-group of $\G$ of type $2$} and denote the above cycle double 
  cover by $\C_2^{(2,1)}(H)$.
\end{definition}

\begin{definition}
  Let $\G=(V,E)$ be a cubic node-transitive graph and $H\leq \Aut(\G)$ such that
  \begin{enumerate}
    \item $H$ acts transitively on the nodes $V$ with $\vert  H\vert =1\cdot\vert V\vert,$ and
    \item $E$ partitions into three $H$-orbits $E_1$, $E_2$, and $E_3$, all of cardinality $\frac{1}{2}\vert V \vert$.
  \end{enumerate}
  The partition $E=E_1 \cup E_2 \cup E_3$ defines the arc-colouring 
  $\kappa:E\to \{1,2,3\},$ by $\kappa(\{F,F'\}):=i$ for $\{F,F'\}\in E_i$. 
  If the (unique) set containing all the $(1,2)$- and $(3,1,3,2)$-cycles with respect to $\kappa$ is a 
  vertex-faithful cycle double cover of $\G$, then we say that $H$ is a 
  \emph{\bf$(2,1)$-group of $\G$ of type $3$,} with the cycle double cover denoted by 
  $\C_3^{(2,1)}(H)$.
\end{definition}
We call a subgroup $H$ of the automorphism group of a node-transitive cubic graph a $(2,1)$-group, if $H$ is a $(2,1)$ group of any type.  

\begin{theorem}\label{thm:21}
 Let $X$ be a simplicial surface. Then $X$ is face-transitive with vertex-face type $\vf(X)=(2,1)$ if and only 
 if $H:=\lambda_X(\Aut(X))$ is a $(2,1)$-group of $\F(X)$ and the cycle double cover $\C^{(2,1)}(H)$ contains exactly the vertex-defining cycles of $X$.
\end{theorem}

\begin{proof}
Let $X$ be a face-transitive surface of vertex-face type $(2,1)$ and hence with
$\vert \Aut(X) \vert =1 \cdot\vert X_2 \vert$. Furthermore, let $V_1$ and $V_2$ be the two 
$\Aut(X)$-orbits on the vertices of $X$. Without loss of generality, we assume that every face of 
$X$ is incident to exactly one vertex in $V_1$ and exactly two vertices in $V_2$. With the orbit stabiliser theorem we see that the number of orbits of $\Aut(X)$ on the edges of $X$ satisfies $\vert X_1^{\Aut(X)}\vert \in \{2,3\}$. We therefore conduct a case distinction to determine the structure of the umbrellas of $X$ with respect to the number of orbits on the edges. 
\paragraph*{Case 1:} 
First, we examine the case $\vert X_1^{\Aut(X)}\vert =2$. The vertex 
orbits $V_1,V_2$ induce a partition of the edges of $X$ by 
\begin{align*}
  \E_1:=\{\{v,v'\} \in X_1 \mid v \in V_1\text{ and }v'\in V_2\},\,
  \E_2:=\{\{v,v'\} \in X_1 \mid v\in V_2 \text{ and }v' \in V_2\}.
\end{align*}
We know that $\vert \E_1\vert =\vert X_2 \vert$ and 
$\vert \E_2\vert =\tfrac{1}{2}\vert X_2 \vert $.
Since $\vert X_1^{\Aut(X)}\vert =2$, these partition sets are $\Aut(X)$-orbits.
The map $\omega:X_1\to \{1,2\} $ defined by $\omega(e):=i$ for $e \in\E_i$ forms an edge-colouring of $X$. This colouring can be translated into an arc-colouring $\kappa$ of 
$\F(X)$ via $\kappa(\{F,F'\}):=\omega(e)$ for $X_2(e)=\{F,F'\}$. 

Let $v_i \in V_i$, 
$i=1,2$, with corresponding umbrellas $u(v_1)=(F_1,\ldots,F_n)$ and 
$u(v_2)=(F_1',\ldots,F_m')$. The umbrellas of the vertices in $v_1^{\Aut(X)}$ are 
mono-coloured and the umbrellas of the vertices in $v_2^{\Aut(X)}$ are bi-coloured 
umbrellas with respect to $\omega$. The mono-coloured cycles representing the vertices 
$v_1^{\Aut(X)}$ are directly obtained from deleting the edges $\{X_2(e)\mid e\in \E_2\}$ 
from $\F(X)$ as also described in the proof of \Cref{theorem22}. 
Hence, the set of these cycles describes $u(v_1)^{\Aut(X)}$.

For the cycles corresponding to the umbrellas in 
${v_2}^{\Aut(X)}$, note that $u(v_2)$ is bi-coloured, and hence $m$ is even. Thus, we define $k=\frac{m}{2}$. Moreover, let the edges $e_1,\ldots,e_{m},$ be determined by 
$X_2(e_1)=\{F_{1}',F_{2}'\},\ldots,X_2(e_m)=\{F_{m}',F_{1}'\},$ respectively. Without loss
of generality, we assume that $e_{2i} \in \E_2$, $i=1,\ldots,k,$ see \Cref{fig:extract21} 
for an illustration.

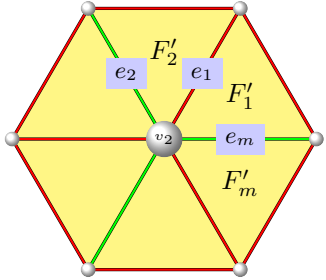
\begin{figure}[H]
 \centering
\scalebox{1}{\begin{tikzpicture}[vertexBall, edgeDouble, faceStyle, scale=2.]

% Define the coordinates of the vertices
\coordinate (V) at (0., 0.);
\coordinate (V1) at (1., 0.);
\coordinate (V2) at (0.4999999999999999, 0.8660254037844386);
\coordinate (V3) at (-0.5000000000000001, 0.8660254037844386);
\coordinate (V4) at (-1, 0.);
\coordinate (V5) at (-0.5, -0.8660);
\coordinate (V6) at (0.5, -0.8660);
\coordinate (V7) at (-1.5, 0.8660);
\coordinate (V8) at (1.5, 0.8660);

% Fill in the faces
\fill[face]  (V) -- (V1) -- (V2) -- cycle;
\fill[face]  (V) -- (V3) -- (V2) -- cycle;
\fill[face]  (V) -- (V3) -- (V4) -- cycle;
\fill[face]  (V) -- (V4) -- (V5) -- cycle;
\fill[face]  (V) -- (V5) -- (V6) -- cycle;
\fill[face]  (V) -- (V1) -- (V6) -- cycle;
\fill[face]  (V2) -- (V3) -- (V5) -- cycle;
%\fill[face]  (V4) -- (V3) -- (V7) -- cycle;
%\fill[face]  (V2) -- (V1) -- (V8) -- cycle;
%\node[faceLabel] at (barycentric cs:V2=1,V3=1,V5=1) {$F_3$};
\node[faceLabel] at (barycentric cs:V=1,V1=1,V2=1) {$F_1'$};
\node[faceLabel] at (barycentric cs:V=1,V1=1,V6=1) {$F_m'$};
\node[faceLabel] at (barycentric cs:V=1,V2=1,V3=1) {$F_2'$};
%\node[faceLabel] at (barycentric cs:V1=1,V2=1,V8=1) {$F'$};

% Draw the edges
\draw[edge=green] (V) --  node[edgeLabel] {$e_m$}  (V1);
\draw[edge=red] (V) -- node[edgeLabel] {$e_1$}(V2);
\draw[edge=green] (V) -- node[edgeLabel] {$e_2$}(V3);
\draw[edge=red] (V) -- (V4);
\draw[edge=green] (V) -- (V5);
\draw[edge=red] (V) --(V6);
\draw[edge=red] (V1) --(V2);
\draw[edge=red] (V3) --(V2);
\draw[edge=red] (V3) --(V4);
\draw[edge=red] (V4) --(V5);
\draw[edge=red] (V5) --(V6);
\draw[edge=red] (V1) --(V6);
%\draw[edge=green] (V3) --(V7);
%\draw[edge=red] (V4) --(V7);
%\draw[edge=green] (V8) --(V1);
%\draw[edge=red] (V8) --(V2);
% Draw the vertices
\vertexLabelR{V1}{left}{$ $}
\vertexLabelR{V2}{left}{$ $}
\vertexLabelR{V3}{left}{$ $}
\vertexLabelR{V4}{left}{$ $}
\vertexLabelR{V5}{left}{$ $}
\vertexLabelR{V6}{left}{$ $}
\vertexLabelR{V}{left}{$v_2$}
%\vertexLabelR{V7}{left}{$ $}
%\vertexLabelR{V8}{left}{$ $}
\end{tikzpicture}}
  \caption{A bi-coloured vertex-defining umbrella of a simplicial surface $X$ with 
    $\vf(X)=(2,1)$}\label{fig:extract21}
\end{figure} 
Due to \Cref{lemma:conjugate}, the edge-stabilisers of 
$e_2,e_4,\ldots,e_m$ are all conjugate in $\Aut(X)$ and have order two. Let $\phi_{2i}$, $1\leq i\leq k$, be the
non-trivial automorphism in $\stab_{\Aut(X)}(e_{2i})$. We have the following two cases: 
\begin{itemize}
    \item[(a)] $\phi_{2i}(v_2)=v_2$ for all $1\leq i\leq k$, that is, these automorphisms stabilise the vertex $v_2$; or
    \item[(b)] $\phi_{2i}(v_2)\neq v_2$ for all $1\leq i\leq k$, that is, the vertices of the corresponding edges are interchanged.
\end{itemize}
\paragraph*{Case (1a):} If all $\phi_2,\phi_4,\ldots,\phi_m$ stabilise $v_2,$ then 
the automorphism $\phi:=\phi_2\circ\phi_m$ of order $\ell$ satisfies 
$\phi(F_i')=F_{i+4}'$. Hence, the umbrella of $v_2$ can be written as 
\begin{align*}
&(F_1',\ldots,F_m')=(\phi(F_1'),\phi(F_2'),\phi(F_3'),\phi(F_4'),\ldots,\phi^\ell(F_1'),\phi^\ell(F_2'),\phi^\ell(F_3'),\phi^\ell(F_4')).
\end{align*}
By defining $\sigma:=\lambda_X(\phi)$ we see that $u(v_2)=\alpha(\sigma,F_1',\ldots,F_5')$ is a 
$\sigma$-induced $\alpha$-cycle in $\F(X)$ which implies that the cycle double cover 
containing all the vertex-defining cycles of $X$ is given by 
$\C=(F_1,\ldots,F_n)^H \cup {\alpha(\sigma,F_1',\ldots,F_5')}^H$. If there exists an automorphism $\pi\in \Aut(\F(X))$ with $\alpha(\pi,F_1',F_2',F_3')=(F_1,\ldots,F_n)$ then without loss of generality $\pi$ maps $F_1'$ onto $F_3'.$ Thus, the automorphism $\sigma':=\pi\circ \lambda_X({\phi_m})$
 is a non-trivial element in the stabiliser of $e_1$, a contradiction to $\vert \stab_{\Aut(X)}(e_1)\vert=1$.

\paragraph*{Case (1b):} Assume that the automorphisms $\phi_2,\phi_4,\ldots,\phi_m$ 
do not stabilise the vertex $v_2$. Let $\psi\in \Aut(X)$ be the automorphism 
that maps $F_1'$ onto $F_2'$. 
Since $|\E_1|=|X_2|$, $\stab_{\Aut(X)} (e_1)$ is trivial, and the automorphism $\psi$ cannot stabilise $e_1$. 
Hence, we know that $\psi(v_2)\neq v_2$ and 
$\psi(e_m)=e_2$ hold. Let $\phi:=\phi_2\circ \psi\in \Aut(X)$. Then $\phi(v_2)=v_2$ and hence 
$\phi \in \stab_{\Aut(X)}(v_2)$. Further, we observe that 
$\phi(F_1')=\phi_2(\psi(F_1'))=\phi_2(F_2')=F_3'$ and 
$\phi(F_m')=\phi_2(\psi(F_m'))=\phi_2(F_3')=F_2'$.
Thus, we conclude $\phi(F_i')=F_{i+2}'$. This instance results in
%\begin{align*}
$(F_1',F_2',\ldots,F_m')=(\phi(F_1'),\phi(F_2'),\ldots,\phi^\ell(F_1'),\phi^\ell(F_2'))$, where $\ell$ denotes the order of $\phi$.
%\end{align*}
By defining $\sigma:=\lambda_X(\phi)$ we see that $u(v_2)=\alpha(\sigma,F_1',F_2',F_3')$ is a 
$\sigma$-induced $\alpha$-cycle in $\F(X)$. Thus, we obtain the cycle double cover 
containing all the vertex-defining cycles of $X$ as 
$\C=(F_1,\ldots,F_n)^H \cup {\alpha(\sigma,F_1',F_2',F_3')}^H$. Together with $\alpha(\sigma,F_1',F_2',F_3')=\alpha(\sigma^2,F_1',F_2',F_3',F_4',F_5')$ it follows that $H$ is a $(2,1)$-group of $\F(X)$ of type $2$ satisfying $\C=\C^{(2,1)}_2(H)$.

\paragraph{Case 2}
Assume that the $\Aut(X)$-orbits of $X_1$ are given by $\mathcal{E}'_1,\mathcal{E}'_2,\mathcal{E}'_3$ with $\vert \mathcal{E}'_i\vert=\tfrac{1}{2}\vert X_2 \vert$ for $i=1,2,3$. We define an edge-colouring $\omega'$ of $X$, given by $\omega(e):=i$ for $e\in \E_i'$, $i=1,2,3$, which can be translated into a corresponding arc-colouring $\kappa$ of $\F(X)$. We obtain a partition of the edges of $X$ via  
\begin{align*}
  \mathcal{P}_1:=\{\{v,v'\} \in X_1 \mid v \in V_1\text{ and }v'\in V_2\},\,
  \mathcal{P}_2:=\{\{v,v'\} \in X_1 \mid v\in V_2 \text{ and }v' \in V_2\}.
\end{align*}
Since $\vert \mathcal{P}_2\vert =\tfrac{1}{2}\vert X_2\vert$, $\mathcal{P}_2$ must be an $\Aut(X)$-orbit, without loss of generality, $\mathcal{P}_2=\E_3$. This implies $\mathcal{P}_1=\E_1\cup \E_2$ and the umbrellas of vertices in $V_1$ must be bi-coloured. This implies that  the edges in $\mathcal{E}_1$ and $\mathcal{E}_2$ are all mirror edges with respect to $\omega$, and the following cases remain to be checked.
\begin{enumerate}
  \item[(a)] all the edges of $X$ are mirror edges,
  \item[(b)] the edges contained in $\E_1\cup \E_2$ are mirror edges and the edges in $\E_3$ are rotational edges.
\end{enumerate}

\paragraph*{Case (2a):}If all edges are mirror edges, the automorphism group $\Aut(X)$ must have three vertex orbits (as described in the proof of \cref{thm:31}), a contradiction to $\vf(X) = (2,1)$. 

\paragraph*{Case (2b):}
In this case, observe that all vertex-defining umbrellas in $v_1^{\Aut(X)}$ are $(1,2,\ldots,1,2)$-coloured and all vertex-defining umbrellas in $v_2^{\Aut(X)}$ are $(3,1,3,2)$-coloured.
These umbrellas can directly be translated into $(1,2)$- and $(3,1,3,2)$-coloured cycles of the face graph $\F(X)$ with respect to $\kappa$. Hence, we obtain a vertex-faithful cycle double cover of $\F(X)$ consisting of only $(1,2)$- and $(3,1,3,2)$-coloured cycles and $H$ is a $(1,1)$-group of $\F(X)$ of type $3$.

Next, we assume that $H$ forms a $(2,1)$-group of $\F(X)$. The arc-colouring
of $\F(X)$ defines an edge-colouring $\omega$ of $X$ by $\omega(e):=\kappa(\{F,F'\})$, 
where $e\in X_1$ and $X_2(e)=\{F,F'\}$. Since $H = \lambda_X(\Aut(X))$, and $H$ acts 
transitively on the nodes of $\F(X)$ with $\vert H \vert = 1\cdot \vert V \vert$, we have 
that $X$ is face-transitive with trivial face stabiliser.
Due to \Cref{lemma:vertexisoumbrella}, two vertices $v_1$ and $v_2$ of $X$ lie
in the same $\Aut(X)$-orbit if and only if the corresponding cycles 
$C_1,C_2 \in \C_1^{(2,1)}(H)$ in $\F(X)$ are contained in the same $H$-orbit. This is 
equivalent to saying that the corresponding vertex-defining cycles of $X$ have the same 
colour with respect to $\kappa$. Thus, there exist exactly two ${\Aut(X)}$-orbits on the vertices, and we have $\vf(X)=(2,1)$.
\end{proof}
The below corollary is a direct consequence of the above result.
\begin{corollary}\label{corollary:construction_vf21}
    Let $\G$ be a cubic node-transitive graph and $H\leq \Aut(\G)$ a subgroup. If $H$ is a $(2,1)$-group of $\G$, then $\G$ is the face graph of a face-transitive surface.
\end{corollary}
The simplicial surface $X^{(2,1)}$, see \Cref{example211}, consisting of $10$ vertices, $30$ edges and $20$ faces is an example of a face-transitive torus of vertex-face type $(2,1)$ of type $1$. This surface satisfies $\Aut(X^{(2,1)})\cong \Aut(\F(X^{(2,1)}))\cong C_5\rtimes C_4.$
Furthermore, the non-orientable simplicial surface $Y^{(2,1)}$ (see \Cref{example212}) is an example of a face-transitive surface with vertex-face type $(2,1)$ of type $2$.
This simplicial surface has $22$ vertices, $84$ edges and $56$ faces. With GAP we are able to verify that $Y^{(2,1)}$ satisfies $\Aut(Y^{(2,1)})\cong \Aut(\F(Y^{(2,1)}))\cong (C_2\times C_2\times C_2)\rtimes C_7$.

Finally, we give an example of a face-transitive surface of vertex-face type $(2,1)$ of type $3,$ namely the simplicial surface $Z^{(2,1)},$ see \Cref{example213}.
With $18$ vertices, $72$ edges and $48$ faces it forms a non-orientable simplicial surface such that $\Aut(Z^{(2,1)})\cong C_2\times S_4$ and $\Aut(\F(Z^{(2,1)}))\cong C_2\times C_2\times S_4$.

\section{Graphs with large automorphism group}
\label{section:problems}

Enumerating vertex-faithful double cycle covers, as described in \Cref{section:construction}, requires the identification of relatively small order subgroups of what can be excessively large order automorphism groups of node-transitive cubic graphs: there are $480$ graphs in the census of node-transitive cubic graphs from \cite{cubicvertextransitive} with an automorphism group of order at least $10^{11}$. The smallest one of them has $124$ nodes. For some of these graphs, computations are infeasible. This section presents theoretical obstructions for a large family of cubic graphs with large automorphism group to be face graphs of face-transitive surfaces. The graphs in our family have the following combinatorial properties:

\begin{itemize}
    \item They have $n=2^k \cdot m$ nodes, for some $k,m \geq 2$.
    \item Their arcs fall into two orbits of sizes $n$ and $\frac{n}{2}$.
    \item The large edge orbit forms a collection of pairwise disjoint $4$-cycles, while the short edge orbit forms a complete matching of the corresponding graph.
    \item The $4$-cycles can be grouped into $m$ collections of $4$-cycles, each of cardinality $2^{k-2}$, and with the $m$ groups arranged equidistantly along a circle, and the short edges of the short orbits exclusively connecting nodes from one collection of $4$-cycles to nodes from the previous and the following collections in the circular order.
    \item For every $4$-cycle, arcs from the short orbit connecting to the previous and the following connections alternate along the $4$-cycle.
\end{itemize}

\begin{figure}[H]
    \centering
    \begin{tikzpicture}[vertexBall, edgeDouble=nolabels, faceStyle, scale=1]

% Define the coordinates
\coordinate (V1) at (2.0*1.0000, 2.0*0.0000);
\coordinate (V2) at (2.0*0.9239, 2.0*0.3827);
\coordinate (V3) at (2.0*0.7071, 2.0*0.7071);
\coordinate (V4) at (2.0*0.3827, 2.0*0.9239);
\coordinate (V5) at (2.0*0.0000, 2.0*1.0000);
\coordinate (V6) at (2.0*-0.3827, 2.0*0.9239);
\coordinate (V7) at (2.0*-0.7071, 2.0*0.7071);
\coordinate (V8) at (2.0*-0.9239, 2.0*0.3827);
\coordinate (V9) at (2.0*-1.0000, 2.0*0.0000);
\coordinate (V10) at (2.0*-0.9239, 2.0*-0.3827);
\coordinate (V11) at (2.0*-0.7071, 2.0*-0.7071);
\coordinate (V12) at (2.0*-0.3827,2.0* -0.9239);
\coordinate (V13) at (2.0*0.0000, 2.0*-1.0000);
\coordinate (V14) at (2.0*0.3827, 2.0*-0.9239);
\coordinate (V15) at (2.0*0.7071, 2.0*-0.7071);
\coordinate (V16) at (2.0*0.9239,2.0* -0.3827);

% Define the second set of coordinates
\coordinate (VV1) at (3.0*1.0000, 3.0*0.0000);
\coordinate (VV2) at (3.0*0.9239, 3.0*0.3827);
\coordinate (VV3) at (3.0*0.7071, 3.0*0.7071);
\coordinate (VV4) at (3.0*0.3827, 3.0*0.9239);
\coordinate (VV5) at (3.0*0.0000, 3.0*1.0000);
\coordinate (VV6) at (3.0*-0.3827, 3.0*0.9239);
\coordinate (VV7) at (3.0*-0.7071, 3.0*0.7071);
\coordinate (VV8) at (3.0*-0.9239, 3.0*0.3827);
\coordinate (VV9) at (3.0*-1.0000, 3.0*0.0000);
\coordinate (VV10) at (3.0*-0.9239, 3.0*-0.3827);
\coordinate (VV11) at (3.0*-0.7071, 3.0*-0.7071);
\coordinate (VV12) at (3.0*-0.3827, 3.0*-0.9239);
\coordinate (VV13) at (3.0*0.0000, 3.0*-1.0000);
\coordinate (VV14) at (3.0*0.3827, 3.0*-0.9239);
\coordinate (VV15) at (3.0*0.7071, 3.0*-0.7071);
\coordinate (VV16) at (3.0*0.9239, 3.0*-0.3827);

\foreach \i/\j in {1/2,3/4,5/6,7/8,9/10,11/12,13/14,15/16} {
    \draw[red] (V\i) -- (V\j);
    \draw[red] (VV\i) -- (VV\j);
}

\foreach \i/\j in {1/4,3/6,5/8,7/10,9/12,11/14,13/16,14/1} {
    \draw (V\i) -- (V\j);

}

\foreach \i/\j in {2/3,4/5,6/7,8/9,10/11,12/13,14/15,16/1} {
    \draw (VV\i) -- (VV\j);

}

\foreach \i in {1,2,...,16} {
    \draw[red] (V\i) -- (VV\i);
}
 % Connect V16 to V1
% Label vertices for the first set
\vertexLabelR{V1}{left} {$ $}
\vertexLabelR{V2}{left} {$ $}
\vertexLabelR{V3}{left} {$ $}
\vertexLabelR{V4}{left} {$ $}
\vertexLabelR{V5}{left} {$ $}
\vertexLabelR{V6}{left} {$ $}
\vertexLabelR{V7}{left} {$ $}
\vertexLabelR{V8}{left} {$ $}
\vertexLabelR{V9}{left} {$ $}
\vertexLabelR{V10}{left} {$ $}
\vertexLabelR{V11}{left} {$ $}
\vertexLabelR{V12}{left} {$ $}
\vertexLabelR{V13}{left} {$ $}
\vertexLabelR{V14}{left} {$ $}
\vertexLabelR{V15}{left} {$ $}
\vertexLabelR{V16}{left} {$ $}

\vertexLabelR{VV1}{left} {$ $}
\vertexLabelR{VV2}{left} {$ $}
\vertexLabelR{VV3}{left} {$ $}
\vertexLabelR{VV4}{left} {$ $}
\vertexLabelR{VV5}{left} {$ $}
\vertexLabelR{VV6}{left} {$ $}
\vertexLabelR{VV7}{left} {$ $}
\vertexLabelR{VV8}{left} {$ $}
\vertexLabelR{VV9}{left} {$ $}
\vertexLabelR{VV10}{left} {$ $}
\vertexLabelR{VV11}{left} {$ $}
\vertexLabelR{VV12}{left} {$ $}
\vertexLabelR{VV13}{left} {$ $}
\vertexLabelR{VV14}{left} {$ $}
\vertexLabelR{VV15}{left} {$ $}
\vertexLabelR{VV16}{left} {$ $}

\end{tikzpicture} \quad 
    \begin{tikzpicture}[vertexBall, edgeDouble=nolabels, faceStyle, scale=0.54]

% Define the coordinates
\coordinate (V1) at (2.0*1.0000, 2.0*0.0000);
\coordinate (V2) at (2.0*0.9239, 2.0*0.3827);
\coordinate (V3) at (2.0*0.7071, 2.0*0.7071);
\coordinate (V4) at (2.0*0.3827, 2.0*0.9239);
\coordinate (V5) at (2.0*0.0000, 2.0*1.0000);
\coordinate (V6) at (2.0*-0.3827, 2.0*0.9239);
\coordinate (V7) at (2.0*-0.7071, 2.0*0.7071);
\coordinate (V8) at (2.0*-0.9239, 2.0*0.3827);
\coordinate (V9) at (2.0*-1.0000, 2.0*0.0000);
\coordinate (V10) at (2.0*-0.9239, 2.0*-0.3827);
\coordinate (V11) at (2.0*-0.7071, 2.0*-0.7071);
\coordinate (V12) at (2.0*-0.3827,2.0* -0.9239);
\coordinate (V13) at (2.0*0.0000, 2.0*-1.0000);
\coordinate (V14) at (2.0*0.3827, 2.0*-0.9239);
\coordinate (V15) at (2.0*0.7071, 2.0*-0.7071);
\coordinate (V16) at (2.0*0.9239,2.0* -0.3827);

% Define the second set of coordinates
\coordinate (VV1) at (3.0*1.0000, 3.0*0.0000);
\coordinate (VV2) at (3.0*0.9239, 3.0*0.3827);
\coordinate (VV3) at (3.0*0.7071, 3.0*0.7071);
\coordinate (VV4) at (3.0*0.3827, 3.0*0.9239);
\coordinate (VV5) at (3.0*0.0000, 3.0*1.0000);
\coordinate (VV6) at (3.0*-0.3827, 3.0*0.9239);
\coordinate (VV7) at (3.0*-0.7071, 3.0*0.7071);
\coordinate (VV8) at (3.0*-0.9239, 3.0*0.3827);
\coordinate (VV9) at (3.0*-1.0000, 3.0*0.0000);
\coordinate (VV10) at (3.0*-0.9239, 3.0*-0.3827);
\coordinate (VV11) at (3.0*-0.7071, 3.0*-0.7071);
\coordinate (VV12) at (3.0*-0.3827, 3.0*-0.9239);
\coordinate (VV13) at (3.0*0.0000, 3.0*-1.0000);
\coordinate (VV14) at (3.0*0.3827, 3.0*-0.9239);
\coordinate (VV15) at (3.0*0.7071, 3.0*-0.7071);
\coordinate (VV16) at (3.0*0.9239, 3.0*-0.3827);

\coordinate (VVV1) at (4.0*1.0000, 4.0*0.0000);
\coordinate (VVV2) at (4.0*0.9239, 4.0*0.3827);
\coordinate (VVV3) at (4.0*0.7071, 4.0*0.7071);
\coordinate (VVV4) at (4.0*0.3827, 4.0*0.9239);
\coordinate (VVV5) at (4.0*0.0000, 4.0*1.0000);
\coordinate (VVV6) at (4.0*-0.3827, 4.0*0.9239);
\coordinate (VVV7) at (4.0*-0.7071, 4.0*0.7071);
\coordinate (VVV8) at (4.0*-0.9239, 4.0*0.3827);
\coordinate (VVV9) at (4.0*-1.0000, 4.0*0.0000);
\coordinate (VVV10) at (4.0*-0.9239, 4.0*-0.3827);
\coordinate (VVV11) at (4.0*-0.7071, 4.0*-0.7071);
\coordinate (VVV12) at (4.0*-0.3827, 4.0*-0.9239);
\coordinate (VVV13) at (4.0*0.0000, 4.0*-1.0000);
\coordinate (VVV14) at (4.0*0.3827, 4.0*-0.9239);
\coordinate (VVV15) at (4.0*0.7071, 4.0*-0.7071);
\coordinate (VVV16) at (4.0*0.9239, 4.0*-0.3827);

\coordinate (VVVV1) at (5.0*1.0000, 5.0*0.0000);
\coordinate (VVVV2) at (5.0*0.9239, 5.0*0.3827);
\coordinate (VVVV3) at (5.0*0.7071, 5.0*0.7071);
\coordinate (VVVV4) at (5.0*0.3827, 5.0*0.9239);
\coordinate (VVVV5) at (5.0*0.0000, 5.0*1.0000);
\coordinate (VVVV6) at (5.0*-0.3827, 5.0*0.9239);
\coordinate (VVVV7) at (5.0*-0.7071, 5.0*0.7071);
\coordinate (VVVV8) at (5.0*-0.9239, 5.0*0.3827);
\coordinate (VVVV9) at (5.0*-1.0000, 5.0*0.0000);
\coordinate (VVVV10) at (5.0*-0.9239, 5.0*-0.3827);
\coordinate (VVVV11) at (5.0*-0.7071, 5.0*-0.7071);
\coordinate (VVVV12) at (5.0*-0.3827, 5.0*-0.9239);
\coordinate (VVVV13) at (5.0*0.0000, 5.0*-1.0000);
\coordinate (VVVV14) at (5.0*0.3827, 5.0*-0.9239);
\coordinate (VVVV15) at (5.0*0.7071, 5.0*-0.7071);
\coordinate (VVVV16) at (5.0*0.9239, 5.0*-0.3827);

\foreach \i/\j in {1/2,3/4,5/6,7/8,9/10,11/12,13/14,15/16} {
    \draw[red] (V\i) -- (V\j);
    \draw[red] (VV\i) -- (VV\j);
    \draw[red] (VVV\i) -- (VVV\j);
    \draw[red] (VVVV\i) -- (VVVV\j);
}

\foreach \i/\j in {1/4,3/6,5/8,7/10,9/12,11/14,13/16,15/2} {
    \draw[] (V\i) -- (V\j);

}

\foreach \i/\j in {2/3,4/5,6/7,8/9,10/11,12/13,14/15,16/1} {
    \draw[] (VV\i) -- (VVV\j);
    \draw[] (VV\j) -- (VVV\i);

}

\foreach \i/\j in {1/4,  3/6,  5/8, 7/10, 9/12,  11/14,  13/16,15/2} {
    \draw[-] (VVVV\i) to[bend right=55] (VVVV\j);
}

\foreach \i in {1,2,...,16} {
    \draw[red] (VVV\i) -- (VVVV\i);

        \draw[red] (V\i) -- (VV\i);
}
 % Connect V16 to V1
% Label vertices for the first set
\vertexLabelR{V1}{left} {$ $}
\vertexLabelR{V2}{left} {$ $}
\vertexLabelR{V3}{left} {$ $}
\vertexLabelR{V4}{left} {$ $}
\vertexLabelR{V5}{left} {$ $}
\vertexLabelR{V6}{left} {$ $}
\vertexLabelR{V7}{left} {$ $}
\vertexLabelR{V8}{left} {$ $}
\vertexLabelR{V9}{left} {$ $}
\vertexLabelR{V10}{left} {$ $}
\vertexLabelR{V11}{left} {$ $}
\vertexLabelR{V12}{left} {$ $}
\vertexLabelR{V13}{left} {$ $}
\vertexLabelR{V14}{left} {$ $}
\vertexLabelR{V15}{left} {$ $}
\vertexLabelR{V16}{left} {$ $}

\vertexLabelR{VV1}{left} {$ $}
\vertexLabelR{VV2}{left} {$ $}
\vertexLabelR{VV3}{left} {$ $}
\vertexLabelR{VV4}{left} {$ $}
\vertexLabelR{VV5}{left} {$ $}
\vertexLabelR{VV6}{left} {$ $}
\vertexLabelR{VV7}{left} {$ $}
\vertexLabelR{VV8}{left} {$ $}
\vertexLabelR{VV9}{left} {$ $}
\vertexLabelR{VV10}{left} {$ $}
\vertexLabelR{VV11}{left} {$ $}
\vertexLabelR{VV12}{left} {$ $}
\vertexLabelR{VV13}{left} {$ $}
\vertexLabelR{VV14}{left} {$ $}
\vertexLabelR{VV15}{left} {$ $}
\vertexLabelR{VV16}{left} {$ $}

\vertexLabelR{VVV1}{left} {$ $}
\vertexLabelR{VVV2}{left} {$ $}
\vertexLabelR{VVV3}{left} {$ $}
\vertexLabelR{VVV4}{left} {$ $}
\vertexLabelR{VVV5}{left} {$ $}
\vertexLabelR{VVV6}{left} {$ $}
\vertexLabelR{VVV7}{left} {$ $}
\vertexLabelR{VVV8}{left} {$ $}
\vertexLabelR{VVV9}{left} {$ $}
\vertexLabelR{VVV10}{left} {$ $}
\vertexLabelR{VVV11}{left} {$ $}
\vertexLabelR{VVV12}{left} {$ $}
\vertexLabelR{VVV13}{left} {$ $}
\vertexLabelR{VVV14}{left} {$ $}
\vertexLabelR{VVV15}{left} {$ $}
\vertexLabelR{VVV16}{left} {$ $}

\vertexLabelR{VVVV1}{left} {$ $}
\vertexLabelR{VVVV2}{left} {$ $}
\vertexLabelR{VVVV3}{left} {$ $}
\vertexLabelR{VVVV4}{left} {$ $}
\vertexLabelR{VVVV5}{left} {$ $}
\vertexLabelR{VVVV6}{left} {$ $}
\vertexLabelR{VVVV7}{left} {$ $}
\vertexLabelR{VVVV8}{left} {$ $}
\vertexLabelR{VVVV9}{left} {$ $}
\vertexLabelR{VVVV10}{left} {$ $}
\vertexLabelR{VVVV11}{left} {$ $}
\vertexLabelR{VVVV12}{left} {$ $}
\vertexLabelR{VVVV13}{left} {$ $}
\vertexLabelR{VVVV14}{left} {$ $}
\vertexLabelR{VVVV15}{left} {$ $}
\vertexLabelR{VVVV16}{left} {$ $}

\end{tikzpicture}
    \caption{The generalised $m$-gon graphs $\G_{8,2}$ and $\G_{8,3}$}
    \label{fig:genmgon}
\end{figure}
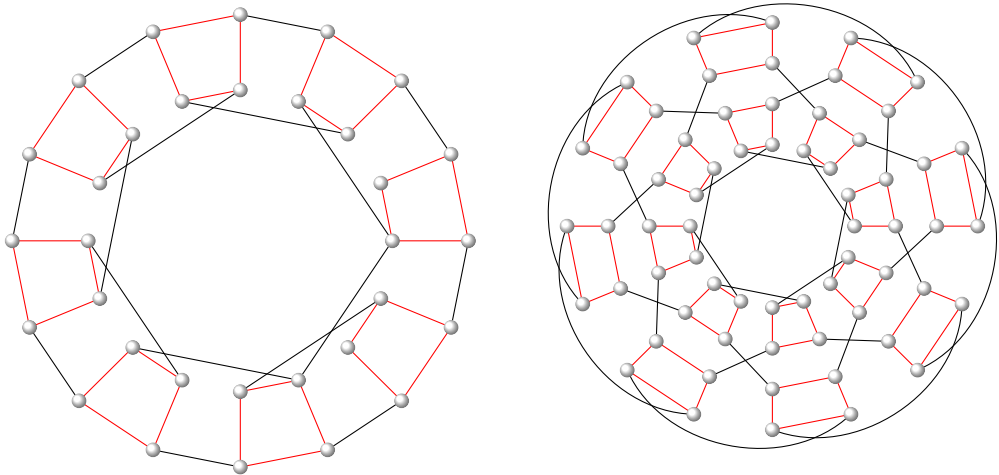
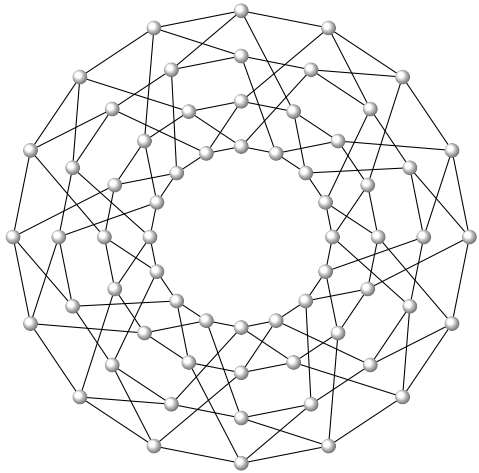
\begin{figure}[H]
    \centering
\begin{tikzpicture}[vertexBall, edgeDouble=nolabels, faceStyle, scale=0.6]

% Define the coordinates
\coordinate (V1) at (2.0*1.0000, 2.0*0.0000);
\coordinate (V2) at (2.0*0.9239, 2.0*0.3827);
\coordinate (V3) at (2.0*0.7071, 2.0*0.7071);
\coordinate (V4) at (2.0*0.3827, 2.0*0.9239);
\coordinate (V5) at (2.0*0.0000, 2.0*1.0000);
\coordinate (V6) at (2.0*-0.3827, 2.0*0.9239);
\coordinate (V7) at (2.0*-0.7071, 2.0*0.7071);
\coordinate (V8) at (2.0*-0.9239, 2.0*0.3827);
\coordinate (V9) at (2.0*-1.0000, 2.0*0.0000);
\coordinate (V10) at (2.0*-0.9239, 2.0*-0.3827);
\coordinate (V11) at (2.0*-0.7071, 2.0*-0.7071);
\coordinate (V12) at (2.0*-0.3827,2.0* -0.9239);
\coordinate (V13) at (2.0*0.0000, 2.0*-1.0000);
\coordinate (V14) at (2.0*0.3827, 2.0*-0.9239);
\coordinate (V15) at (2.0*0.7071, 2.0*-0.7071);
\coordinate (V16) at (2.0*0.9239,2.0* -0.3827);

% Define the second set of coordinates
\coordinate (VV1) at (3.0*1.0000, 3.0*0.0000);
\coordinate (VV2) at (3.0*0.9239, 3.0*0.3827);
\coordinate (VV3) at (3.0*0.7071, 3.0*0.7071);
\coordinate (VV4) at (3.0*0.3827, 3.0*0.9239);
\coordinate (VV5) at (3.0*0.0000, 3.0*1.0000);
\coordinate (VV6) at (3.0*-0.3827, 3.0*0.9239);
\coordinate (VV7) at (3.0*-0.7071, 3.0*0.7071);
\coordinate (VV8) at (3.0*-0.9239, 3.0*0.3827);
\coordinate (VV9) at (3.0*-1.0000, 3.0*0.0000);
\coordinate (VV10) at (3.0*-0.9239, 3.0*-0.3827);
\coordinate (VV11) at (3.0*-0.7071, 3.0*-0.7071);
\coordinate (VV12) at (3.0*-0.3827, 3.0*-0.9239);
\coordinate (VV13) at (3.0*0.0000, 3.0*-1.0000);
\coordinate (VV14) at (3.0*0.3827, 3.0*-0.9239);
\coordinate (VV15) at (3.0*0.7071, 3.0*-0.7071);
\coordinate (VV16) at (3.0*0.9239, 3.0*-0.3827);

\coordinate (VVV1) at (4.0*1.0000, 4.0*0.0000);
\coordinate (VVV2) at (4.0*0.9239, 4.0*0.3827);
\coordinate (VVV3) at (4.0*0.7071, 4.0*0.7071);
\coordinate (VVV4) at (4.0*0.3827, 4.0*0.9239);
\coordinate (VVV5) at (4.0*0.0000, 4.0*1.0000);
\coordinate (VVV6) at (4.0*-0.3827, 4.0*0.9239);
\coordinate (VVV7) at (4.0*-0.7071, 4.0*0.7071);
\coordinate (VVV8) at (4.0*-0.9239, 4.0*0.3827);
\coordinate (VVV9) at (4.0*-1.0000, 4.0*0.0000);
\coordinate (VVV10) at (4.0*-0.9239, 4.0*-0.3827);
\coordinate (VVV11) at (4.0*-0.7071, 4.0*-0.7071);
\coordinate (VVV12) at (4.0*-0.3827, 4.0*-0.9239);
\coordinate (VVV13) at (4.0*0.0000, 4.0*-1.0000);
\coordinate (VVV14) at (4.0*0.3827, 4.0*-0.9239);
\coordinate (VVV15) at (4.0*0.7071, 4.0*-0.7071);
\coordinate (VVV16) at (4.0*0.9239, 4.0*-0.3827);

\coordinate (VVVV1) at (5.0*1.0000, 5.0*0.0000);
\coordinate (VVVV2) at (5.0*0.9239, 5.0*0.3827);
\coordinate (VVVV3) at (5.0*0.7071, 5.0*0.7071);
\coordinate (VVVV4) at (5.0*0.3827, 5.0*0.9239);
\coordinate (VVVV5) at (5.0*0.0000, 5.0*1.0000);
\coordinate (VVVV6) at (5.0*-0.3827, 5.0*0.9239);
\coordinate (VVVV7) at (5.0*-0.7071, 5.0*0.7071);
\coordinate (VVVV8) at (5.0*-0.9239, 5.0*0.3827);
\coordinate (VVVV9) at (5.0*-1.0000, 5.0*0.0000);
\coordinate (VVVV10) at (5.0*-0.9239, 5.0*-0.3827);
\coordinate (VVVV11) at (5.0*-0.7071, 5.0*-0.7071);
\coordinate (VVVV12) at (5.0*-0.3827, 5.0*-0.9239);
\coordinate (VVVV13) at (5.0*0.0000, 5.0*-1.0000);
\coordinate (VVVV14) at (5.0*0.3827, 5.0*-0.9239);
\coordinate (VVVV15) at (5.0*0.7071, 5.0*-0.7071);
\coordinate (VVVV16) at (5.0*0.9239, 5.0*-0.3827);

\foreach \i/\j in {2/3,4/5,6/7,8/9,10/11,12/13,14/15,16/1} {
    \draw (V\i) -- (V\j);
    \draw (VVVV\i) -- (VVVV\j);
    \draw (VV\i) -- (VV\j);
    \draw (VVV\i) -- (VVV\j);
    \draw (VVVV\i) -- (VV\j);
    \draw (VVV\i) -- (V\j);
    
    \draw (V\i) -- (VVV\j);    
       \draw (VV\i) -- (VVVV\j);    
}

\foreach \i/\j in {1/2,3/4,5/6,7/8,9/10,11/12,13/14,15/16} {
    \draw (V\i) -- (V\j);
        \draw (VV\i) -- (VV\j);
    \draw (VVV\i) -- (VVV\j);
    \draw (VVVV\i) -- (VVVV\j);
    \draw (VVVV\i) -- (VVV\j);
    \draw (VVV\i) -- (VVVV\j);    
   \draw (V\i) -- (VV\j);
   \draw (VV\i) -- (V\j);
}

 % Connect V16 to V1
% Label vertices for the first set
\vertexLabelR{V1}{left} {$ $}
\vertexLabelR{V2}{left} {$ $}
\vertexLabelR{V3}{left} {$ $}
\vertexLabelR{V4}{left} {$ $}
\vertexLabelR{V5}{left} {$ $}
\vertexLabelR{V6}{left} {$ $}
\vertexLabelR{V7}{left} {$ $}
\vertexLabelR{V8}{left} {$ $}
\vertexLabelR{V9}{left} {$ $}
\vertexLabelR{V10}{left} {$ $}
\vertexLabelR{V11}{left} {$ $}
\vertexLabelR{V12}{left} {$ $}
\vertexLabelR{V13}{left} {$ $}
\vertexLabelR{V14}{left} {$ $}
\vertexLabelR{V15}{left} {$ $}
\vertexLabelR{V16}{left} {$ $}

\vertexLabelR{VV1}{left} {$ $}
\vertexLabelR{VV2}{left} {$ $}
\vertexLabelR{VV3}{left} {$ $}
\vertexLabelR{VV4}{left} {$ $}
\vertexLabelR{VV5}{left} {$ $}
\vertexLabelR{VV6}{left} {$ $}
\vertexLabelR{VV7}{left} {$ $}
\vertexLabelR{VV8}{left} {$ $}
\vertexLabelR{VV9}{left} {$ $}
\vertexLabelR{VV10}{left} {$ $}
\vertexLabelR{VV11}{left} {$ $}
\vertexLabelR{VV12}{left} {$ $}
\vertexLabelR{VV13}{left} {$ $}
\vertexLabelR{VV14}{left} {$ $}
\vertexLabelR{VV15}{left} {$ $}
\vertexLabelR{VV16}{left} {$ $}

\vertexLabelR{VVV1}{left} {$ $}
\vertexLabelR{VVV2}{left} {$ $}
\vertexLabelR{VVV3}{left} {$ $}
\vertexLabelR{VVV4}{left} {$ $}
\vertexLabelR{VVV5}{left} {$ $}
\vertexLabelR{VVV6}{left} {$ $}
\vertexLabelR{VVV7}{left} {$ $}
\vertexLabelR{VVV8}{left} {$ $}
\vertexLabelR{VVV9}{left} {$ $}
\vertexLabelR{VVV10}{left} {$ $}
\vertexLabelR{VVV11}{left} {$ $}
\vertexLabelR{VVV12}{left} {$ $}
\vertexLabelR{VVV13}{left} {$ $}
\vertexLabelR{VVV14}{left} {$ $}
\vertexLabelR{VVV15}{left} {$ $}
\vertexLabelR{VVV16}{left} {$ $}

\vertexLabelR{VVVV1}{left} {$ $}
\vertexLabelR{VVVV2}{left} {$ $}
\vertexLabelR{VVVV3}{left} {$ $}
\vertexLabelR{VVVV4}{left} {$ $}
\vertexLabelR{VVVV5}{left} {$ $}
\vertexLabelR{VVVV6}{left} {$ $}
\vertexLabelR{VVVV7}{left} {$ $}
\vertexLabelR{VVVV8}{left} {$ $}
\vertexLabelR{VVVV9}{left} {$ $}
\vertexLabelR{VVVV10}{left} {$ $}
\vertexLabelR{VVVV11}{left} {$ $}
\vertexLabelR{VVVV12}{left} {$ $}
\vertexLabelR{VVVV13}{left} {$ $}
\vertexLabelR{VVVV14}{left} {$ $}
\vertexLabelR{VVVV15}{left} {$ $}
\vertexLabelR{VVVV16}{left} {$ $}

\end{tikzpicture}
    \caption{The generalised $m$-gon graph $\G_{16,4}$ with $4$-cycles contracted to a point.}
    \label{fig:genmgon2}
\end{figure}

We call such graphs with $n=2^k \cdot m$ nodes {\bf generalised $m$-gon graphs of order $k$}, written $\G_{m,k}$. The set of $480$ graphs with automorphism groups of order at least $10^{11}$ contains $290$ $m$-gon graphs of order $2$, $130$ of order $3$, $50$ of order $4$, and $10$ additional graphs with a different structure. We prove the following statement.
\begin{lemma}
    \label{lem:genmgons}
    Let $\G_{m,k}$ be a generalised $m$-gon graph of order $k$ with $m \geq 2^k$. Then $\G_{m,k}$ is not the face graph of a vertex-transitive simplicial surface.
\end{lemma}

\begin{proof}
  Let $X$ be a simplicial surface with face graph $\F(X) =\G_{m,k}$, $m,k\geq 2$, $m \geq 2^k$, and let $H = \lambda_X(\Aut(X))$ be a subgroup of $\Aut(\G_{m,k})$.

  We go through the $13$ types of face-transitive simplicial surfaces from \Cref{section:construction}, obtaining an explicit obstruction to a vertex-faithful cycle double cover, and hence a contradiction to the existence of $X$, in every case. We group these $13$ cases into three main cases.

  \medskip

  \noindent 
  {\bf Case $\vf(X) \in \{ (1,3), (1,6)\}$}. In these two cases, $H$ acts transitive on the arcs of $\G_{m,k}$, which is a contradiction because the automorphism group of $\G_{m,k}$ splits its arcs into two orbits.

  \medskip

  \noindent 
  {\bf Case $\vf(X) \in \{ (1,1), (1,2)\}$}. For $\vf(X) = (1,2)$, $H$ splits the arcs of $\G_{m,k}$ into two orbits and the vertex defining umbrellas for $X$ are of type $(1,1,2,\ldots , 1,1,2)$, where $1$ denotes the large orbit, made out of $4$-cycles. In a double cycle cover defining $X$, choose an arc in orbit $2$, and follow the two cycles adjacent to it. Both of these cycles enter a $4$-cycle, and must exit this $4$-cycle after two arcs, and hence into the same arc. Hence the two cycles share more than a single arc and the cycle double cover is not vertex-faithful. A contradiction to the assumption that $X$ is simplicial.
  
  For $\vf(X) = (1,1)$ there are four types to distinguish. For type $1$, we have three arc orbits and vertex defining umbrellas of type $(1,2,3,\ldots , 1,2,3)$. Here, the same argument as before contradicts the double cycle cover to be vertex faithful and hence $X$ cannot be simplicial. For type 2, if the short $\Aut(\G)$ arc orbit is $1$, the same argument yields another contradiction. Otherwise, since $H$ acts transitive on vertex defining umbrellas, every cycle intersects the $4$-cycles of $\G_{m,k}$ in one or three arcs. Fix a cycle and a $4$-cycle intersecting this cycle in three arcs. Consider the two cycles intersecting the fourth arc of the this $4$-cycle. Since we are looking for a vertex-faithful double cover, each of these two cycles intersects the $4$-cycle in only one edge. But this means these cycles must intersect in at least three arcs, a contradiction.

  For $\vf(X) = (1,1)$, types $3$ and $4$, $H$ recovers the two arc orbits of $\Aut(\G_{m,k})$. The cycle double cover must contain the mono-coloured $4$-cycles, meaning that the remaining cycles must also be of length $4$. This is an immediate contradiction to the combinatorial structure of $\G_{m,k}$.

  \medskip

  \noindent 
  {\bf Case $\vf(X) \in \{ (3,1), (2,2), (2,1) \}$}. For  $\vf(X) = (3,1)$, $H$ defines three arc orbits and the double cycle cover is obtained by the cycles defined by removing one of the orbits, respectively. This yields the $4$-cycles consisting of the long $\Aut(\G_{m,k})$ arc orbit, and two other types of cycles with arcs alternating between the short and long arc orbits. Due to the combinatorial structure of $\G_{m,k}$ these cycles must run around the entire circular layout of the graph. This means, a given cycle $c$ must run through at least $m$ arcs of the short arc orbit, and that there are at most $2^k-1$ other cycles. Since at every arc, $c$ must meet a different other cycle, this is a contradiction to $m > 2^k-1$ by the pigeonhole principle. This same argument extends to the case $\vf(X) = (2,1)$ of type $3$.

  For $\vf(X) = (2,1)$ of type 1 and 2, as well as $\vf(X) = (2,2)$, the cycle double cover must contain the mono-coloured $4$-cycles. It follows that the other cycles must contain arcs alternating between the short and long arc orbits, and we can conclude as in the previous paragraph.
\end{proof}

Applying \Cref{lem:genmgons} to our list of $480$ graphs with large automorphism groups leaves $10$ leftover cases. For these remaining cubic graphs, we employ a modified version of our algorithm, with limited conjugacy tests which turns out to be powerful enough to prove that none of them admits a vertex-faithful cycle double cover, see \Cref{sec:implementation} for details.

\section{Implementation and Results}
\label{sec:implementation}

\label{section:implementation}

In this section we provide a brief description of how we use our theoretical results to compute a census of face-transitive surfaces. Moreover, we present and comment on our findings. Our implementations and the corresponding database of face-transitive surfaces are available in \cite{facetransitivesurfaces}.

\subsection{Notes on implementation}
\label{rem:multiple}

For our implementation we make use of GAP \cite{GAP4}, specifically, GAP-packages \texttt{SimplicialSurfaces} \cite{simplicialsurfacegap}, \texttt{Simpcomp} \cite{simpcomp}, \texttt{GraphSym} \cite{cubicvertextransitive} and \texttt{Digraphs} \cite{DeBeule2024aa}, and Magma \cite{magma}.  
For a given cubic node-transitive graph $\G=(V,E)$ with corresponding automorphism group $\Aut(\G)$, we apply the following steps.
\begin{enumerate}
  \item Compute the set $\mathcal{H}$ of all subgroups $H\leq \Aut(\G)$ satisfying $\vert H\vert= s \cdot \vert V\vert$, where $s\in \{1,2,3,6\}$. \\

  Classifying all subgroups of a given automorphism group and a given order (that is, classifying all subgroups of a given index) is computationally expensive. We therefore speed up computations by classifying subgroups up to conjugation. This restriction is justified, since conjugate subgroups yield isomorphic surfaces. We classify subgroups by constructing chains of maximal subgroups while checking node-transitivity and the order of the subgroups satisfy$\geq \vert V \vert$. This step is done in GAP or Magma, depending on properties of the input graph.
  \item For all $H\in \mathcal{H}$, compute all face-transitive surfaces with corresponding automorphism group and with $\G$ as a face graph (we apply the results established in \Cref{section:construction} for all cases matching the order of $H$). \\

  For each subgroup $H$ of order $\vert H\vert= s \cdot \vert V\vert$, $s\in \{1,2,3,6\}$, we construct cycle double covers of the given cubic graph with automorphism group $H$. Since not every cycle double cover is vertex-faithful, this is checked in a separate step. To speed up computations, we only construct one cycle of a cycle double cover, and generate the other cycles by the action of $H$ on the computed cycle.
  \item For graphs with large automorphism group we use \Cref{lem:genmgons}, if applicable, or run a modified version of our code with limited conjugacy tests. This strategy proved to be feasible for all cubic graphs in the census. 
\end{enumerate}

Note that the face-transitive surfaces $X$ stored in \cite{facetransitivesurfaces} are represented as lists of faces, where each face is given by the set of its incident vertices. In GAP, these vertices are encoded as integers. 

\subsection{Tests}

We perform multiple tests to ensure that our census of face-transitive surfaces \cite{facetransitivesurfaces} is complete:
\begin{enumerate}
    \item For all node-transitive cubic graphs $\G$ with fewer than $50$ nodes, we use brute-force methods to list their cycle double covers using the functionalities provided in \cite{simplicialsurfacegap}. We then compare the output with the result of our classification algorithm, up to isomorphism.
    \item For every node-transitive cubic graph $\G=(V,E)$ with $\vert \Aut(\G)\vert\leq 4000$ and $\vert V\vert \leq 600$, we compute all subgroups $H\leq \Aut(\G)$ of order $\vert H\vert =s\vert V\vert $, where $s=1,2,3,6$, using the function \texttt{ConjugacyClassesSubgroups} from \cite{GAP4}. We then compare these results to the candidate subgroups computed using our algorithm, which is based on chains of maximal subgroups, ensuring consistency up to isomorphism.
    \item For every node-transitive cubic graph $\G$ with $\vert \Aut(X)\vert \leq 2000000$ and $\vert V\vert \leq 800$ we compute a copy $\G'$ of $\G$ with randomly relabelled nodes and check whether the set of face-transitive surfaces with $\G$ as face graph is equal to the set of face-transitive surfaces with $\G'$ as face graph, up to isomorphism.
    \item We take $200$ known face-transitive surfaces from \cite{weddslist,facetransitivesurfaces,simpcomp} and check whether we can reconstruct them, by first computing their face graphs and then applying our algorithm. We further construct random copies of these simplicial surfaces, construct the face graphs of the original simplicial surfaces and their copies, and check whether the set of candidate groups is the same, up to isomorphism.
\end{enumerate}

\subsection{Results}
\label{subsection:results}

The results of our computations are summarised in the following table. In this table, we provide the numbers of different face-transitive surfaces that we have computed in our investigation. For this purpose we denote the set of all face-transitive surfaces $X$ with vertex-face type $\vf(X)$ and with up to $1280$ faces
by $M_{\vf(X)}.$  

\begin{table}[H]
    \centering
    \begin{tabular}{|*{8}{r|}}
        \hline
        $\vf(X)$ &  $\vert M_{\vf(X)}\vert$ & Type $1$ & Type $2$ & Type $3$ & Type $4$\\
        & & & &  &\vspace{-0.4cm}\\
        \hline
        $(3,1)$ & $6\, 645$& $\times\quad$& $\times\quad$ & $\times\quad$ & $\times\quad $\\
        $(2,2)$ & $4\, 373$& $\times\quad$& $\times\quad$ & $\times\quad$ & $\times \quad$\\

        $(2,1)$& $10\, 491$ & $1\,160$& $528$ & $8\,803$  & $\times\quad $\\
        $(1,1)$&$60\, 733$ &$56\,592$& $3\,376$ & $555$ & $210$ \\
        $(1,2)$& $4\, 098$ & $\times\quad$ &$\times\quad$ &$\times\quad$ &$\times\quad$  \\
        $(1,3)$& $223$& $210$& $13$ & $\times\quad$ & $\times \quad$ \\
        $(1,6)$& $239$& $\times\quad$& $\times\quad$ & $\times\quad$ & $\times \quad$\\
        \hline
    \end{tabular}
    \caption{Number of face-transitive surfaces with different vertex-face types.}
    \label{tab:ft1}
\end{table}
Here, an entry ``$\times$'' means that the corresponding type of face-transitive surfaces is not defined. Hence, up to isomorphism, there are exactly $86\,802$ face-transitive surfaces with at most $1\,280$ faces, coming from $75\,555$ cubic node-transitive graphs. Out of these face-transitive surfaces, exactly $80\,243$ are orientable and $6\,559$ are non-orientable. In particular, this proves \Cref{thm:main}.

Let $X$ be a surface in our census. Then the Euler characteristic of $X$ satisfies $-550 \leq \chi(X)\leq -2,$ if $X$ is orientable, and $-512 \leq \chi(X)\leq 1$ if $X$ is non-orientable. 

Of the $80\,243$ orientable face-transitive surfaces, $647$ are spheres (i.e., have positive Euler characteristic). Of these, $638$ are suspensions over a cycle and hence fall into the only infinite family of face-transitive spheres (note that this family contains the simplicial octahedron). The other $9$ are sporadic examples. A further $57\,405$ are of genus $1$, that is, simplicial tori.

Of the $6\,559$ non-orientable surfaces, exactly $6$ have positive Euler characteristic (that is, they triangulate the real projective plane). Since the orientable double cover of such a face-transitive projective plane must show up in the list of face-transitive spheres, we can conclude, that this is a complete classification of face-transitive projective planes.
To see this, note that the suspension over a cycle does not have a fixed-point free involution (as a set) and hence cannot be the double cover of a face-transitive projective plane.

\begin{proposition}
    There are exactly $6$ face-transitive projective planes
    with $6$, $7$, $13$, $16$, $16$ and $31$ vertices, respectively.
\end{proposition}

Even more interesting, none of the non-orientable face-transitive surfaces have Euler characteristic $0$. In fact, we have the following conjecture.

\begin{conjecture}
    \label{conj:KleinBottle}
    There is no face-transitive Klein bottle.
\end{conjecture}

Note that there are highly-symmetric triangulations of the Klein bottle: There are multiple types of equivelar and semi-equivelar maps on the Klein bottle that could, in principle, include face-transitive triangulations of Klein bottles, see \cite{Datta17SemiEquivelar} for a full classification and a canonical starting point for a possible proof of \Cref{conj:KleinBottle}. Moreover, there are numerous triangulations of the Klein bottle with vertex-transitive symmetry, see for example \cite[Table 2.13]{LutzDiss}. These latter examples, however, all consist of two face orbits. 

\medskip

Another interesting class of face-transitive surfaces are those that are minimal, i.e., contain the smallest possible number of vertices. According to \cite{mapcoloring}, such a surface must satisfy
\begin{equation}
  \label{eq:heawood}
  \vert X_0\vert = \Big\lceil\frac{7+\sqrt{49-24\chi(X)}}{2}\Big\rceil. 
\end{equation}
Here, we see that our data base contains exactly five minimal triangulations, namely
\begin{itemize}
    \item the simplicial tetrahedron,
    \item the triangulation of the torus with $7$ vertices \cite{mintorus},
    \item the triangulation of the real projective plane with $6$ vertices,
    \item a triangulation of a non-orientable surface of Euler characteristic $-5$ with $10$ vertices and
    \item a triangulation of an orientable surface of Euler characteristic $-40$ with $20$ vertices.
\end{itemize}

Notably, for all five examples, the fraction on the right hand side of \Cref{eq:heawood} is already an integer.

Moreover, we can observe three missing subclasses of face-transitive surfaces in our data:
\begin{itemize}
    \item Non-orientable face-transitive surfaces $X$ with $\vf(X)=(2,1)$, where $\lambda_X(\Aut(X))$ forms a $(2,1)$-group of $\F(X)$ of type $1$,
    \item Non-orientable face-transitive surfaces $X$ with $\vf(X)=(1,1)$, where $\lambda_X(\Aut(X))$ forms a $(1,1)$-group of $\F(X)$ of type $1$,
    \item Non-orientable face-transitive surfaces $X$ with $\vf(X)=(1,3)$, where $\lambda_X(\Aut(X))$ forms a $(1,3)$-group of $\F(X)$ of type $2.$
\end{itemize}

Our census \cite{facetransitivesurfaces} contains a file with a full list of all surfaces grouped by orientability, Euler characteristic, and vertex-face types. Note that we write $(v,s).i$ for face-transitive surfaces $X$ with vertex-face type $(v,s),$ where the groups $\lambda_X(\Aut(X))$ are $(v,s)$-groups of type $i.$ 
\Cref{tab:smallest} contains the smallest number $n(\vf(X))$ of faces of a face-transitive surface with vertex-face type $\vf(X)$. Here, we consider a face-transitive surface $X$ as smallest, if for every face-transitive surface $Y$, where (1) $X$ and $Y$ have the same vertex-face type and (2) $\lambda_X(\Aut(X))$ and $\lambda_Y(\Aut(Y))$ are both $(v,s)$-groups of the same type, the inequality $\vert X_2\vert \leq \vert Y_2\vert$ holds.
\begin{table}[htb]
    \centering
    \begin{tabular}{|*{11}{c|}}
        \hline
          \bm{$\vf(X)$}& \bm{$(3,1)$} & \bm{$(2,2)$} & \bm{$(2,1).1$} & \bm{$(2,1).2$} & \bm{$(2,1).3$} \\
        & & & &  &\vspace{-0.4cm}\\  
        \hline
  \bm{$ n({\vf(X)}) $} & $24$& $6$& $20$ & $56 $ & $32$\\

        \hline
          \bm{$\vf(X)$}& \bm{$(1,1).1$} & \bm{$(1,1).2$} & \bm{$(1,1).3$ }& \bm{$(1,1).4$} &\bm{$(1,2)$ }\\
        & & &  &\vspace{-0.4cm}&\\  
        \hline
 \bm{ $n({\vf(X)}) $ }& $18$& $72$& $72$ & $112$&$24$\\
        \hline
         \bm{ $\vf(X)$}&  \bm{$(1,3).1$} & \bm{$(1,3).2$} & \bm{$(1,6)$ }&&\\
        & & & &  \vspace{-0.4cm}&\\  
        \hline
  \bm{$n({\vf(X)}) $} &  $14$& $144$ & $4$ &&\\
        \hline
    \end{tabular}
    \caption{Smallest face numbers of face-transitive surfaces for all $13$ automorphism group types. \label{tab:smallest}}
\end{table}

\begin{remark}
  \label{rem:nonisomorphic}
  As shown in \Cref{section:construction}, we construct face-transitive surfaces from cubic node-transitive graphs by computing suitable subgroups of the automorphism groups of the given cubic graphs. Note that a given graph and subgroup of its automorphism group does not always uniquely determine an isomorphism type of a face-transitive surface.

  For instance, there is a cubic node-transitive graph $\G$ on $32$ nodes with a subgroup $H\leq \Aut(\G)$ of order $\vert H\vert =6\cdot 32$ that yields two non-isomorphic face-transitive surfaces $X$ and $Y$, both satisfying $H=\lambda_X(\Aut(X))=\lambda_Y(\Aut(Y))$ and hence both of vertex-face type $\vf(X)=\vf(Y)=(1,6)$. The surface $X$ is the $12$-vertex triangulation of the orientable surface of genus $3$, better known as the dual Dyck map with Schl\"afli symbol $\{3,8\}_6$ \cite{CONDER2001224,weddslist}. Surface $Y$ is the $16$-vertex regular triangulation of the torus with Schl\"afli symbol $\{3,6\}_8$, see \cite{kuehnel,weddslist}. See \Cref{fig:G} for a picture of $\G$ and \Cref{fig:groups} for a picture of $X$ and $Y$.

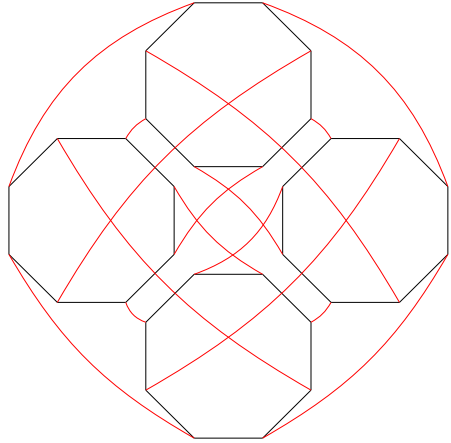
\begin{figure}[htb]
    \centering
  \scalebox{0.45}{\begin{tikzpicture}[vertexBall, edgeDouble=nolabels, faceStyle, scale=1]

% Define the coordinates

\coordinate (V1) at (2.414+4.0, -1.0);
\coordinate (V2) at (2.414+4.0, 1.0);
\coordinate (V3) at (1.0+4.0, 2.414);
\coordinate (V4) at (-1.0+4.0, 2.414);
\coordinate (V5) at (-2.414+4.0, 1.0);
\coordinate (V6) at (-2.414+4.0, -1.0);
\coordinate (V7) at (-1.0+4.0, -2.414);
\coordinate (V8) at (1.0+4.0, -2.414);

\coordinate (V11) at (2.414, -1.0+4.0);
\coordinate (V12) at (2.414, 1.0+4.0);
\coordinate (V13) at (1.0, 2.414+4.0);
\coordinate (V14) at (-1.0, 2.414+4.0);
\coordinate (V15) at (-2.414, 1.0+4.0);
\coordinate (V16) at (-2.414, -1.0+4.0);
\coordinate (V17) at (-1.0, -2.414+4.0);
\coordinate (V18) at (1.0, -2.414+4.0);

\coordinate (V21) at (2.414-4.0, -1.0);
\coordinate (V22) at (2.414-4.0, 1.0);
\coordinate (V23) at (1.0-4.0, 2.414);
\coordinate (V24) at (-1.0-4.0, 2.414);
\coordinate (V25) at (-2.414-4.0, 1.0);
\coordinate (V26) at (-2.414-4.0, -1.0);
\coordinate (V27) at (-1.0-4.0, -2.414);
\coordinate (V28) at (1.0-4.0, -2.414);

\coordinate (V31) at (2.414, -1.0-4.0);
\coordinate (V32) at (2.414, 1.0-4.0);
\coordinate (V33) at (1.0, 2.414-4.0);
\coordinate (V34) at (-1.0, 2.414-4.0);
\coordinate (V35) at (-2.414, 1.0-4.0);
\coordinate (V36) at (-2.414, -1.0-4.0);
\coordinate (V37) at (-1.0, -2.414-4.0);
\coordinate (V38) at (1.0, -2.414-4.0);

\foreach \i/\j in {1/2,2/3,3/4,4/5,5/6,6/7,7/8,8/1} {
    \draw (V\i) -- (V\j);
}

\foreach \i/\j in {11/12,12/13,13/14,14/15,15/16,16/17,17/18,18/11} {
    \draw (V\i) -- (V\j);
}

\foreach \i/\j in {21/22,22/23,23/24,24/25,25/26,26/27,27/28,28/21} {
    \draw (V\i) -- (V\j);
}

\foreach \i/\j in {31/32,32/33,33/34,34/35,35/36,36/37,37/38,38/31} {
    \draw (V\i) -- (V\j);
}

\draw[red] (V2) to[bend right=25] (V13);
\draw[red] (V4) to[bend right=15] (V11);
\draw[red] (V6) to[bend right=15] (V17);
\draw[red] (V8) to[bend right=15] (V15);

\draw[red] (V16) to[bend right=15] (V23);
\draw[red] (V18) to[bend right=15] (V21);
\draw[red] (V12) to[bend right=15] (V27);
\draw[red] (V14) to[bend right=25] (V25);

\draw[red] (V22) to[bend right=15] (V33);
\draw[red] (V24) to[bend right=15] (V31);
\draw[red] (V26) to[bend right=15] (V37);
\draw[red] (V28) to[bend right=25] (V35);

\draw[red] (V36) to[bend right=15] (V3);
\draw[red] (V38) to[bend right=15] (V1);
\draw[red] (V32) to[bend right=15] (V7);
\draw[red] (V34) to[bend right=25] (V5);

%\foreach \i/\j in {1/4,  3/6,  5/8, 7/10, 9/12,  11/14,  13/16,15/2} {
%    \draw[-] (VVVV\i) to[bend right=55] (VVVV\j);
%}

\end{tikzpicture}}
     \caption{The highly-symmetric $32$-node face graph of both $\{3, 6\}_8$ and $\{3, 8\}_6$. Note the similarities to the generalised $m$-gon graphs from \Cref{section:problems}}
    \label{fig:G}
\end{figure}

  Note that condition $H=\lambda_X(\Aut(X))=\lambda_Y(\Aut(Y))$ requires that every automorphism in $H$ acting on $\F(X) = \F(Y)$ leaves both cycle double covers (the one defining $X$ and the one defining $Y$) invariant. Inspecting $\G$, the $12$ cycles of length $8$ are given by the four octagons, and eight cycles passing through all four octagons cyclically. The cycles of length $6$ are simply given by the set of shortest cycles in $\G$.

  More generally, there are $6\,813$ tuples $(\G,H)$ such that $\G=(V,E)$ is a node-transitive cubic graph with $\vert V\vert\leq 1280$ that yields multiple non-isomorphic face-transitive surfaces $X$ with $\Aut(X) = H.$

\begin{figure}[H]
    \centering
  \scalebox{0.65}{\begin{tikzpicture}[vertexBall, edgeDouble, faceStyle, scale=2]

% Define the coordinates of the vertices
\coordinate (V1_1) at (-2.499999999999999, -0.8660254037844395);
\coordinate (V2_1) at (-1.5, -0.866025403784439);
\coordinate (V3_1) at (-0.4999999999999998, -0.8660254037844387);
\coordinate (V4_1) at (-0.9999999999999996, -1.732050807568877);
\coordinate (V5_1) at (-1.999999999999999, -1.732050807568878);
\coordinate (V6_1) at (0., -1.732050807568877);
\coordinate (V7_1) at (-2.999999999999999, -1.732050807568878);
\coordinate (V7_2) at (1., -1.732050807568877);
\coordinate (V8_1) at (0.5000000000000001, -2.598076211353316);
\coordinate (V9_1) at (-2.499999999999998, -2.598076211353316);
\coordinate (V9_2) at (1.5, -2.598076211353316);
\coordinate (V10_1) at (-1., 0.);
\coordinate (V10_2) at (1., -3.464101615137754);
\coordinate (V11_1) at (0.5000000000000001, -0.8660254037844386);
\coordinate (V11_2) at (-3.499999999999999, -0.8660254037844399);
\coordinate (V12_1) at (-2., 0.);
\coordinate (V12_2) at (0., -3.464101615137754);
\coordinate (V13_1) at (-0.4999999999999998, -2.598076211353316);
\coordinate (V14_1) at (0., 0.);
\coordinate (V14_2) at (2., -3.464101615137754);
\coordinate (V15_1) at (-4., 0.);
\coordinate (V15_2) at (-2.999999999999999, 0.);
\coordinate (V15_3) at (-1., -3.464101615137754);
\coordinate (V16_1) at (0.4999999999999999, 0.8660254037844386);
\coordinate (V16_2) at (-1.499999999999999, -2.598076211353316);
%\coordinate (V16_3) at (2.5, -2.598076211353316);
\coordinate (V14_3) at (-2.0, -3.464101615137754);

% Fill in the faces
\fill[face]  (V15_3) -- (V16_2) -- (V14_3) -- cycle;
\node[faceLabel] at (barycentric cs:V15_3=1,V16_2=1,V14_3=1) {$1$};

\draw[edge] (V16_2) -- node[edgeLabel] {$47$} (V14_3);
\draw[edge] (V15_3) -- node[edgeLabel] {$46$} (V14_3);

\fill[face]  (V13_1) -- (V15_3) -- (V12_2) -- cycle;
\node[faceLabel] at (barycentric cs:V13_1=1,V15_3=1,V12_2=1) {$2$};
\fill[face]  (V15_2) -- (V11_2) -- (V15_1) -- cycle;
\node[faceLabel] at (barycentric cs:V15_1=1,V11_2=1,V15_2=1) {$3$};
\fill[face]  (V8_1) -- (V10_2) -- (V9_2) -- cycle;
\node[faceLabel] at (barycentric cs:V8_1=1,V10_2=1,V9_2=1) {$4$};
\fill[face]  (V6_1) -- (V7_2) -- (V11_1) -- cycle;
\node[faceLabel] at (barycentric cs:V6_1=1,V7_2=1,V11_1=1) {$5$};
\fill[face]  (V7_1) -- (V9_1) -- (V5_1) -- cycle;
\node[faceLabel] at (barycentric cs:V7_1=1,V9_1=1,V5_1=1) {$6$};
\fill[face]  (V5_1) -- (V16_2) -- (V4_1) -- cycle;
\node[faceLabel] at (barycentric cs:V5_1=1,V16_2=1,V4_1=1) {$7$};
\fill[face]  (V6_1) -- (V13_1) -- (V8_1) -- cycle;
\node[faceLabel] at (barycentric cs:V6_1=1,V13_1=1,V8_1=1) {$8$};
\fill[face]  (V6_1) -- (V8_1) -- (V7_2) -- cycle;
\node[faceLabel] at (barycentric cs:V6_1=1,V8_1=1,V7_2=1) {$9$};
\fill[face]  (V3_1) -- (V14_1) -- (V10_1) -- cycle;
\node[faceLabel] at (barycentric cs:V3_1=1,V14_1=1,V10_1=1) {$10$};
\fill[face]  (V10_1) -- (V12_1) -- (V2_1) -- cycle;
\node[faceLabel] at (barycentric cs:V10_1=1,V12_1=1,V2_1=1) {$11$};
\fill[face]  (V3_1) -- (V10_1) -- (V2_1) -- cycle;
\node[faceLabel] at (barycentric cs:V3_1=1,V10_1=1,V2_1=1) {$12$};
\fill[face]  (V12_1) -- (V15_2) -- (V1_1) -- cycle;
\node[faceLabel] at (barycentric cs:V12_1=1,V15_2=1,V1_1=1) {$13$};
\fill[face]  (V8_1) -- (V13_1) -- (V12_2) -- cycle;
\node[faceLabel] at (barycentric cs:V8_1=1,V13_1=1,V12_2=1) {$14$};
\fill[face]  (V3_1) -- (V6_1) -- (V11_1) -- cycle;
\node[faceLabel] at (barycentric cs:V3_1=1,V6_1=1,V11_1=1) {$15$};
\fill[face]  (V1_1) -- (V7_1) -- (V5_1) -- cycle;
\node[faceLabel] at (barycentric cs:V1_1=1,V7_1=1,V5_1=1) {$16$};
\fill[face]  (V2_1) -- (V12_1) -- (V1_1) -- cycle;
\node[faceLabel] at (barycentric cs:V2_1=1,V12_1=1,V1_1=1) {$17$};
\fill[face]  (V13_1) -- (V16_2) -- (V15_3) -- cycle;
\node[faceLabel] at (barycentric cs:V13_1=1,V16_2=1,V15_3=1) {$18$};
\fill[face]  (V3_1) -- (V4_1) -- (V6_1) -- cycle;
\node[faceLabel] at (barycentric cs:V3_1=1,V4_1=1,V6_1=1) {$19$};
\fill[face]  (V14_1) -- (V3_1) -- (V11_1) -- cycle;
\node[faceLabel] at (barycentric cs:V14_1=1,V3_1=1,V11_1=1) {$20$};
\fill[face]  (V3_1) -- (V2_1) -- (V4_1) -- cycle;
\node[faceLabel] at (barycentric cs:V3_1=1,V2_1=1,V4_1=1) {$21$};
\fill[face]  (V1_1) -- (V15_2) -- (V11_2) -- cycle;
\node[faceLabel] at (barycentric cs:V1_1=1,V15_2=1,V11_2=1) {$22$};
\fill[face]  (V8_1) -- (V9_2) -- (V7_2) -- cycle;
\node[faceLabel] at (barycentric cs:V8_1=1,V9_2=1,V7_2=1) {$23$};
\fill[face]  (V4_1) -- (V16_2) -- (V13_1) -- cycle;
\node[faceLabel] at (barycentric cs:V4_1=1,V16_2=1,V13_1=1) {$24$};
\fill[face]  (V10_2) -- (V14_2) -- (V9_2) -- cycle;
\node[faceLabel] at (barycentric cs:V10_2=1,V14_2=1,V9_2=1) {$25$};
\fill[face]  (V2_1) -- (V5_1) -- (V4_1) -- cycle;
\node[faceLabel] at (barycentric cs:V2_1=1,V5_1=1,V4_1=1) {$26$};
\fill[face]  (V14_3) -- (V16_2) -- (V9_1) -- cycle;
\node[faceLabel] at (barycentric cs:V14_3=1,V16_2=1,V9_1=1) {$27$};
\fill[face]  (V8_1) -- (V12_2) -- (V10_2) -- cycle;
\node[faceLabel] at (barycentric cs:V8_1=1,V12_2=1,V10_2=1) {$28$};
\fill[face]  (V2_1) -- (V1_1) -- (V5_1) -- cycle;
\node[faceLabel] at (barycentric cs:V2_1=1,V1_1=1,V5_1=1) {$29$};
\fill[face]  (V4_1) -- (V13_1) -- (V6_1) -- cycle;
\node[faceLabel] at (barycentric cs:V4_1=1,V13_1=1,V6_1=1) {$30$};
\fill[face]  (V1_1) -- (V11_2) -- (V7_1) -- cycle;
\node[faceLabel] at (barycentric cs:V1_1=1,V11_2=1,V7_1=1) {$31$};
\fill[face]  (V9_1) -- (V16_2) -- (V5_1) -- cycle;
\node[faceLabel] at (barycentric cs:V9_1=1,V16_2=1,V5_1=1) {$32$};

% Draw the edges
\draw[edge] (V1_1) -- node[edgeLabel] {$1$} (V2_1);
\draw[edge] (V5_1) -- node[edgeLabel] {$2$} (V1_1);
\draw[edge] (V7_1) -- node[edgeLabel] {$3$} (V1_1);
\draw[edge] (V11_2) -- node[edgeLabel] {$4$} (V1_1);
\draw[edge] (V1_1) -- node[edgeLabel] {$5$} (V12_1);
\draw[edge] (V15_2) -- node[edgeLabel] {$6$} (V1_1);
\draw[edge] (V2_1) -- node[edgeLabel] {$7$} (V3_1);
\draw[edge] (V4_1) -- node[edgeLabel] {$8$} (V2_1);
\draw[edge] (V5_1) -- node[edgeLabel] {$9$} (V2_1);
\draw[edge] (V2_1) -- node[edgeLabel] {$10$} (V10_1);
\draw[edge] (V12_1) -- node[edgeLabel] {$11$} (V2_1);
\draw[edge] (V4_1) -- node[edgeLabel] {$12$} (V3_1);
\draw[edge] (V6_1) -- node[edgeLabel] {$13$} (V3_1);
\draw[edge] (V10_1) -- node[edgeLabel] {$14$} (V3_1);
\draw[edge] (V11_1) -- node[edgeLabel] {$15$} (V3_1);
\draw[edge] (V3_1) -- node[edgeLabel] {$16$} (V14_1);
\draw[edge] (V4_1) -- node[edgeLabel] {$17$} (V5_1);
\draw[edge] (V6_1) -- node[edgeLabel] {$18$} (V4_1);
\draw[edge] (V13_1) -- node[edgeLabel] {$19$} (V4_1);
\draw[edge] (V4_1) -- node[edgeLabel] {$20$} (V16_2);
\draw[edge] (V5_1) -- node[edgeLabel] {$21$} (V7_1);
\draw[edge] (V5_1) -- node[edgeLabel] {$22$} (V9_1);
\draw[edge] (V16_2) -- node[edgeLabel] {$23$} (V5_1);
\draw[edge] (V7_2) -- node[edgeLabel] {$24$} (V6_1);
\draw[edge] (V8_1) -- node[edgeLabel] {$25$} (V6_1);
\draw[edge] (V11_1) -- node[edgeLabel] {$26$} (V6_1);
\draw[edge] (V6_1) -- node[edgeLabel] {$27$} (V13_1);
\draw[edge] (V7_2) -- node[edgeLabel] {$28$} (V8_1);
\draw[edge] (V9_1) -- node[edgeLabel] {$29$} (V7_1);
\draw[edge] (V7_2) -- node[edgeLabel] {$29$} (V9_2);
\draw[edge] (V7_1) -- node[edgeLabel] {$30$} (V11_2);
\draw[edge] (V11_1) -- node[edgeLabel] {$30$} (V7_2);
\draw[edge] (V9_2) -- node[edgeLabel] {$31$} (V8_1);
\draw[edge] (V10_2) -- node[edgeLabel] {$32$} (V8_1);
\draw[edge] (V12_2) -- node[edgeLabel] {$33$} (V8_1);
\draw[edge] (V8_1) -- node[edgeLabel] {$34$} (V13_1);
\draw[edge] (V9_2) -- node[edgeLabel] {$35$} (V10_2);
\draw[edge] (V9_2) -- node[edgeLabel] {$36$} (V14_2);
\draw[edge] (V16_2) -- node[edgeLabel] {$37$} (V9_1);
%\draw[edge] (V9_2) -- node[edgeLabel] {$37$} (V16_3);
\draw[edge] (V12_1) -- node[edgeLabel] {$38$} (V10_1);
\draw[edge] (V10_2) -- node[edgeLabel] {$38$} (V12_2);
\draw[edge] (V10_1) -- node[edgeLabel] {$39$} (V14_1);
\draw[edge] (V14_2) -- node[edgeLabel] {$39$} (V10_2);
\draw[edge] (V11_1) -- node[edgeLabel] {$40$} (V14_1);
\draw[edge] (V15_1) -- node[edgeLabel] {$40$} (V11_2);
\draw[edge] (V11_2) -- node[edgeLabel] {$41$} (V15_2);
\draw[edge] (V12_2) -- node[edgeLabel] {$42$} (V13_1);
\draw[edge] (V15_2) -- node[edgeLabel] {$43$} (V12_1);
\draw[edge] (V12_2) -- node[edgeLabel] {$43$} (V15_3);
\draw[edge] (V15_3) -- node[edgeLabel] {$44$} (V13_1);
\draw[edge] (V13_1) -- node[edgeLabel] {$45$} (V16_2);
\draw[edge] (V15_1) -- node[edgeLabel] {$46$} (V15_2);
%\draw[edge] (V14_1) -- node[edgeLabel] {$47$} (V16_1);
%\draw[edge] (V16_3) -- node[edgeLabel] {$47$} (V14_2);
%\draw[edge] (V16_1) -- node[edgeLabel] {$48$} (V15_1);
\draw[edge] (V15_3) -- node[edgeLabel] {$48$} (V16_2);
\draw[edge] (V16_2) -- node[edgeLabel] {$47$} (V14_3);
\draw[edge] (V15_3) -- node[edgeLabel] {$46$} (V14_3);
\draw[edge] (V9_1) -- node[edgeLabel] {$36$} (V14_3);

% Draw the vertices
\vertexLabelR{V1_1}{left}{$1$}
\vertexLabelR{V2_1}{left}{$2$}
\vertexLabelR{V3_1}{left}{$3$}
\vertexLabelR{V4_1}{left}{$4$}
\vertexLabelR{V5_1}{left}{$5$}
\vertexLabelR{V6_1}{left}{$6$}
\vertexLabelR{V7_1}{left}{$7$}
\vertexLabelR{V7_2}{left}{$7$}
\vertexLabelR{V8_1}{left}{$8$}
\vertexLabelR{V9_1}{left}{$9$}
\vertexLabelR{V9_2}{left}{$9$}
\vertexLabelR{V10_1}{left}{$10$}
\vertexLabelR{V10_2}{left}{$10$}
\vertexLabelR{V11_1}{left}{$11$}
\vertexLabelR{V11_2}{left}{$11$}
\vertexLabelR{V12_1}{left}{$12$}
\vertexLabelR{V12_2}{left}{$12$}
\vertexLabelR{V13_1}{left}{$13$}
\vertexLabelR{V14_1}{left}{$14$}
\vertexLabelR{V14_2}{left}{$14$}
\vertexLabelR{V15_1}{left}{$14$}
\vertexLabelR{V15_2}{left}{$15$}
\vertexLabelR{V15_3}{left}{$15$}
%\vertexLabelR{V16_1}{left}{$16$}
\vertexLabelR{V16_2}{left}{$16$}
%\vertexLabelR{V16_3}{left}{$16$}
\vertexLabelR{V14_3}{left}{$14$}

\end{tikzpicture}}
  \raisebox{-1.0cm}{\scalebox{0.6}{\begin{tikzpicture}[vertexBall, edgeDouble, faceStyle, scale=2.5]

% Define the coordinates of the vertices

\coordinate (V0) at (0.0, 0.0);
\coordinate (V1) at (1.207, -.5);
\coordinate (V2) at (1.207, .5);
\coordinate (V3) at (.5, 1.207);
\coordinate (V4) at (-.5, 1.207);
\coordinate (V5) at (-1.207,.5);
\coordinate (V6) at (-1.207,-.5);
\coordinate (V7) at (-.5, -1.207);
\coordinate (V8) at (.5, -1.207);

\coordinate (V9) at (1.7, 0.0);
\coordinate (V10) at (1.207,1.207);
\coordinate (V11) at (0.0, 1.7);
\coordinate (V12) at (-1.207,1.207);
\coordinate (V13) at (-1.7,0.0);
\coordinate (V14) at (-1.207,-1.207);
\coordinate (V15) at (0.0, -1.7);
\coordinate (V16) at (1.207,-1.207);

\coordinate (V17) at (2.2, -.8);
\coordinate (V18) at (2.1,.85);
\coordinate (V19) at (.8, 2.2);
\coordinate (V20) at (-.85,2.1);
\coordinate (V21) at (-2.2,.8);
\coordinate (V22) at (-2.1,-.85);
\coordinate (V23) at (-.8, -2.2);
\coordinate (V24) at (.85,-2.1);

\coordinate (V25) at (1.7, -1.15);
\coordinate (V26) at (2.0, 0.3);
\coordinate (V27) at (1.15, 1.7);
\coordinate (V28) at (-.3,2.0);
\coordinate (V29) at (-1.7,1.15);
\coordinate (V30) at (-2.0,-.3);
\coordinate (V31) at (-1.15, -1.7);
\coordinate (V32) at (.3,-2.0);

% Fill in the faces
\fill[face]  (V0) -- (V1) -- (V2) -- cycle;
\node[faceLabel] at (barycentric cs:V0=1,V1=1,V2=1) {$1$};
\fill[face]  (V0) -- (V2) -- (V3) -- cycle;
\node[faceLabel] at (barycentric cs:V0=1,V2=1,V3=1) {$2$};
\fill[face]  (V0) -- (V3) -- (V4) -- cycle;
\node[faceLabel] at (barycentric cs:V0=1,V3=1,V4=1) {$3$};
\fill[face]  (V0) -- (V4) -- (V5) -- cycle;
\node[faceLabel] at (barycentric cs:V0=1,V4=1,V5=1) {$4$};
\fill[face]  (V0) -- (V5) -- (V6) -- cycle;
\node[faceLabel] at (barycentric cs:V0=1,V5=1,V6=1) {$5$};
\fill[face]  (V0) -- (V6) -- (V7) -- cycle;
\node[faceLabel] at (barycentric cs:V0=1,V6=1,V7=1) {$6$};
\fill[face]  (V0) -- (V7) -- (V8) -- cycle;
\node[faceLabel] at (barycentric cs:V0=1,V7=1,V8=1) {$7$};
\fill[face]  (V0) -- (V8) -- (V1) -- cycle;
\node[faceLabel] at (barycentric cs:V0=1,V8=1,V1=1) {$8$};

\fill[face]  (V1) -- (V2) -- (V9) -- cycle;
\node[faceLabel] at (barycentric cs:V1=1,V2=1,V9=2) {$9$};
\fill[face]  (V2) -- (V3) -- (V10) -- cycle;
\node[faceLabel] at (barycentric cs:V2=1,V3=1,V10=2) {$10$};
\fill[face]  (V3) -- (V4) -- (V11) -- cycle;
\node[faceLabel] at (barycentric cs:V3=1,V4=1,V11=2) {$11$};
\fill[face]  (V4) -- (V5) -- (V12) -- cycle;
\node[faceLabel] at (barycentric cs:V4=1,V5=1,V12=2) {$12$};
\fill[face]  (V5) -- (V6) -- (V13) -- cycle;
\node[faceLabel] at (barycentric cs:V5=1,V6=1,V13=2) {$13$};
\fill[face]  (V6) -- (V7) -- (V14) -- cycle;
\node[faceLabel] at (barycentric cs:V6=1,V7=1,V14=2) {$14$};
\fill[face]  (V7) -- (V8) -- (V15) -- cycle;
\node[faceLabel] at (barycentric cs:V7=1,V8=1,V15=2) {$15$};
\fill[face]  (V8) -- (V1) -- (V16) -- cycle;
\node[faceLabel] at (barycentric cs:V8=1,V1=1,V16=2) {$16$};

\fill[face]  (V1) -- (V9) -- (V17) -- cycle;
\node[faceLabel] at (barycentric cs:V1=1,V9=1,V17=1) {$17$};
\fill[face]  (V2) -- (V10) -- (V18) -- cycle;
\node[faceLabel] at (barycentric cs:V2=1,V10=1,V18=1) {$18$};
\fill[face]  (V3) -- (V11) -- (V19) -- cycle;
\node[faceLabel] at (barycentric cs:V3=1,V11=1,V19=1) {$19$};
\fill[face]  (V4) -- (V12) -- (V20) -- cycle;
\node[faceLabel] at (barycentric cs:V4=1,V12=1,V20=1) {$20$};
\fill[face]  (V5) -- (V13) -- (V21) -- cycle;
\node[faceLabel] at (barycentric cs:V5=1,V13=1,V21=1) {$21$};
\fill[face]  (V6) -- (V14) -- (V22) -- cycle;
\node[faceLabel] at (barycentric cs:V6=1,V14=1,V22=1) {$22$};
\fill[face]  (V7) -- (V15) -- (V23) -- cycle;
\node[faceLabel] at (barycentric cs:V7=1,V15=1,V23=1) {$23$};
\fill[face]  (V8) -- (V16) -- (V24) -- cycle;
\node[faceLabel] at (barycentric cs:V8=1,V16=1,V24=1) {$24$};

\fill[face]  (V1) -- (V17) -- (V25) -- cycle;
\node[faceLabel] at (barycentric cs:V1=0.7,V25=1,V17=1) {$25$};
\fill[face]  (V2) -- (V18) -- (V26) -- cycle;
\node[faceLabel] at (barycentric cs:V2=0.7,V26=1,V18=1) {$26$};
\fill[face]  (V3) -- (V19) -- (V27) -- cycle;
\node[faceLabel] at (barycentric cs:V3=0.7,V27=1,V19=1) {$27$};
\fill[face]  (V4) -- (V20) -- (V28) -- cycle;
\node[faceLabel] at (barycentric cs:V4=0.7,V28=1,V20=1) {$28$};
\fill[face]  (V5) -- (V21) -- (V29) -- cycle;
\node[faceLabel] at (barycentric cs:V5=0.7,V29=1,V21=1) {$29$};
\fill[face]  (V6) -- (V22) -- (V30) -- cycle;
\node[faceLabel] at (barycentric cs:V6=0.7,V30=1,V22=1) {$30$};
\fill[face]  (V7) -- (V23) -- (V31) -- cycle;
\node[faceLabel] at (barycentric cs:V7=0.7,V31=1,V23=1) {$31$};
\fill[face]  (V8) -- (V24) -- (V32) -- cycle;
\node[faceLabel] at (barycentric cs:V8=0.7,V32=1,V24=1) {$32$};

% Draw the edges
\draw[edge] (V0) -- node[edgeLabel] {$1$} (V1);
\draw[edge] (V0) -- node[edgeLabel] {$2$} (V2);
\draw[edge] (V0) -- node[edgeLabel] {$3$} (V3);
\draw[edge] (V0) -- node[edgeLabel] {$4$} (V4);
\draw[edge] (V0) -- node[edgeLabel] {$5$} (V5);
\draw[edge] (V0) -- node[edgeLabel] {$6$} (V6);
\draw[edge] (V0) -- node[edgeLabel] {$7$} (V7);
\draw[edge] (V0) -- node[edgeLabel] {$8$} (V8);
\draw[edge] (V1) -- node[edgeLabel] {$9$} (V2);
\draw[edge] (V2) -- node[edgeLabel] {$10$} (V3);
\draw[edge] (V3) -- node[edgeLabel] {$11$} (V4);
\draw[edge] (V4) -- node[edgeLabel] {$12$} (V5);
\draw[edge] (V5) -- node[edgeLabel] {$13$} (V6);
\draw[edge] (V6) -- node[edgeLabel] {$14$} (V7);
\draw[edge] (V7) -- node[edgeLabel] {$15$} (V8);
\draw[edge] (V8) -- node[edgeLabel] {$16$} (V1);

\draw[edge] (V9) -- node[edgeLabel] {$17$} (V1);
\draw[edge] (V9) -- node[edgeLabel] {$18$} (V2);
\draw[edge] (V10) -- node[edgeLabel] {$19$} (V2);
\draw[edge] (V10) -- node[edgeLabel] {$20$} (V3);
\draw[edge] (V11) -- node[edgeLabel] {$21$} (V3);
\draw[edge] (V11) -- node[edgeLabel] {$22$} (V4);
\draw[edge] (V12) -- node[edgeLabel] {$23$} (V4);
\draw[edge] (V12) -- node[edgeLabel] {$24$} (V5);
\draw[edge] (V13) -- node[edgeLabel] {$25$} (V5);
\draw[edge] (V13) -- node[edgeLabel] {$26$} (V6);
\draw[edge] (V14) -- node[edgeLabel] {$27$} (V6);
\draw[edge] (V14) -- node[edgeLabel] {$28$} (V7);
\draw[edge] (V15) -- node[edgeLabel] {$29$} (V7);
\draw[edge] (V15) -- node[edgeLabel] {$30$} (V8);
\draw[edge] (V16) -- node[edgeLabel] {$31$} (V8);
\draw[edge] (V16) -- node[edgeLabel] {$32$} (V1);

\draw[edge] (V1) -- node[edgeLabel] {$33$} (V17);
\draw[edge] (V9) -- node[edgeLabel] {$22$} (V17);
\draw[edge] (V2) -- node[edgeLabel] {$34$} (V18);
\draw[edge] (V10) -- node[edgeLabel] {$24$} (V18);
\draw[edge] (V3) -- node[edgeLabel] {$35$} (V19);
\draw[edge] (V11) -- node[edgeLabel] {$26$} (V19);
\draw[edge] (V4) -- node[edgeLabel] {$36$} (V20);
\draw[edge] (V12) -- node[edgeLabel] {$28$} (V20);
\draw[edge] (V5) -- node[edgeLabel] {$37$} (V21);
\draw[edge] (V13) -- node[edgeLabel] {$30$} (V21);
\draw[edge] (V6) -- node[edgeLabel] {$38$} (V22);
\draw[edge] (V14) -- node[edgeLabel] {$32$} (V22);
\draw[edge] (V7) -- node[edgeLabel] {$39$} (V23);
\draw[edge] (V15) -- node[edgeLabel] {$18$} (V23);
\draw[edge] (V8) -- node[edgeLabel] {$40$} (V24);
\draw[edge] (V16) -- node[edgeLabel] {$20$} (V24);

\draw[edge] (V1) -- node[edgeLabel] {$41$} (V25);
\draw[edge] (V17) -- node[edgeLabel] {$44$} (V25);
\draw[edge] (V2) -- node[edgeLabel] {$42$} (V26);
\draw[edge] (V18) -- node[edgeLabel] {$45$} (V26);
\draw[edge] (V3) -- node[edgeLabel] {$43$} (V27);
\draw[edge] (V19) -- node[edgeLabel] {$46$} (V27);
\draw[edge] (V4) -- node[edgeLabel] {$44$} (V28);
\draw[edge] (V20) -- node[edgeLabel] {$47$} (V28);
\draw[edge] (V5) -- node[edgeLabel] {$45$} (V29);
\draw[edge] (V21) -- node[edgeLabel] {$48$} (V29);
\draw[edge] (V6) -- node[edgeLabel] {$46$} (V30);
\draw[edge] (V22) -- node[edgeLabel] {$41$} (V30);
\draw[edge] (V7) -- node[edgeLabel] {$47$} (V31);
\draw[edge] (V23) -- node[edgeLabel] {$42$} (V31);
\draw[edge] (V8) -- node[edgeLabel] {$48$} (V32);
\draw[edge] (V24) -- node[edgeLabel] {$43$} (V32);

% Draw the vertices
\vertexLabelR{V0}{left}{$1$}
\vertexLabelR{V1}{left}{$2$}
\vertexLabelR{V2}{left}{$3$}
\vertexLabelR{V3}{left}{$4$}
\vertexLabelR{V4}{left}{$5$}
\vertexLabelR{V5}{left}{$6$}
\vertexLabelR{V6}{left}{$7$}
\vertexLabelR{V7}{left}{$8$}
\vertexLabelR{V8}{left}{$9$}

\vertexLabelR{V9}{left}{$10$}
\vertexLabelR{V10}{left}{$11$}
\vertexLabelR{V11}{left}{$10$}
\vertexLabelR{V12}{left}{$11$}
\vertexLabelR{V13}{left}{$10$}
\vertexLabelR{V14}{left}{$11$}
\vertexLabelR{V15}{left}{$10$}
\vertexLabelR{V16}{left}{$11$}

\vertexLabelR{V17}{left}{$5$}
\vertexLabelR{V18}{left}{$6$}
\vertexLabelR{V19}{left}{$7$}
\vertexLabelR{V20}{left}{$8$}
\vertexLabelR{V21}{left}{$9$}
\vertexLabelR{V22}{left}{$2$}
\vertexLabelR{V23}{left}{$3$}
\vertexLabelR{V24}{left}{$4$}

\vertexLabelR{V25}{left}{$12$}
\vertexLabelR{V26}{left}{$12$}
\vertexLabelR{V27}{left}{$12$}
\vertexLabelR{V28}{left}{$12$}
\vertexLabelR{V29}{left}{$12$}
\vertexLabelR{V30}{left}{$12$}
\vertexLabelR{V31}{left}{$12$}
\vertexLabelR{V32}{left}{$12$}

\end{tikzpicture}}}
    \caption{The $16$-vertex regular triangulation
of the torus with Schl\"afli symbol $\{3, 6\}_8$ (left) and the dual Dyck map with Schl\"afli symbol $\{3, 8\}_6$ (right)}
    \label{fig:groups}
\end{figure}
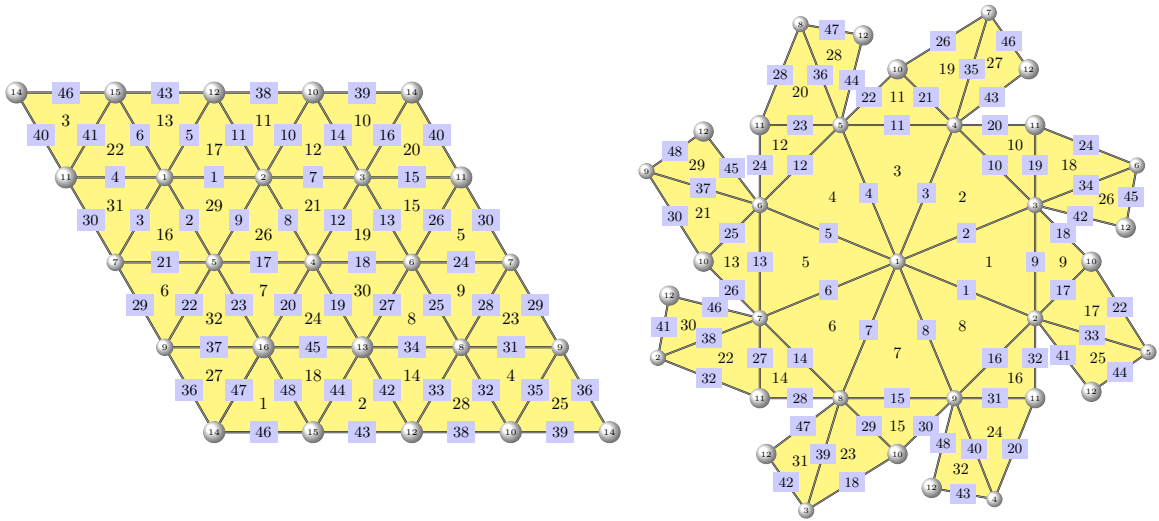

\end{remark}

\subsection{Future directions}
\label{sec:summary}

We conclude this paper with a small number of questions and remarks about possible generalisations of our work. 
\begin{enumerate}
    \item Can we exploit larger databases of node-transitive cubic graphs to construct a larger census of face-transitive surfaces?

    Theoretically, this is likely computationally feasible. However, to the best of our knowledge, there is no database containing a complete list of node-transitive cubic graphs with more than $1280$ nodes. As a partial solution, we can, for instance, use \texttt{GraphSym}'s list of arc-transitive cubic graphs with at most $10\,000$ nodes, to compute larger examples of face-transitive surfaces.  
    
    \item Can the ideas developed in this paper be adapted to construct classifications of other classes of highly-symmetric simplicial surfaces?

    This question has recently been answered in \cite{Akpanya2025} for the class of edge-transitive surfaces, i.e.\ simplicial surfaces whose automorphism groups act transitively on their edges. There are exactly four types of edge-transitive surfaces that further split into an overall total of five sub-types. Similarly to the study presented in this paper, the classification of edge-transitive surfaces is achieved by introducing an invariant to analyse the structure of edge-transitive surfaces on a theoretical level. The classification in \cite{Akpanya2025} is complete up to $5000$ faces. 

    \item For each vertex-face type $(v,s)$ in \Cref{theorem:invariants}, can we find infinite families of face-transitive surfaces whose vertex-face types are all equal to $(v,s)?$

    For some vertex-face types, this question can be answered by giving well-known examples of face-transitive surfaces. For instance, we know that the suspension over the $n$-cycle, $n\geq 3$ (also known as double-pyramid), gives rise to face-transitive spheres with vertex-face type $(2,2).$ In general, this proves to be an intriguing question. Specifically, for vertex-face types, where we only have a small number of examples. If the answer to this question is ``yes'', it can be further checked whether there exist infinite families of face-transitive surfaces for each of the thirteen subtypes of automorphism groups. Such an investigation is a first step towards a complete classification of face-transitive surfaces.
\end{enumerate}

\newpage
\bibliographystyle{plain}
%\nocite{*}
%\bibliography{main}

\begin{thebibliography}{10}

\bibitem{Akpanya2025}
Reymond Akpanya.
\newblock {\em Constructing symmetric simplicial surfaces}.
\newblock PhD thesis, RWTH Aachen University, 2025.
\newblock PhD thesis, submitted.

\bibitem{simplicialsurfacegap}
Reymond Akpanya, Markus Baumeister, Tom Görtzen, Alice Niemeyer, and Meike
  Weiß.
\newblock Simplicicalsurfaces, version 0.6.
\newblock \url{https://github.com/gap-packages/SimplicialSurfaces}, 2023.

\bibitem{automorphism}
Reymond {Akpanya} and Tom {Goertzen}.
\newblock {Surfaces with given Automorphism Group}.
\newblock {\em arXiv e-prints}, page arXiv:2307.12681, July 2023.

\bibitem{facetransitivesurfaces}
Reymond Akpanya and Spreer Jonathan.
\newblock facetransitivesurfaces, 2024.
\newblock \url{https://github.com/jspreer/faceTransitiveSurfaces}.

\bibitem{Altmann24}
Eduardo~G. Altmann and Jonathan Spreer.
\newblock Sampling triangulations of manifolds using monte carlo methods.
\newblock arXiv:2310.07372, 29 pages, 6 figures, 2024.
\newblock To appear in Exp. Math.

\bibitem{DeBeule2024aa}
Jan~De Beule, Julius Jonu{\v s}as, James~D. Mitchell, Michael Torpey, Maria
  Tsalakou, and Wilf~A. Wilson.
\newblock Digraphs - {GAP} package, version 1.9.0, Sep 2024.

\bibitem{magma}
Wieb Bosma, John Cannon, and Catherine Playoust.
\newblock The {M}agma algebra system. {I}. {T}he user language.
\newblock {\em J. Symbolic Comput.}, 24(3-4):235--265, 1997.
\newblock Computational algebra and number theory (London, 1993).

\bibitem{onetriangle}
K.-H Brakhage, A.~Niemeyer, W.~Plesken, and A.~Strzelczyk.
\newblock Simplicial surfaces controlled by one triangle.
\newblock {\em Journal for Geometry and Graphics}, 21:141--152, 01 2017.

\bibitem{icosahedron}
Karl-Heinz Brakhage, Alice Niemeyer, Wilhelm Plesken, Daniel Robertz, and
  Ansgar Strzelczyk.
\newblock The icosahedra of edge length 1.
\newblock {\em Journal of Algebra}, 545, 05 2019.

\bibitem{kuehnel}
Ulrich Brehm and Wolfgang Kühnel.
\newblock Equivelar maps on the torus.
\newblock {\em European Journal of Combinatorics}, 29(8):1843--1861, 2008.
\newblock {In honor of Ludwig Danzer's 80th birthday}.

\bibitem{burton2018pachner}
Benjamin~A. Burton, Basudeb Datta, and Jonathan Spreer.
\newblock The pachner graph of 2-spheres, 2018.

\bibitem{CONDER2001224}
Marston Conder and Peter Dobcsányi.
\newblock Determination of all regular maps of small genus.
\newblock {\em Journal of Combinatorial Theory, Series B}, 81(2):224--242,
  2001.

\bibitem{coolsaet}
Kris Coolsaet and Stan Schein.
\newblock Some new symmetric equilateral embeddings of platonic and archimedean
  polyhedra.
\newblock {\em Symmetry}, 10(9), 2018.

\bibitem{mintorus}
\'Akos Cs\'asz\'ar.
\newblock A polyhedron without diagonals.
\newblock {\em Acta Univ. Szeged. Sect. Sci. Math.}, 13:140--142, 1949.

\bibitem{Datta17SemiEquivelar}
Basudeb Datta and Dipendu Maity.
\newblock Semi-equivelar and vertex-transitive maps on the torus.
\newblock {\em Beitr\"age Algebra Geom.}, 58:617--634, 2017.

\bibitem{DATTA20183296}
Basudeb Datta and Dipendu Maity.
\newblock Semi-equivelar maps on the torus and the klein bottle are
  archimedean.
\newblock {\em Discrete Mathematics}, 341(12):3296--3309, 2018.

\bibitem{DATTA2022112652}
Basudeb Datta and Dipendu Maity.
\newblock Platonic solids, archimedean solids and semi-equivelar maps on the
  sphere.
\newblock {\em Discrete Mathematics}, 345(1):112652, 2022.

\bibitem{simpcomp}
Felix Effenberger and Jonathan Spreer.
\newblock simpcomp: a gap toolbox for simplicial complexes.
\newblock {\em ACM Commun. Comput. Algebra}, 44(3/4):186–189, jan 2011.

\bibitem{isosceles}
David Eppstein.
\newblock On polyhedral realization with isosceles triangles.
\newblock {\em Graphs Combin.}, 37(4):1247--1269, 2021.

\bibitem{GAP4}
The GAP~Group.
\newblock {\em {GAP -- Groups, Algorithms \& Programming, Vers.\ 4.12.2}}.

\bibitem{grunbaum1969conjecture}
B~Gr{\"u}nbaum.
\newblock Conjecture 6, in “{R}ecent {P}rogress in {C}ombinatorics”, {E}d.
  {WT} {T}utte, 1969.

\bibitem{Hoffman1991}
F.~Hoffman, S.C. Locke, and A.D. Meyerowitz.
\newblock A note on cycle double covers in cayley graphs.
\newblock {\em Mathematica Pannonica}, 2(1):63--66, 1991.

\bibitem{hurwitz_ueber_1892}
A.~Hurwitz.
\newblock Ueber algebraische {Gebilde} mit eindeutigen {Transformationen} in
  sich.
\newblock {\em Mathematische Annalen}, 41(3):403--442, September 1892.

\bibitem{novik}
Ivan Izmestiev, Steven Klee, and Isabella Novik.
\newblock Simplicial moves on balanced complexes.
\newblock {\em Adv. Math.}, 320:82--114, 2017.

\bibitem{JAEGER19851}
Francois Jaeger.
\newblock A survey of the cycle double cover conjecture.
\newblock In B.R. Alspach and C.D. Godsil, editors, {\em Annals of Discrete
  Mathematics (27): Cycles in Graphs}, volume 115 of {\em North-Holland
  Mathematics Studies}, pages 1--12. North-Holland, 1985.

\bibitem{LutzDiss}
Frank~H. Lutz.
\newblock {\em Triangulated Manifolds with Few Vertices and Vertex-Transitive
  Group Actions}.
\newblock Shaker Verlag, Aachen, 1999.

\bibitem{MAITI2020111911}
Arun Maiti.
\newblock Quasi-vertex-transitive maps on the plane.
\newblock {\em Discrete Mathematics}, 343(7):111911, 2020.

\bibitem{Evencycles}
Klas Markström.
\newblock Even cycle decompositions of 4-regular graphs and line graphs.
\newblock {\em Discrete Mathematics}, 312(17):2676--2681, 2012.
\newblock Proceedings of the 8th French Combinatorial Conference.

\bibitem{NEGAMI1994225}
Seiya Negami.
\newblock Diagonal flips in triangulations of surfaces.
\newblock {\em Discrete Mathematics}, 135(1):225--232, 1994.

\bibitem{simplicialsurfacesbook}
Alice~C. Niemeyer, Wilhelm Plesken, and Daniel Robertz.
\newblock {\em Simplicial Surfaces of Congruent Triangles}.
\newblock In preparation, 2024.

\bibitem{cubicvertextransitive}
Primo\v{z} Poto\v{c}nik, Pablo Spiga, and Gabriel Verret.
\newblock Cubic vertex-transitive graphs on up to 1280 vertices.
\newblock {\em Journal of Symbolic Computation}, 50:465--477, 2013.

\bibitem{mapcoloring}
Gerhard Ringel and J.~W.~T. Youngs.
\newblock Solution of the {H}eawood map-coloring problem.
\newblock {\em Proc. Nat. Acad. Sci. U.S.A.}, 60:438--445, 1968.

\bibitem{seymour}
Paul~D Seymour.
\newblock Sums of circuits.
\newblock {\em Graph theory and related topics}, 1:341--355, 1979.

\bibitem{szekeres}
G.~Szekeres.
\newblock Polyhedral decompositions of cubic graphs.
\newblock {\em Bulletin of the Australian Mathematical Society},
  8(3):367–387, 1973.

\bibitem{NPcomplete}
Carsten Thomassen.
\newblock The genus problem for cubic graphs.
\newblock {\em Journal of Combinatorial Theory, Series B}, 1997.

\bibitem{weddslist}
N.~S. Wedd.
\newblock Weddslist, 2025.
\newblock Accessed: 2025-01-22.

\end{thebibliography}

\appendix
\newpage
\section{Examples}
\label{section:examples}

This section contains the vertices of faces of the face-transitive surfaces mentioned in \Cref{section:construction}.
\subsection{Set of faces of the surface \texorpdfstring{$X^{(3,1)}$}{X{(1,1)}}}
\label{example31}

\begin{align*}
  \{&\{ 1, 2, 3 \}, \{ 4, 5, 6 \}, \{ 2, 5, 7 \}, \{ 3, 6, 8 \}, \{ 2, 5, 9 \}, 
 \{ 5, 7, 10 \}, \{ 3, 4, 6 \}, \{ 5, 10, 11 \},\\
 &\{ 6, 8, 12 \}, \{ 8, 10, 12 \}, \{ 6, 9, 12 \}, \{ 10, 11, 12 \}, \{ 3, 4, 13 \}, \{ 2, 3, 7 \},
 \{ 1, 3, 13 \}, \{ 3, 8, 10 \},\\
 &\{ 3, 7, 10 \}, \{ 11, 12, 13 \}, 
 \{ 5, 11, 13 \}, \{ 4, 5, 13 \}, \{ 5, 6, 9 \},\{ 2, 9, 12 \}, \{ 1, 12, 13 \}, \{ 1, 2, 12 \} \}
\end{align*}
\subsection{Set of faces of the surface \texorpdfstring{$X^{(2,2)}$}{X(2,2)}}\label{example22}
\begin{align*}
  \{\{ 1, 4, 5 \}, \{ 1, 4, 6 \}, \{ 1, 5, 7 \}, \{ 1, 6, 7 \}, \{ 2, 3, 5 \}, 
 \{ 2, 3, 6 \}, \{ 2, 4, 5 \}, \{ 2, 4, 6 \}, \{ 3, 5, 7 \}, \{ 3, 6, 7 \}\}
\end{align*}
\subsection{Set of faces of the surface \texorpdfstring{$X^{(1,6)}$}{X(1,6)}}\label{example16}
\begin{align*}
\{\{ 1, 2, 5 \},\{ 1, 2, 6 \},\{ 1, 4, 6 \},\{ 1, 3, 4 \},\{ 1, 3, 5 \},\{ 2, 4, 5 \},\{ 2, 3, 6 \},\{ 4, 5, 6 \},\{ 2, 3, 4 \},\{ 3, 5, 6 \}\}
\end{align*}

\subsection{Set of faces of the surface \texorpdfstring{$X^{(2,1)}$}{X(2,1)}}
\label{example211}

\begin{align*}
\{&\{4,5,6\}, \{2,3,7\}, \{1,2,7\}, \{1,2,8\}, \{1,3,6\}, \{2,4,9\}, \{3,5,9\}, \{3,4,7\}, \{1,5,10\}, \{3,4,6\}, \\
&\{2,4,10\}, \{2,5,8\}, \{4,5,9\}, \{1,3,8\}, \{1,4,10\}, \{2,5,10\}, \{1,4,7\}, \{2,3,9\}, \{1,5,6\}, \{3,5,8\}\}
  \end{align*}
  
\subsection{Set of faces of the surface \texorpdfstring{$Y^{(2,1)}$}{Y(2,1)}}\label{example212}
\begin{align*}
\{&\{ 13, 14, 15 \}, \{ 11, 12, 16 \}, \{ 9, 10, 17 \}, \{ 11, 12, 18 \}, \{ 8, 14, 19 \}, \{ 9, 10, 20 \}, \{ 13, 14, 21 \},
\{ 6, 7, 17 \},\\ &
\{ 5, 9, 20 \}, \{ 4, 11, 16 \}, \{ 2, 3, 19 \}, \{ 6, 8, 19 \}, \{ 2, 3, 21 \}, \{ 1, 2, 19 \},
\{ 4, 13, 22 \}, \{ 4, 13, 15 \}, \{ 1, 2, 18 \}, \\ &
\{ 2, 11, 17 \}, \{ 5, 8, 18 \}, \{ 6, 9, 16 \}, \{ 2, 11, 18 \},
\{ 3, 4, 20 \}, \{ 5, 12, 21 \}, \{ 10, 13, 17 \}, \{ 9, 14, 16 \},
\{ 8, 10, 18 \}, \\ &\{ 3, 7, 19 \}, \{ 9, 14, 21 \},
\{ 3, 7, 20 \}, \{ 10, 13, 22 \}, \{ 5, 7, 20 \}, \{ 6, 12, 22 \}, \{ 2, 13, 21 \}, \{ 3, 12, 21 \}, \{ 8, 14, 15 \},\\ &
\{ 6, 8, 22 \}, \{ 1, 4, 16 \}, \{ 1, 10, 20 \}, \{ 5, 12, 18 \}, \{ 1, 10, 18 \}, \{ 7, 11, 15 \}, \{ 4, 11, 15 \},
\{ 6, 9, 17 \}, \{ 5, 9, 21 \},\\ &
\{ 3, 12, 22 \}, \{ 3, 4, 22 \}, \{ 1, 14, 19 \}, \{ 1, 14, 16 \}, \{ 6, 7, 19 \},
\{ 1, 4, 20 \}, \{ 8, 10, 22 \}, \{ 5, 8, 15 \}, \{ 7, 11, 17 \},
\\ &\{ 2, 13, 17 \}, \{ 5, 7, 15 \}, \{ 6, 12, 16 \}
\end{align*}

\subsection{Set of faces of the surface \texorpdfstring{$Z^{(2,1)}$}{Z(2,1)}}
\label{example213}
\begin{align*}
\{
&\{ 11, 12, 13 \}, \{ 9, 10, 13 \}, \{ 7, 8, 14 \}, \{ 6, 11, 13 \}, \{ 5, 12, 15 \}, \{ 4, 9, 15 \}, \{ 9, 11, 15 \}, 
\{ 3, 10, 16 \}, \\
&
\{ 9, 12, 13 \}, \{ 2, 8, 17 \}, \{ 3, 5, 15 \}, \{ 1, 4, 14 \}, \{ 3, 10, 17 \}, \{ 3, 12, 17 \}, 
\{ 4, 10, 15 \}, \{ 9, 12, 18 \},\\& \{ 5, 8, 18 \}, \{ 1, 2, 13 \}, \{ 3, 7, 15 \}, \{ 3, 5, 18 \}, \{ 9, 10, 16 \}, 
\{ 1, 6, 18 \}, \{ 7, 10, 17 \}, \{ 2, 4, 13 \}, \{ 5, 6, 17 \},\\& \{ 11, 12, 15 \}, \{ 6, 11, 14 \}, \{ 2, 7, 14 \}, 
\{ 1, 4, 16 \}, \{ 5, 12, 17 \}, \{ 6, 8, 17 \}, \{ 5, 8, 14 \}, \{ 3, 7, 16 \},\\& \{ 2, 8, 16 \}, \{ 2, 4, 14 \}, 
\{ 4, 10, 13 \}, \{ 1, 6, 13 \}, \{ 1, 11, 14 \}, \{ 9, 11, 18 \}, \{ 2, 7, 17 \}, \{ 1, 2, 16 \},\\& \{ 6, 8, 18 \}, 
\{ 3, 12, 18 \}, \{ 1, 11, 18 \}, \{ 7, 8, 16 \}, \{ 4, 9, 16 \}, \{ 5, 6, 14 \}, \{ 7, 10, 15 \}
\}
\end{align*}
\subsection{Set of faces of the surface \texorpdfstring{$X^{(1,2)}$}{X(1,2)}}\label{example12}

\begin{align*}
    \{&
\{ 6, 7, 8 \}, \{ 4, 5, 8 \}, \{ 4, 5, 7 \}, \{ 2, 3, 7 \}, \{ 2, 4, 7 \}, \{ 1, 2, 4 \}, \{ 3, 4, 8 \}, \{ 5, 6, 7 \}, \\&
\{ 1, 3, 4 \}, \{ 1, 2, 5 \}, \{ 1, 6, 8 \}, \{ 1, 5, 8 \}, \{ 2, 5, 6 \}, \{ 1, 3, 6 \}, \{ 3, 7, 8 \}, \{ 2, 3, 6 \}
\}
\end{align*}

\subsection{Set of faces of the surface \texorpdfstring{$X^{(1,3)}$}{X(1,3)}}\label{example131}

\begin{align*}
    \{ &\{ 1, 2, 3 \}, \{ 1, 3, 4 \}, \{ 2, 4, 6 \}, \{ 2, 3, 5 \}, \{ 3, 5, 6 \}, \{ 3, 4, 7 \}, \{ 3, 6, 7 \},\\
    & \{ 2, 4, 7 \},\{ 2, 5, 7 \}, \{ 1, 2, 6 \}, 
  \{ 1, 5, 7 \}, \{ 1, 4, 5 \}, \{ 4, 5, 6 \}, \{ 1, 6, 7 \} \}
\end{align*}
\subsection{Set of faces of the surface \texorpdfstring{$Y^{(1,3)}$}{Y(1,3)}}\label{example132}
\begin{align*}
    \{& \{ 34, 35, 36 \}, \{ 31, 32, 33 \}, \{ 28, 29, 30 \}, \{ 25, 26, 27 \}, \{ 24, 32, 35 \}, \{ 32, 34, 35 \}, \{ 23, 24, 32 \}, \{ 22, 29, 33 \},\\& \{ 20, 21, 27 \}, \{ 18, 19, 20 \}, 
  \{ 17, 28, 29 \}, \{ 15, 16, 22 \}, \{ 14, 23, 26 \}, \{ 13, 15, 22 \}, \{ 11, 12, 36 \}, \{ 10, 21, 33 \},\\& \{ 8, 9, 19 \}, \{ 21, 35, 36 \}, \{ 9, 21, 23 \}, \{ 19, 29, 30 \}, 
  \{ 8, 24, 35 \}, \{ 7, 11, 29 \}, \{ 21, 29, 35 \}, \{ 6, 23, 28 \},\\& \{ 15, 19, 36 \}, \{ 20, 22, 34 \}, \{ 4, 5, 24 \}, \{ 18, 19, 31 \}, \{ 3, 10, 21 \}, \{ 5, 28, 30 \}, 
  \{ 8, 9, 13 \}, \{ 21, 29, 33 \},\\& \{ 9, 12, 17 \}, \{ 16, 28, 31 \}, \{ 7, 25, 29 \}, \{ 8, 16, 28 \}, \{ 2, 3, 17 \}, \{ 2, 5, 15 \}, \{ 6, 12, 25 \}, \{ 17, 27, 28 \},\\& \{ 11, 13, 18 \}, 
  \{ 1, 3, 16 \}, \{ 3, 10, 24 \}, \{ 6, 10, 28 \}, \{ 6, 25, 35 \}, \{ 14, 17, 18 \}, \{ 19, 31, 36 \}, \{ 23, 31, 32 \},\\& \{ 1, 11, 29 \}, \{ 3, 16, 31 \}, \{ 12, 16, 20 \}, 
  \{ 4, 7, 31 \}, \{ 10, 15, 19 \}, \{ 1, 32, 34 \}, \{ 3, 25, 26 \}, \{ 1, 19, 29 \},\\& \{ 5, 9, 25 \}, \{ 3, 17, 24 \}, \{ 14, 26, 30 \}, \{ 18, 31, 33 \}, \{ 13, 30, 32 \}, \{ 4, 14, 30 \}, 
  \{ 6, 10, 12 \}, \{ 3, 25, 31 \},\\& \{ 8, 19, 24 \}, \{ 17, 22, 24 \}, \{ 3, 26, 30 \}, \{ 7, 25, 31 \}, \{ 10, 19, 24 \}, \{ 8, 16, 35 \}, \{ 17, 22, 29 \}, \{ 16, 20, 22 \}, \\&
  \{ 9, 12, 25 \}, \{ 18, 20, 27 \}, \{ 15, 27, 32 \}, \{ 10, 14, 15 \}, \{ 4, 5, 30 \}, \{ 3, 13, 30 \}, \{ 7, 11, 13 \}, \{ 11, 26, 36 \}, \\&\{ 6, 13, 22 \}, \{ 11, 23, 26 \}, \{ 12, 32, 33 \}, 
  \{ 23, 28, 31 \}, \{ 5, 9, 34 \}, \{ 4, 21, 36 \}, \{ 11, 12, 16 \}, \{ 2, 7, 8 \},\\& \{ 1, 2, 3 \}, \{ 4, 21, 27 \}, \{ 22, 26, 34 \}, \{ 20, 21, 23 \}, \{ 2, 6, 23 \}, 
  \{ 26, 34, 36 \}, \{ 9, 14, 17 \}, \{ 3, 13, 21 \}, \\&\{ 2, 17, 36 \}, \{ 8, 26, 33 \}, \{ 12, 30, 32 \}, \{ 12, 20, 30 \}, \{ 7, 10, 14 \}, \{ 15, 25, 27 \}, \{ 5, 28, 34 \}, \{ 14, 15, 16 \}, \\&
  \{ 10, 28, 34 \}, \{ 2, 15, 36 \}, \{ 17, 18, 27 \}, \{ 14, 16, 35 \}, \{ 6, 13, 18 \}, \{ 13, 15, 32 \}, \{ 19, 20, 30 \}, \{ 22, 26, 33 \},\\& \{ 4, 22, 24 \}, \{ 7, 8, 13 \}, 
  \{ 1, 9, 34 \}, \{ 1, 2, 6 \}, \{ 2, 5, 33 \}, \{ 1, 9, 19 \}, \{ 7, 20, 34 \}, \{ 8, 27, 28 \}, \{ 5, 18, 33 \},\\& \{ 1, 11, 16 \}, \{ 8, 26, 27 \}, \{ 7, 10, 34 \}, \{ 2, 7, 20 \}, 
  \{ 14, 18, 35 \}, \{ 2, 20, 23 \}, \{ 9, 13, 21 \}, \{ 25, 29, 35 \},\\& \{ 1, 4, 27 \}, \{ 9, 14, 23 \}, \{ 12, 17, 36 \}, \{ 5, 11, 18 \}, \{ 4, 31, 36 \}, \{ 1, 4, 6 \}, 
  \{ 11, 23, 24 \}, \{ 2, 8, 33 \}, \{ 1, 27, 32 \},\\& \{ 5, 15, 25 \}, \{ 5, 11, 24 \}, \{ 4, 7, 14 \}, \{ 6, 18, 35 \}, \{ 4, 6, 22 \}, \{ 10, 12, 33 \} \}
\end{align*}

\subsection{Set of faces of the surface \texorpdfstring{$X^{(1,1)}$}{X(1,1)}}\label{example11}
\begin{align*}
\{&\{ 7, 8, 9 \},\{ 6, 7, 8 \},\{ 3, 4, 5 \},\{ 4, 8, 9 \},\{ 2, 3, 8 \},\{ 3, 5, 7 \},\{ 1, 2, 3 \},\{ 3, 4, 8 \},\{ 1, 2, 5 \},\\
&\{ 4, 6, 9 \},\{ 2, 6, 8 \},\{ 1, 6, 7 \},\{ 1, 3, 7 \},\{ 2, 4, 6 \},\{ 2, 4, 5 \},\{ 1, 5, 9 \},\{ 1, 6, 9 \},\{ 5, 7, 9 \}\}.
\end{align*}

\subsection{Set of faces of the surface \texorpdfstring{$Y^{(1,1)}$}{Y(1,1)}}
\label{example112}
\begin{align*}
    \{& \{ 16, 17, 18 \}, \{ 13, 14, 15 \}, \{ 11, 12, 14 \}, \{ 8, 9, 10 \}, \{ 6, 7, 9 \}, \{ 5, 8, 17 \}, \{ 4, 12, 13 \}, \{ 3, 15, 16 \}, \\&\{ 5, 10, 18 \}, \{ 2, 7, 8 \}, 
  \{ 2, 9, 17 \}, \{ 7, 10, 11 \}, \{ 3, 4, 13 \}, \{ 1, 11, 16 \}, \{ 6, 12, 17 \}, \{ 13, 16, 17 \},\\& \{ 2, 3, 14 \}, \{ 6, 13, 18 \}, \{ 2, 6, 11 \}, \{ 1, 8, 9 \}, 
  \{ 3, 5, 8 \}, \{ 12, 13, 18 \}, \{ 14, 16, 18 \}, \{ 1, 5, 7 \},\\& \{ 4, 9, 17 \}, \{ 2, 11, 15 \}, \{ 6, 12, 16 \}, \{ 1, 4, 11 \}, \{ 6, 10, 15 \}, \{ 3, 13, 16 \}, 
  \{ 8, 12, 17 \}, \{ 1, 14, 18 \},\\& \{ 5, 11, 14 \}, \{ 2, 7, 13 \}, \{ 2, 10, 16 \}, \{ 4, 8, 12 \}, \{ 4, 5, 14 \}, \{ 7, 8, 18 \}, \{ 1, 5, 15 \}, \{ 4, 11, 15 \}, \\&
  \{ 1, 12, 16 \}, \{ 5, 6, 7 \}, \{ 3, 9, 12 \}, \{ 8, 10, 15 \}, \{ 2, 3, 8 \}, \{ 1, 7, 17 \}, \{ 5, 6, 11 \}, \{ 9, 13, 14 \},\\& \{ 2, 13, 15 \}, \{ 3, 4, 10 \}, 
  \{ 2, 9, 10 \}, \{ 11, 12, 18 \}, \{ 5, 17, 18 \}, \{ 1, 14, 15 \}, \{ 3, 12, 14 \}, \{ 6, 10, 18 \},\\& \{ 1, 9, 12 \}, \{ 6, 15, 16 \}, \{ 2, 6, 17 \}, \{ 3, 7, 9 \}, 
  \{ 4, 5, 10 \}, \{ 3, 5, 15 \}, \{ 1, 8, 18 \}, \{ 2, 14, 16 \}, \\&\{ 7, 13, 17 \}, \{ 6, 9, 13 \}, \{ 3, 7, 10 \}, \{ 4, 8, 15 \}, \{ 1, 4, 17 \}, \{ 10, 11, 16 \}, 
  \{ 4, 9, 14 \}, \{ 7, 11, 18 \} \}
\end{align*}
\subsection{Set for faces of the surface \texorpdfstring{$\overline{X}^{(1,1)}$}{X(1,1)} }\label{example113}
\begin{align*}
\{& \{ 16, 17, 18 \}, \{ 14, 15, 17 \}, \{ 13, 14, 17 \}, \{ 11, 12, 13 \}, \{ 10, 16, 17 \}, \{ 9, 12, 18 \}, \{ 8, 10, 18 \},\\& \{ 6, 7, 13 \}, \{ 5, 10, 18 \}, \{ 4, 6, 14 \}, 
  \{ 3, 11, 15 \}, \{ 4, 6, 13 \}, \{ 2, 11, 13 \}, \{ 4, 9, 13 \}, \{ 1, 17, 18 \},\\& \{ 1, 3, 15 \}, \{ 9, 15, 17 \}, \{ 1, 6, 16 \}, \{ 9, 16, 18 \}, \{ 1, 2, 5 \}, 
  \{ 8, 10, 16 \}, \{ 1, 8, 11 \}, \{ 2, 9, 16 \},\\& \{ 4, 8, 16 \}, \{ 2, 11, 18 \}, \{ 1, 6, 8 \}, \{ 2, 10, 12 \}, \{ 6, 7, 12 \}, \{ 3, 5, 11 \}, \{ 3, 12, 14 \}, 
  \{ 6, 11, 12 \},\\& \{ 1, 2, 4 \}, \{ 12, 13, 18 \}, \{ 7, 11, 15 \}, \{ 2, 7, 10 \}, \{ 2, 7, 15 \}, \{ 2, 5, 18 \}, \{ 3, 10, 14 \}, \{ 8, 13, 17 \},\\& \{ 4, 7, 8 \}, 
  \{ 4, 5, 15 \}, \{ 8, 11, 17 \}, \{ 7, 8, 13 \}, \{ 5, 8, 14 \}, \{ 2, 13, 15 \}, \{ 3, 6, 16 \}, \{ 10, 14, 15 \},\\& \{ 3, 6, 9 \}, \{ 13, 14, 18 \}, \{ 6, 9, 14 \}, 
  \{ 9, 12, 14 \}, \{ 3, 5, 12 \}, \{ 5, 10, 12 \}, \{ 3, 7, 17 \}, \{ 7, 11, 17 \},\\& \{ 1, 11, 18 \}, \{ 5, 6, 11 \}, \{ 7, 12, 16 \}, \{ 4, 5, 14 \}, \{ 2, 4, 9 \}, 
  \{ 3, 10, 17 \}, \{ 2, 12, 16 \}, \{ 9, 13, 15 \},\\& \{ 4, 10, 15 \}, \{ 1, 9, 17 \}, \{ 5, 6, 8 \}, \{ 8, 14, 18 \}, \{ 1, 3, 9 \}, \{ 3, 7, 16 \}, \{ 1, 4, 16 \}, 
  \{ 4, 7, 10 \}, \{ 1, 5, 15 \} \}
  \end{align*}

\subsection{Set of faces of the surface \texorpdfstring{$\overline{Y}^{(1,1)}$}{Y(1,1)}}
\label{example114}
\begin{align*}
\{ &\{ 26, 27, 28 \}, \{ 23, 24, 25 \}, \{ 20, 21, 22 \}, \{ 18, 19, 22 \}, \{ 17, 19, 27 \}, \{ 15, 16, 18 \}, \{ 14, 19, 25 \}, \{ 13, 14, 17 \},\\& \{ 12, 13, 18 \}, \{ 14, 22, 28 \}, 
  \{ 9, 10, 11 \}, \{ 6, 7, 8 \}, \{ 5, 10, 19 \}, \{ 8, 9, 17 \}, \{ 4, 9, 26 \}, \{ 5, 12, 19 \}, \\&\{ 3, 10, 24 \}, \{ 2, 15, 18 \}, \{ 23, 24, 28 \}, \{ 23, 25, 27 \}, \{ 3, 9, 25 \}, \{ 17, 26, 27 \}, \{ 9, 10, 23 \}, 
  \{ 5, 16, 22 \},\\& \{ 5, 6, 7 \}, \{ 5, 12, 16 \}, \{ 12, 19, 21 \}, \{ 16, 22, 23 \}, \{ 12, 20, 21 \}, \{ 1, 10, 23 \}, \{ 2, 6, 21 \}, \{ 1, 19, 25 \},\\& \{ 5, 13, 22 \}, \{ 12, 14, 28 \}, 
  \{ 2, 17, 18 \}, \{ 6, 8, 16 \}, \{ 16, 23, 27 \}, \{ 17, 18, 19 \}, \{ 13, 21, 22 \}, \{ 3, 16, 24 \},\\& \{ 11, 17, 24 \}, \{ 6, 20, 23 \}, \{ 8, 16, 27 \}, \{ 2, 5, 6 \}, \{ 4, 11, 24 \}, \{ 15, 26, 28 \}, 
  \{ 10, 13, 21 \}, \{ 12, 16, 26 \},\\& \{ 6, 23, 28 \}, \{ 11, 18, 25 \}, \{ 4, 5, 13 \}, \{ 11, 14, 17 \}, \{ 10, 13, 24 \}, \{ 6, 11, 14 \}, \{ 1, 4, 26 \}, \{ 2, 25, 27 \}, \\& \{ 4, 19, 21 \}, 
  \{ 1, 2, 4 \}, \{ 3, 12, 20 \}, \{ 3, 9, 20 \}, \{ 1, 20, 23 \}, \{ 2, 4, 21 \}, \{ 10, 11, 21 \}, \{ 4, 19, 27 \}, \\&\{ 2, 14, 25 \}, \{ 1, 10, 19 \}, \{ 7, 8, 9 \}, \{ 7, 12, 26 \}, \{ 1, 8, 17 \}, 
  \{ 7, 22, 23 \}, \{ 1, 3, 7 \}, \{ 16, 24, 26 \}, \\&\{ 6, 11, 18 \}, \{ 4, 8, 24 \}, \{ 12, 13, 14 \}, \{ 11, 15, 21 \}, \{ 3, 10, 27 \}, \{ 8, 20, 22 \}, \{ 4, 8, 27 \}, \{ 6, 16, 18 \}, \\&\{ 13, 24, 25 \}, 
  \{ 9, 11, 25 \}, \{ 9, 20, 26 \}, \{ 3, 27, 28 \}, \{ 7, 18, 22 \}, \{ 15, 20, 26 \}, \{ 13, 18, 25 \}, \{ 15, 21, 28 \}, \\&\{ 1, 2, 17 \}, \{ 6, 21, 28 \}, \{ 4, 5, 11 \}, \{ 14, 19, 22 \}, 
  \{ 7, 12, 18 \}, \{ 2, 5, 10 \}, \{ 2, 10, 27 \}, \{ 6, 14, 20 \}, \\&\{ 7, 9, 23 \}, \{ 1, 8, 20 \}, \{ 1, 7, 26 \}, \{ 9, 13, 17 \}, \{ 17, 24, 26 \}, \{ 5, 11, 15 \}, \{ 8, 24, 28 \}, \{ 4, 9, 13 \}, \\&
  \{ 3, 7, 15 \}, \{ 14, 15, 20 \}, \{ 5, 7, 15 \}, \{ 2, 14, 15 \}, \{ 8, 22, 28 \}, \{ 3, 15, 16 \}, \{ 1, 3, 25 \}, \{ 3, 12, 28 \} \}
\end{align*}

\end{document}